\documentclass[reqno,11pt]{amsart}

\usepackage[usenames,dvipsnames]{xcolor}
\usepackage{amssymb,color,hyperref,mathrsfs,stmaryrd,tikz-cd}
\usepackage{enumitem}
\usepackage{amsmath}
\usepackage{mathtools}

\usepackage[curve,matrix,arrow]{xy}

\numberwithin{equation}{section}
\newtheorem{theorem}{Theorem}
\numberwithin{theorem}{section}
\newtheorem{lemma}[theorem]{Lemma}

\newtheorem{prop}[theorem]{Proposition}

\theoremstyle{definition}

\newtheorem{hypo}[theorem]{Hypothesis}

\newtheorem{rmk}[theorem]{Remark}
\newtheorem{eg}[equation]{Example}
\newtheorem{definition}[theorem]{Definition}

\newtheorem{example}[theorem]{Example}
\newtheorem{property}{}
\numberwithin{property}{theorem}

\newcommand{\bF}{\mathbb{F}}

\newcommand{\F}{\mathcal{F}}
\newcommand{\E}{\mathcal{E}}
\newcommand{\N}{\mathcal{N}}

\newcommand{\K}{\mathcal{K}}
\newcommand{\mG}{\mathcal{G}}
\newcommand{\mD}{\mathcal{D}}

\renewcommand{\L}{\mathcal{L}}
\renewcommand{\H}{\mathcal{H}}

\newcommand{\Hom}{\operatorname{Hom}}
\newcommand{\Aut}{\operatorname{Aut}}
\newcommand{\Out}{\operatorname{Out}}
\newcommand{\Inn}{\operatorname{Inn}}
\newcommand{\End}{\operatorname{End}}
\newcommand{\Syl}{\operatorname{Syl}}

\newcommand{\m}{\mathcal}
\newcommand{\ov}{\overline}

\newcommand{\D}{\mathbf{D}}
\newcommand{\id}{\operatorname{id}}
\newcommand{\SL}{\operatorname{SL}}

\newcommand{\Sp}{\operatorname{Sp}}
\newcommand{\SU}{\operatorname{SU}}
\renewcommand{\O}{\operatorname{O}}
\newcommand{\Sym}{\operatorname{Sym}}
\newcommand{\Alt}{\operatorname{Alt}}
\newcommand{\G}{\operatorname{G}}

\newcommand{\mF}{\mathbb{F}}
\newcommand{\norm}{\trianglelefteq}
\newcommand{\inv}{^{-1}}
\newcommand{\wde}{\widetilde}
\renewcommand{\tilde}{\widetilde}
\newcommand{\fto}{\longrightarrow}
\newcommand{\hyp}{\mathfrak{hyp}}
\renewcommand{\hat}{\widehat}

\def \<{\langle }
\def \>{\rangle }

\renewcommand{\phi}{\varphi}

\newcommand{\GL}{\operatorname{GL}}

\newcommand{\One}{\operatorname{\mathbf{1}}}
\newcommand{\W}{\mathbf{W}}

\title[Large subgroups of fusion systems and localities]{Large subgroups of fusion systems and localities}

\author[E.~Henke]{Ellen Henke}
\address{Institut f{\"u}r Algebra, Fakult{\"a}t Mathematik, Technische Universit{\"a}t Dresden, 01062 Dresden, Germany}
\email{ellen.henke@tu-dresden.de}
\author[E.~Salati]{Edoardo Salati}
\address{Fachbereich Mathematik (RPTU in Kaiserslautern), RPTU Kaiserslautern-Landau, 67663 Kaiserslautern, Germany}
\email{edoardo.salati@rptu.de}

\begin{document}

\begin{abstract}
Saturated fusion systems are categories modeling properties of conjugacy of p-subgroups in finite groups. It was shown by Chermak that they correspond nicely to group-like structures called localities. In this paper we start to explore how concepts and results from a program of Meierfrankenfeld, Stellmacher and Stroth, aiming to reprove and generalize parts of the classification of the finite simple groups, translate to fusion systems and localities. Central in the program is the notion of a large $p$-subgroup. The presence of a large $p$-subgroup in a finite group turns out to be strong enough information to nearly classify the entire $p$-local structure, while also accommodating a very large class of groups of interest including many groups of Lie type in defining characteristic $p$.\\
Utilizing the group-theoretical definition of a large $p$-subgroup as a blueprint, we define large subgroups of fusion systems and localities. We then analyze how the three definitions relate to each other, showing in particular that the newly defined notions behave well under the correspondence between saturated fusion systems and localities with certain properties. We further proceed with an example of how classification results from the program of Meierfrankenfeld et.al. translate to fusion systems and localities. In more detail, we give a new characterization of the $2$-fusion system of $\Aut(\G_2(3))$ following the strategy in a paper of Meierfrankenfeld and Stroth, where the group $\Aut(\G_2(3))$ is characterized in a similar fashion.
\end{abstract}

\maketitle

\section{Introduction}

\subsection{Context}

A saturated fusion system is a category whose objects are the subgroups of a $p$-group $S$ (where $p$ is a prime) and whose morphisms are injective group homomorphisms subject to certain axioms. Every finite group $G$ leads to a saturated fusion system $\F_S(G)$ encoding the conjugacy relations between subgroups of a fixed Sylow $p$-subgroup $S$ of $G$. As all Sylow $p$-subgroups are conjugate, up to a suitable notion of isomorphism the fusion system $\F_S(G)$ depends only on $p$ and $G$ and not on the choice of the Sylow $p$-subgroup $S$. Thus, it is simply referred to as the \emph{$p$-fusion system} of $G$. Its structure is closely connected to the structure of the \emph{$p$-local subgroups} of $G$, i.e. the normalizers of non-trivial $p$-subgroups. The study of $p$-local subgroups, in particular for the prime $2$, is of major importance in the proof of the classification of the finite simple groups.

\smallskip

Saturated fusion systems which do not arise as $p$-fusion systems of finite groups are called exotic. It is of significant interest in the field to learn more about the behaviour and relative abundance of exotic fusion systems. In this context, classification results for fusion systems play a major role. At the same time, a program of Michael Aschbacher aims to simplify the proof of the classification of finite simple groups through classification results for fusion systems at the prime $2$. Aschbacher's program and the general philosphy behind it is a major motivation to study fusion systems and to develop a theory similar to the theory of finite groups. In particular, Aschbacher \cite{Aschbacher:2011} introduced components of fusion systems.

\smallskip

Another motivation to study fusion systems comes from homotopy theory. The Martino--Priddy conjecture (now a theorem proved first by Oliver \cite{Oliver2004,Oliver2006}) states that two finite groups have isomorphic $p$-fusion system if and only if the Bousfield-Kan $p$-completions of their classifying spaces are homotopy equivalent. The Martino--Priddy conjecture follows also from the existence and uniqueness of a centric linking system associated to a given saturated fusion system, which was proved first by Chermak \cite{Chermak:2013} and subsequently by Oliver \cite{Oliver:2013}, with the dependence on the classification of finite simple groups removed by Glauberman and Lynd \cite{GlaubermanLynd2016}. The existence and uniqueness of centric linking systems allows one to attach to each saturated fusion system a ``$p$-completed classifying space'', which in the case of the $p$-fusion system of a finite group $G$ coincides with the $p$-completed classifying space of $G$. This is one motivation for the interest in exotic fusion systems: They lead to new spaces with properties akin to the properties of $p$-completed classifying spaces of finite groups. 

\smallskip

Questions about the exoticity of fusion systems play also a role in the representation theory of finite groups, since every $p$-block of a finite group leads to a saturated fusion system over its defect group. The $p$-fusion system of a finite group is indeed a special case of this, as it can be seen as the fusion system of the principal block. It is conjectured that the fusion systems coming from $p$-blocks of finite groups are precisely the ones which can be obtained as $p$-fusion systems of finite groups.

\smallskip

As mentioned above, classification theorems play a major role in determining the cases in which exotic fusion systems arise. Many classification theorems, in particular many theorems within Aschbacher's program, focus on fusion systems which are simple in a certain well-defined sense. Understanding the exotic behaviour of simple fusion systems helps actually also to understand the exotic behaviour of saturated fusion systems in general. A lemma of Aschbacher \cite[Lemma~2.54]{HenkeLynd2022} combined with a relatively recent result of Broto, M{\o}ller, Oliver and Ruiz \cite[Corollary~B]{BrotoMollerOliverRuiz2023} shows that, assuming the classification of finite simple groups, a given saturated $2$-fusion system is exotic if and only if  at least one of its components is exotic, which in turn is equivalent to the exoticity of a certain simple $2$-fusion system occuring inside of this component. For odd primes, a similar but slightly more complicated reduction exists through the work of Broto, M{\o}ller, Oliver and Ruiz \cite{BrotoMollerOliverRuiz2023} and the reformulation of their work in \cite{HenkeLynd2026}.

\smallskip

For $p=2$ we feel that there is principally some hope that all simple saturated $2$-fusion systems can be classified using ideas from the classification of finite simple groups. The only known exotic simple $2$-fusion systems are the Benson--Solomon fusion systems; see \cite{LeviOliver2002,AschbacherChermak2010}. It is conjectured by Solomon that every simple $2$-fusion system is either a Benson--Solomon fusion system or the $2$-fusion system of some finite simple group. 

\smallskip

As part of the program of Michael Aschbacher, a large subclass of the simple saturated $2$-fusion systems of component type is classified. Here a $2$-fusion system is of component type if the centralizer of a (suitably chosen) involution has a component. The 
generic examples are fusion systems of finite simple groups of Lie type and Lie rank at least $2$ in odd characteristic. To get a more complete picture,  one that includes the saturated fusion systems that are not of component type, it may be possible to complement the results from Aschbacher's program for $p=2$ by translating  theorems from an ongoing program of Meierfrankenfeld, Stellmacher and Stroth to fusion systems. It is the purpose of this paper to make a first attempt at such translations. Central to our approach is the existence of a group-like structure called a \emph{locality} attached to every saturated fusion system. 

\smallskip

In the program of Meierfrankenfeld et.al. the finite groups under consideration are actually not nessarily assumed to be simple, and properties of these groups are investigated with respect to  arbitrary primes. Indeed, many almost simple groups are characterized starting with assumptions on the $p$-local subgroups for some arbitrary prime $p$. Thus, if translations of these results to fusion systems are successful, one can actually expect to get some classification results for saturated fusion systems at arbitrary primes and characterizations of many almost simple fusion systems. Assuming that the program of Meierfrankenfeld et.al. is one day complete, a successful translation would for example include some characterizations of the $p$-fusion systems of finite almost simple groups of Lie type and Lie rank at least $2$ in defining characteristic $p$. It should be pointed out though that, even if such results can be proved, a complete classification of the simple saturated fusion systems at odd primes does not appear to be within reach, in particular since it is unclear how to treat the component type case for odd primes. At the same time, even incomplete results (for even or odd primes) will at least limit the cases where simple exotic fusion systems might occur.

\smallskip

In this paper we translate some basic background results from the program of Meierfrankenfeld et.al. to saturated fusion systems and localities at arbitrary primes. To give a concrete example how this can be used to prove a classification theorem for fusion systems, we do however focus on a configuration which appears only for $2$-fusion systems. More precisely, we obtain a characterization of the $2$-fusion system of $\Aut(\G_2(3))$ by translating a paper of Meierfrankenfeld and Stroth \cite{MStr:2008} to fusion systems and localities. Our choice of the topic is somewhat arbitrary here: We simply picked a paper of relatively short length. It is investigated in the second author's PhD thesis \cite{Salati:PhD} in how far some much more central results by Meierfrankenfeld, Stellmacher and Stroth \cite{MSS} can be translated to fusion systems and localities at arbitrary primes. Journal articles on this PhD thesis and continuations of the project will use many of the background results provided in this paper.

\smallskip

Let us now summarize some more details. The program of Meierfrankenfeld, Stellmacher and Stroth aims to understand groups of local characteristic $p$ and to revise a significant part of the classification of finite simple groups by classifying groups of local characteristic $2$. Here a group $G$ is said to be \emph{of characteristic $p$} if
\[C_G(O_p(G))\leq O_p(G).\]
Moreover, $G$ is of \emph{local characteristic $p$} if every $p$-local subgroup is of characteristic $p$. The generic examples of groups of local characteristic $p$ are the groups of Lie type in defining characteristic $p$.

\smallskip

As mentioned before, when translating results to fusion systems, it is central to our approach that there is a \emph{locality} attached to every saturated fusion system. Here localities are group-like structures introduced by Chermak \cite{Chermak:2013}. Slightly more precisely, a locality consists of a ``partial group'' $\L$ (i.e. a finite set with a multivariable product defined on certain words in $\L$ subject to certain axioms), a maximal $p$-subgroup $S$ of $\L$ which can be thought of as a ``Sylow $p$-subgroup'' of $\L$, and a set $\Delta$ of subgroups of $S$ which in some sense controls the domain of the product on $\L$. There are naturally defined normalizers of elements of $\Delta$ which form finite groups and can thus be thought of as $p$-local subgroups. Usually, we want to consider localities where these normalizers are groups of characteristic $p$. Under some additional hypothesis, we call such localities \emph{linking localities} (cf. Definition~\ref{D:LinkingLoc}). It is a consequence of the existence and uniqueness of centric linking systems that there is a linking locality attached to every saturated fusion systems. A linking locality $(\L,\Delta,S)$ looks locally very much like a group in which the normalizers of elements of $\Delta$ are subgroups of characteristic $p$. Indeed, if $\Delta$ is the set of all subgroups of $S$, then the locality looks locally like a group of local characteristic $p$. The reader is referred to Section~\ref{S:Localities} for a more detailed introduction to partial groups and localities.

\smallskip

We will now give an overview over our main concepts and results. The most important definitions are introduced in Subsection~\ref{SS:IntroLargeDef} below. The results we prove based on these definitions are summarized in Subsection~\ref{SS:IntroLargeProp}. In Subsection~\ref{SS:IntroMStr} we state the classification theorem characterizing the $2$-fusion system of $\Aut(\G_2(3))$, which we obtain by translating the paper by Meierfrankenfeld and Stroth \cite{MStr:2008}.

\subsection{Large subgroups}\label{SS:IntroLargeDef}

For many parts of the program of Meierfrankenfeld et.al., the assumption that the group is of local characteristic $p$ can indeed be relaxed. Central to the major case distinction in the program is the following definition. 

\begin{definition}[{\cite[Definition~1.51]{MSS}}]\label{D:LargeGroup}
A $p$-subgroup $Q$ of a finite group $G$ is called \textbf{large} in $G$ if $C_G(Q)\leq Q$ and 
\begin{equation}\label{Q!}\tag{$Q!$}
 N_G(U)\leq N_G(Q)\mbox{ for all }1\neq U\leq Z(Q).
\end{equation}
The property \eqref{Q!} is also called the \textbf{$Q$-uniqueness property}.
\end{definition}

In fact, if $G$ is a finite simple group of Lie type in defining characteristic $p$ and $S$ is a Sylow $p$-subgroup of $G$, then $O_p(N_G(\Omega_1(Z(S))))$ is a large subgroup if and only if $\Omega_1(Z(S))$ is a root subgroup of $G$, which is the case unless $p=2$ and $G\cong \Sp_{2n}(2^m)$ ($n\geq 2$) or $G\cong \mathrm{F}_4(2^m)$, or $p=3$ and $G\cong \G_2(3^m)$; cf. \cite[p.2]{MSS}. 

\smallskip

The program of Meierfrankenfeld et.al. splits into the treatment of finite groups with a large $p$-subgroup and the treatment of groups of local characteristic $p$ which do not have a large $p$-subgroup. In the first case, local characteristic $p$ is not assumed. However, if a finite group $G$ has a large $p$-subgroup $Q$, then it turns out that every $p$-local subgroup of $G$ containing $Q$ is of characteristic $p$. In particular, every $p$-local subgroup of $G$ containing a Sylow $p$-subgroup of $G$ is of characteristic $p$.

\smallskip

We translate the concept of a large $p$-group to fusion systems and localities as follows. 

\begin{definition}\label{D:LargeFusLoc}
A \textbf{large subgroup} of a (not necessarily saturated) fusion system $\F$ over $S$ is a subgroup $Q\leq S$ with $C_S(Q)\leq Q$ such that
\begin{equation}\label{Q!F}\tag{$Q!\F$}
 N_\F(U)\subseteq N_\F(Q)\mbox{ for all }1\neq U\leq Z(Q).
\end{equation}
If $(\L,\Delta,S)$ is a locality and $Q\leq S$, then $Q$ is a \textbf{large subgroup} of $\L$ if $C_\L(Q)\subseteq Q$ and 
\begin{equation}\label{Q!L}\tag{$Q!\L$}
 N_\L(U)\subseteq N_\L(Q)\mbox{ for all }1\neq U\leq Z(Q).
\end{equation} 
\end{definition}

\subsection{General results about large subgroups of groups, fusion systems and localities}\label{SS:IntroLargeProp}

If $Q$ is a large subgroup of a group $G$, then the following properties hold:
\begin{itemize}
 \item [(1)] $O_p(N_G(Q))$ is a large subgroup of $G$.
 \item [(2)] $Q$ is weakly closed in $G$.
 \item [(3)] If $R$ is another large subgroup of $G$ such that $Q$ and $R$ lie in a common Sylow $p$-subgroup of $G$, then $QR$ is a large subgroup.
 \item [(4)] If $P$ is a non-trivial $p$-subgroup of $G$ normalized by $Q$, then $N_G(P)$ is a group of characteristic $p$.
\end{itemize}
We prove versions of the results (1)-(4) for saturated fusion systems and of the results (1)-(3) for localities (cf. Lemma~\ref{L:LargeFusMain}(a),(c), Lemma~\ref{L:LargeSatFusMain} and Lemma~\ref{L:LargeLocMain}). A result similar to (4) does not need to hold in a locality $(\L,\Delta,S)$ with a large subgroup $Q$. However, for most purposes we want to restrict attention to localities with such a property. Thus, we define a locality $(\L,\Delta,S)$ to be \emph{$Q$-replete} if $Q\leq S$ and $P\in\Delta$ for every non-trivial subgroup $P$ of $S$ normalized by $Q$. Since a result similar to (4) holds for saturated fusion systems (cf. Lemma~\ref{L:LargeSatFusMain}(b)), if $Q$ is large in a saturated fusion system $\F$, then there is a linking locality $(\L,\Delta,S)$ attached to $\F$ which is $Q$-replete (see Example~\ref{E:SubcentricQReplete}).

\smallskip

If $(\L,\Delta,S)$ is a locality over $\F$ and $Q$ is a subgroup of $S$ such that $(\L,\Delta,S)$ is $Q$-replete and $Q$ is large in $\L$, then it turns out that $\F$ is saturated and $Q$ is large in $\F$. On the other hand, for a saturated fusion system $\F$, a large subgroup of $\F$ is large in every linking locality over $\F$. In particular, large subgroups of saturated fusion systems correspond nicely to large subgroups of linking localities with big enough object sets. More details can be found in Lemma~\ref{L:LargeLF}.

\smallskip

Another question one might wonder about is the relationship between large subgroups of a given finite group and large subgroups of the $p$-fusion system of this group. To discuss this let $G$ be a finite group, $S$ a Sylow $p$-subgroup of $G$ and $Q\leq S$. It is indeed elementary to see that $Q$ is large in $\F_S(G)$ if $Q$ is large in $G$. One sees easily that the converse is false: For example, if $G$ is of the form $G=H\times R$, where $S$ is a Sylow $p$-subgroup of $H$, the group $R$ is a $p^\prime$-group, and $Q\leq S$ is large in $H$, then $Q$ is large in $\F_S(G)=\F_S(H)$. However, $Q$ is not large in $G$ as $R\leq C_G(Q)\not\leq Q$.

\smallskip

The just stated example may appear rather artificial for, while $Q$ is not large in $G$, its image is large in $G/O_{p^\prime}(G)$. Moreover, while the self-centralizing property fails in $G$, the $Q$-uniqueness property is true. We next give two series of examples where $O_{p^\prime}(G)=1$, and $Q$ is large $\F_S(G)$, but not in $G$. Furthermore, in these examples the $Q$-uniqueness property fails.

\begin{example}\label{Ex:AutG23}
The group $\G_2(3)$ has a subgroup $Q$ which is isomorphic to the central product of two subgroups isomorphic to $Q_8$. It is large both in $\G_2(3)$ and in $\Aut(\G_2(3))$. In particular, if we fix a Sylow $2$-subgroup $S$ of $\Aut(\G_2(3))$ containing $Q$ and set $T:=S\cap \G_2(3)$, then $Q$ is large in $\F_T(\G_2(3))$ and in $\F_S(\Aut(\G_2(3)))$. 
\begin{itemize}
 \item [(a)] If $q>3$ is an odd prime power with $q\equiv 3,5\mod 8$, then the $2$-fusion systems of $\G_2(3)$ and $\G_2(q)$ are isomorphic. Thus, we may identify $T$ with a Sylow $2$-subgroup of $\G_2(q)$ in such a way that $\F_T(\G_2(q))=\F_T(\G_2(3))$. Then $Q$ is large in $\F_T(\G_2(q))$, but it is not large in $\G_2(q)$. More precisely, the self-centralizing property holds, while the $Q$-uniqueness property fails in $\G_2(q)$. 
 \item [(b)] If $q=3^r$ for some odd integer $r>1$, then the $2$-fusion system of $\Aut(\G_2(3))$ is isomorphic to the $2$-fusion system of $\Aut(\G_2(q))$. Thus, identifying $S$ with a Sylow $2$-subgroup of $\Aut(\G_2(q))$ in a suitable way, we have $\F_S(\Aut(\G_2(q)))=\F_S(\Aut(\G_2(3)))$ and so $Q$ is large in $\F_S(\Aut(\G_2(q)))$. However, $Q$ is not large in $\Aut(\G_2(q))$. Indeed, both the self-centralizing property and the $Q$-uniqueness property are violated in $\Aut(\G_2(q))$. 
\end{itemize}
\end{example}

Example~\ref{Ex:AutG23} is restated and extended in Examples~\ref{Ex:G23G2q} and \ref{Ex:AutG23Extended}, where we also supply proofs of the stated claims. The isomorphism of $2$-fusion systems in (a) is shown in \cite[Theorem~A]{BMO}.

\smallskip

We suspect that examples of finite groups, where the $p$-fusion system possesses a large subgroup which is not large in $G/O_{p^\prime}(G)$, are rather exceptional case. The properties stated in the following lemma together with theorems about groups with strongly closed $p$-subgroups (like Goldschmidt's Theorem for $p=2$) may help to analyze the situation at least in special cases.

\begin{lemma}\label{L:QlargeFSGIntro}
 Let $G$ be a finite group, $S\in\Syl_p(G)$ and $Q\leq S$. Then $Q$ is large in $\F_S(G)$ if and only if the following two conditions hold:
 \begin{itemize}
  \item $C_S(Q)\leq Q$;
  \item for every non-trivial subgroup $P$ of $Z(Q)$, every characteristic subgroup of $Q$ is strongly closed in $N_G(P)$.
 \end{itemize}
\end{lemma}

Lemma~\ref{L:QlargeFSGIntro} follows from Lemma~\ref{L:QlargeFSG}, where more conditions are listed which are equivalent to $Q$ being large in $\F_S(G)$.

\subsection{A characterization of the $2$-fusion system of $\Aut(\G_2(3))$}\label{SS:IntroMStr}

As stated before, we want to give a first example here how results from the program of Meierfrankenfeld et.al. can be translated to fusion systems. Thus, we characterize the $2$-fusion system of $\Aut(\G_2(3))$ following mainly the ideas in a paper by Meierfrankenfeld and Stroth \cite{MStr:2008} characterizing the group $\Aut(\G_2(3))$. Along the way, we point out a gap in the original paper and outline how it can be filled (cf. Remark~\ref{R:GapMStr}).

\smallskip

If $M$ is a finite group, then $Y_M$ is the largest normal elementary abelian $2$-subgroup of $M$ such that $O_2(M/C_M(Y_M))=1$. The existence of $M$ is shown in \cite[Lemma~2.2(a)]{MSS}. 

\smallskip

The theorem below is proved in Section~\ref{S:MS}. As a first step, we prove with Theorem~\ref{unique local} a result about the fusion system of a locality fulfilling a somewhat similar hypothesis. Then we deduce Theorem~\ref{mainMS} from Theorem~\ref{unique local} at the very end of this paper.

\begin{theorem}\label{mainMS}
 Let $\F$ be a saturated fusion system over a finite $2$-group $S$. Suppose $Q\leq S$ is large in $\F$ and $O_2(N_\F(Q))=Q$. Suppose there exists a finite group $M$ of characteristic $2$ containing $S$ as a Sylow $2$-subgroup such that $\F_S(M)\subseteq \F$ and, setting $V:=[Y_M,M]$, the following hold:
 \begin{itemize}
  \item [(I)] $M/O_2(M)\cong \SL_3(2)$ and $V$ is a natural $\SL_3(2)$-module for $M/O_2(M)$.
  \item [(II)] $Y_M\not\leq Q$ and $V\leq Q$.
 \end{itemize}
Then $\F$ is isomorphic to the $2$-fusion system of $\Aut(\G_2(3))$. 
\end{theorem}

Comparing with \cite[Theorem 1]{MStr:2008} it turns out that, among the groups having the same $2$-fusion system of $\Aut(\G_2(3))$, the group $\Aut(\G_2(3))$ is characterized by possessing a large $2$-subgroup. In view of Example~\ref{Ex:AutG23Extended} it seems interesting to notice that a classification of finite groups with a $2$-fusion system fulfilling the hypothesis of Theorem~\ref{mainMS} would not only characterize $\Aut(\G_2(3))$, but a larger class of groups including the groups $\Aut(\G_2(3^r))$ with $r$ odd. However, to achieve such a classification, one would have to characterize finite groups whose $2$-fusion system is isomorphic to the $2$-fusion system of $\Aut(\G_2(3))$. This is beyond the scope of this paper.



\section{Some background results on fusion systems and finite groups}

\subsection{Lemmas on fusion systems}

We refer the reader to \cite[Chapter~I]{Aschbacher/Kessar/Oliver:2011} for background on fusion systems. However, we state some specialized results here. Throughout this subsection let $p$ be a prime and $\F$ a fusion system over a $p$-group $S$.

\begin{lemma}\label{L:ConjFulNormZQ}
Let $\F$ be saturated and suppose that $Q\leq S$ is weakly closed in $\F$. Let $U\leq Z(Q)$ and let $V\in U^\F$ be fully $\F$-normalized. Then $V\leq Z(Q)$ and we have the following implication:
\[N_\F(V)\subseteq N_\F(Q)\Longrightarrow N_\F(U)\subseteq N_\F(Q).\]
\end{lemma}

\begin{proof}
By \cite[Lemma~I.2.6(c)]{Aschbacher/Kessar/Oliver:2011}, there exists $\alpha\in\Hom_\F(N_S(U),S)$ with $U\alpha=V$. As $Q\leq N_S(U)$ and $Q$ is weakly closed in $\F$, it follows that $Q=Q\alpha$. Hence, $\alpha$ is a morphism in $N_\F(Q)$ and $V=U\alpha\leq Z(Q)\alpha=Z(Q)$. 

\smallskip

Suppose now $N_\F(V)\subseteq N_\F(Q)$. Given $X,Y\leq N_S(V)$ and $\phi\in\Hom_{N_\F(U)}(X,Y)$, one sees easily that $\alpha^{-1}\phi\alpha\in\Hom_{N_\F(V)}(X\alpha,Y\alpha)$ and so $\alpha^{-1}\phi\alpha$ is by assumption a morphism in $N_\F(Q)$. As $\alpha$ is a morphism in $N_\F(Q)$, it follows thus that $\phi$ is a morphism in $N_\F(Q)$. This shows $N_\F(U)\subseteq N_\F(Q)$ and thus the required implication.
\end{proof}

 \begin{lemma}\label{L:ZRfullynorm}
Let $\F$ be saturated. Then for every $R\leq S$, there exists $R'\in R^\F$ such that $Z(R')$ is fully $\F$-normalized.
 \end{lemma}
 
\begin{proof}
For every $R\leq S$, there exists by \cite[Lemma~I.2.6(c)]{Aschbacher/Kessar/Oliver:2011} a morphism $\alpha\in\Hom_\F(N_S(Z(R)),S)$ such that $Z(R)\alpha$ is fully $\F$-normalized. We have then $R':=R\alpha\in R^\F$ and  $Z(R')=Z(R)\alpha$ is fully $\F$-normalized.
\end{proof}

\begin{lemma}\label{L:FusNormEquivalent}
Let $\F$ be saturated and $Q\leq S$. Then the following conditions are equivalent:
 \begin{itemize}
  \item [(i)] $Q\unlhd \F$.
  \item [(ii)] Every characteristic subgroup of $Q$ is strongly closed in $\F$.
  \item [(iii)] There exists a series $1=Q_0\leq Q_1\leq\cdots \leq Q_n=Q$ of subgroups of $Q$ such that $Q_i$ is strongly closed in $\F$ and $[Q_i,Q]\leq Q_{i-1}$ for all $i=1,2,\dots,n$.
 \end{itemize}
\end{lemma}
 
\begin{proof}
One checks easily that (i) implies (ii). Considering an upper or lower central series for $Q$, one sees moreover that  (ii) implies (iii). It is shown in \cite[Proposition~I.4.6]{Aschbacher/Kessar/Oliver:2011} that (iii) implies (i).
\end{proof}

\begin{lemma}\label{CentralSeries}
 Let $P\in\F^{cr}$ such that there exists a series $1=P_0\leq P_1\leq \dots \leq P_n=P$ which is $\Aut_\F(P)$-invariant. Let $X\leq S$ such that $[P_i,N_X(P)]\leq P_{i-1}$ for $i=1,\dots,n$ and $PX=XP$. Then $X\leq P$.
\end{lemma}

\begin{proof}
 Note that $[P_i,\Aut_X(P)]\leq P_{i-1}$ for $i=1\dots,n$. Set $N:=\<\Aut_X(P)^{\Aut_\F(P)}\>$. As the $P_i$ are $\Aut_\F(P)$-invariant, it follows that $[P_i,N]\leq P_{i-1}$ for $i=1,\dots,n$. Hence, by \cite[Lemma~A.2]{Aschbacher/Kessar/Oliver:2011}, $O^p(N)=1$ and $N$ is thus a normal $p$-subgroup of $\Aut_\F(P)$. Therefore, $\Aut_X(P)\leq N\leq O_p(\Aut_\F(P))=\Inn(P)$ as $P$ is radical. So $N_X(P)\leq PC_S(P)=P$ as $P$ is centric. Since by assumption $XP=PX$, the product $PX$ is a subgroup of $S$ and thus a $p$-group. Hence, $N_{PX}(P)=PN_X(P)=P$ implies $PX=P$ and thus $X\leq P$. 
\end{proof}

For every subgroup $R\leq S$, we set 
\[\Out_\F(R):=\Aut_\F(R)/\Inn(R)\mbox{ and }\Out_S(R):=\Aut_S(R)/\Inn(R).\]

\begin{lemma}\label{ActionRModPhiR}
 Let $R\in\F^{cr}$. Then $\Out_\F(R)$ acts faithfully on $R/\Phi(R)$.
\end{lemma}

\begin{proof}
As $R/\Phi(R)$ is abelian, $\Out_\F(R)=\Aut_\F(R)/\Inn(R)$ acts on $R/\Phi(R)$. By \cite[8.2.9(b)]{Kurzweil/Stellmacher:2004}, $C_{\Aut_\F(R)}(R/\Phi(R))$ is a $p$-group and thus contained in $O_p(\Aut_\F(R))=\Inn(R)$. This proves the assertion. 
\end{proof}

We recall from \cite[Definition~I.3.2]{Aschbacher/Kessar/Oliver:2011} that a subgroup $R\leq S$ is called \emph{essential} in $\F$, if $R$ is $\F$-centric and fully $\F$-normalized, and moreover $\Out_\F(R)$ has a strongly $p$-embedded subgroup. Every essential subgroup of $\F$ is in $\F^{cr}$.

\smallskip

If $G$ is a finite group with a strongly $p$-embedded subgroup $H$, then $H\cap O^{p^\prime}(G)$ is strongly $p$-embedded in $O^{p^\prime}(G)$. So if $R$ is essential, then $O^{p^\prime}(\Out_\F(R))$ has a strongly $p$-embedded subgroup. We use this fact in the proof of the following lemma.

\begin{lemma}\label{FF}
 Let $R$ be an essential subgroup of $\F$. Suppose there exists $A\leq N_S(R)$ such that for $\ov{R}=R/\Phi(R)$, we have $|\ov{R}/C_{\ov{R}}(A)|\leq |A/(A\cap R)|$. Then we have
 \[
 q:=|N_S(R)/R| =|A/A\cap R| , \qquad O^{p^\prime}(\Out_\F(R)) \cong \SL_2(q)
 \]
 and $\ov{R}/C_{\ov{R}}(O^{p^\prime}(\Out_\F(R)))$ is a natural $\SL_2(q)$-module for $O^{p^\prime}(\Out_\F(R))$. In particular, 
 \[|\ov{R}|\geq |N_S(R)/R|^2.\]
\end{lemma}

\begin{proof}
As $R$ is essential, $R\in\F^{cr}$ and thus, by Lemma~\ref{ActionRModPhiR}, $\Out_\F(R)$ acts faithfully on $\ov{R}$. As $R\in\F^f$, $\Aut_S(R)\in\Syl_p(\Aut_\F(R))$. Since $R\in\F^c$, we have $N_S(R)/R\cong \Aut_S(R)/\Inn(R)$ and $A/A\cap R\cong AR/R\cong \Aut_A(R)\Inn(R)/\Inn(R)$. In particular, $\Aut_A(R)\Inn(R)/\Inn(R)$ is an offender on $\ov{R}$. Now the assertion follows from Lemma~4.6(a), Lemma~4.10 and Theorem~5.6 in \cite{Henke:2010}.

\smallskip

In detail, $AR/R$ is an offender on $\ov{R}/C_{\ov{R}}(O^{2^\prime}(\Out_\F(R)))$ by  \cite[Lemma~4.6(a)]{Henke:2010} with $(\ov{R}, O^{2^\prime}(\Out_\F(R)))$ in place of $(V,G)$. Then by \cite[Theorem~5.6]{Henke:2010}, $O^{2^\prime}(\Out_\F(R)) \cong \SL_2(q)$, where $q$ is the size of a Sylow $2$-subgroup of $\Out_\F(R)$, thus $q=|\Aut_S(R)/\Inn(R)|$, and finally \cite[Lemma~4.10]{Henke:2010} gives that $AR/R$ is a Sylow $2$-subgroup of $O^{2^\prime}(\Out_\F(R))$ and thus has size $q$.
\end{proof}

The following lemma is solely needed to prove Example~\ref{Ex:AutG23Extended}(b). It may also be useful in other cases to find isomorphisms between $p$-fusion systems of finite almost simple groups of Lie Type.

\begin{lemma}\label{L:AlmostSimpleFusIsoHelp}
Let $\F$ be saturated and suppose $\mG$ is a saturated subsystem of $\F$ over $S$. Let $\E$ be a subsystem of $\mG$ over $T\leq S$ such that $\E$ is a normal subsystem of $\F$ of $p$-power index. Then $\F=\mG$. 
\end{lemma}

\begin{proof}
As $\E$ has $p$-power index in $\F$, it has also $p$-power index in $\G$. Now it follows from \cite[Lemma~2.29]{Chermak/Henke} that $O^p(\F)=O^p(\E)=O^p(\G)$. By \cite[Theorem~1]{Henke:2013} there is a unique saturated subsystem $\mD=\E S$ of $\F$ over $S$ with $O^p(\mD)=O^p(\E)$. Our assumption yields thus $\F=\mD=\mG$.
\end{proof}

\subsection{Background results on finite groups} As before let $p$ always be a prime.

\smallskip

\begin{definition}~\label{D:YG}
 \begin{itemize}
  \item If $H$ is a finite group and $V$ an $\bF_p H$-module, then $V$ is called \emph{$p$-reduced}, if $O_p(H/C_H(V))=1$.
  \item For any finite group $G$, we denote by $Y_G$ the largest elementary abelian normal $p$-subgroup of $G$ which is a $p$-reduced $G$-module.
 \end{itemize}
\end{definition}

The existence of $Y_G$ was shown in \cite[Lemma~2.2(a)]{MS:2009}. If $G$ is a group of characteristic $p$ (i.e. $C_G(O_p(G))\leq O_p(G)$), then $\<\Omega_1(Z(T))\colon T\in\Syl_p(G)\>$ is a non-trivial elementary abelian normal subgroup of $G$ which is a $p$-reduced $G$-module. Hence, $Y_G$ non-trivial if $G$ is a group of characteristic $p$. 

\smallskip

We will need the following elementary technical lemma to deduce Theorem~\ref{mainMS} from Theorem~\ref{unique local}, which is stated later on.  

\begin{lemma}\label{L:pReduced}
 Let $G$ be a finite group, $R:=O_p(G)$ and $H:=\Aut_G(R)$. Then $Y_G$ is the largest $p$-reduced $H$-submodule of $\Omega_1(Z(R))$. 
\end{lemma}

\begin{proof}
 As $Y_G$ is elementary abelian and a $p$-reduced $G$-module, we have $Y_G\leq \Omega_1(Z(R))$. Moreover, a subgroup $Y\leq \Omega_1(Z(R))$ is an $H$-submodule if and only if it is normal in $G$. Fix now an $H$-submodule $Y$ of $\Omega_1(Z(R))$. Note that $C_G(R)\leq C_G(Y)$, $H\cong G/C_G(Y)$ and 
 \[H/C_H(Y)\cong G/C_G(Y).\]
 So $Y$ is a $p$-reduced $H$-submodule of $\Omega_1(Z(R))$ if and only if it is a $p$-reduced $G$-module. This implies the assertion.
\end{proof}

\begin{lemma}[{Cf. \cite[Lemma~1.3]{MStr:2008}}]\label{Lemma1.3}
 Let $H$ be a group, $\bF$ a field, $W$ an $\bF H$-module, and $A\unlhd B\leq H$. Suppose there exists an irreducible $\bF B$-submodule $Y$ of $W$ with $[W,A]\leq Y$ and $W=\<Y^H\>$. Then every proper $\bF H$-submodule of $W$ is centralized by $\<A^H\>$. In particular, $W/C_W(\<A^H\>)$ is either trivial or an irreducible $\bF H$-module. 
\end{lemma}

\begin{proof}
 This is \cite[Lemma~1.3]{MStr:2008} except that the authors forgot to mention the possibility that $W/C_W(\<A^H\>)$ might be trivial.
\end{proof}

In this paper we will frequently use a theorem of Gasch{\"u}tz in the form of the following lemma.

\begin{lemma}\label{L:ApplyGaschutz}
 Let $G$ be a finite group acting on an abelian $p$-subgroup $V$. Let $S$ be a Sylow $p$-subgroup of $G$. Then the following hold.
 \begin{itemize}
  \item [(a)] If $V=U_1\times U_2$ for a $G$-invariant subgroup $U_1$ of $V$ and an $S$-invariant subgroup $U_2$ of $V$, then there exists a $G$-invariant subgroup $W$ of $V$ such that $V=U_1\times W$.
  \item [(b)] If $V$ is elementary abelian and $C_V(G)\not\leq [V,S]$, then $[V,G]\neq V$.
  \item [(c)] If $V$ is elementary abelian and $V=[V,G]C_V(S)$, then $V=[V,G]C_V(G)$.
 \end{itemize}
\end{lemma}

\begin{proof}
\textbf{(a)} Let $U_1$ and $U_2$ be as in (a). Consider the semidirect product $G\ltimes V$ and note that it contains $S\ltimes V$ as a Sylow $p$-subgroup. Moreover, $U_1$ is normal in $G\ltimes V$, and $S\ltimes U_2$ is a complement to $U_1$ in $S\ltimes V$. Hence, by a Gasch{\"u}tz' Theorem (see e.g. \cite[3.3.2]{Kurzweil/Stellmacher:2004}), there exists a complement $H$ to $U_1$ in $G\ltimes V$, i.e. $G\ltimes V=HU_1$ and $H\cap U_1=1$. Thus, setting $W:=H\cap V$, we have $V=WU_1$ and then $V=U_1\times W$. Since $V$ is abelian and $W$ is $H$-invariant, it follows that $W$ is normal in $G\ltimes V=HU_1$. In particular, $W$ is $G$-invariant. This shows (a).

\smallskip

\textbf{(b)} Suppose $C_V(G)\not\leq [V,S]$ and $V$ is elementary abelian. We may fix then a non-trivial subgroup $U_1$ of $C_V(G)$ with $U_1\cap [V,S]=1$. Furthermore, there exists a complement $U_2$ to $U_1$ in $V$ with $[V,S]\leq U_2$. Note that $U_2$ is $S$-invariant as $[U_2,S]\leq [V,S]\leq U_2$. Clearly $U_1$ is $G$-invariant. Hence, by (a), there exists $W\leq V$ such that $V=U_1\times W$ and $W$ is $G$-invariant. As $U_1$ is non-trivial and centralized by $G$, it follows $[V,G]\leq W<V$ and so (b) holds.

\smallskip

\textbf{(c)} Suppose $V=[V,G]C_V(S)$ and $V$ is elementary abelian. We may fix then $U_2\leq C_V(S)$ such that $U_2$ is a complement to $[V,G]$ in $V$. By (a) applied with $U_1=[V,G]$, there exists a $G$-invariant subgroup $W$ of $V$ such that $V=[V,G]\times W$. Then $[W,G]\leq [V,G]\cap W=1$. So $W\leq C_V(G)$ and $V=[V,G]C_V(G)$. This shows (c).
\end{proof}

\begin{lemma}\label{AutSemidirect}
 Let $G$ be a finite group and suppose $G=\<x\>\ltimes H$ for some $x\in G$ and some normal subgroup $H$ of $G$. Let $z\in Z(H)$ such that $o(x)=o(xz)$. Then there exists an automorphism $\beta\in\Aut(G)$ such that $\beta|_H=\id_H$ and $x^\beta=xz$.
\end{lemma}

\begin{proof}
Notice that $G=\<x,H\>=\<xz,H\>$ as $z\in H$. As $H$ is normal in $G$, this implies $\G=\<xz\>H$. Since $o(xz)=o(x)=|G:H|$, it follows that $\<xz\>\cap H=1$ and so $G=\<xz\>\ltimes H$. Note that $x$ and $xz$ induce the same automorphism of $H$ via conjugation as $z\in Z(H)$. From that one sees easily that the map $\beta\colon G\rightarrow G$ sending $x^ih$ to $(xz)^ih$ for all integers $i\geq 0$ and all $h\in H$ is well-defined and an automorphism of $G$ with the required properties.
\end{proof}

We continue now with some more specialized results.

\begin{lemma}\label{charact-of-symm-gr}
Let $G$ be a finite group with $O_2(G)=1$ and suppose there exists a subgroup $H \norm G$ such that $H \cong \Alt(n)$ for $n\geq 3$ with $n\neq 6$ and $G/H$ is a non-trivial 2-group. Then $G \cong \Sym(n)$.
\end{lemma}

\begin{proof}
Note that $C:=C_G(H)$ is normal in $G$. Since  $H$ is simple, we have $C \cap H = 1$ so that $CH/H \cong C$ showing that $C$ is a 2-group. Thus $C = 1$ as $O_2(G)=1$. Therefore, $G$ is isomorphic to a subgroup of $\Aut(\Alt(n))\cong \Sym(n)$. As $H\cong\Alt(n)$ is properly contained in $G$, it follows $G \cong \Sym(n)$.
\end{proof}

\begin{lemma}\label{L:GroupOrder180}
 Let $G$ be a group of order $180$. Let $S$ be a Sylow $2$-subgroup of $G$. Then $C_G(S)\neq S$. In particular, $N_G(S)$ is not isomorphic to $\Alt(4)$.
\end{lemma}

\begin{proof}
Note that $|G|=180=2^2\cdot 3^2\cdot 5$. For $q\in\{3,5\}$ let $P_q$ be a Sylow $q$-subgroup of $G$. By Sylow's theorem, $|\Syl_5(G)|=|G:N_G(P_5)|$ divides $|G:P_5|$ and is congruent to $1$ modulo $5$. Hence, we have
\[|\Syl_5(G)|\in\{6,2^2\cdot 3^2\}.\]
If $|\Syl_5(G)|=2^2\cdot 3^2$, then $N_G(P_5)=P_5=C_G(P_5)$. Thus, $G$ has a normal $5$-complement $N$ by Burnside's Theorem. Then $S\in \Syl_2(N)$ and the Frattini argument yields that $5$ divides the order of $N_G(S)$ and thus the order of $C_G(S)$, as $\Aut(S)$ is not divisible by $5$. Thus, we may assume from now on that 
\begin{equation}\label{E:Syl5}
|\Syl_5(G)|=6\mbox{ and }|N_G(P_5)|=2\cdot 3\cdot 5. 
\end{equation}
Note that $\Aut(P_3)\cong GL_2(3)$ if $P_3$ is elementary abelian, and $\Aut(P_3)\cong C_6$ if $P_3$ is cyclic. So in either case, $\Aut(P_3)$ is not divisible by $5$. Thus, a Sylow $5$-subgroup of $G$ which normalizes $P_3$ would also centralize $P_3$, which would contradict \eqref{E:Syl5}. Hence,
\begin{equation}\label{E:Syl3}
|N_G(P_3)|\mbox{ is not divisible by }5.
\end{equation}
Let now $D\in \Syl_3(N_G(P_5))$. According to \eqref{E:Syl5}, $D$ has order $3$. Choose $P_3$ such that $D\leq P_3$. As $\Aut(P_5)\cong C_4$, we have $P_5\leq C_G(D)$. Since $P_3$ is abelian, we have also $P_3\leq C_G(D)$. In particular, $3^2\cdot 5$ divides $|C_G(D)|$. If $4$ divides the order of $C_G(D)$, then $D$ centralizes a Sylow $2$-subgroup of $G$, and thus $3$ divides $C_G(S)$. Thus, we may assume $|C_G(D)|\in \{2\cdot 3^2\cdot 5, 3^2\cdot 5\}$. If $|C_G(D)|= 2\cdot 3^2\cdot 5$, then Burnside's Theorem implies that $C_G(D)$ has a normal $2$-complement. So in any case, there is a normal subgroup $H$ of $C_G(D)$ of order $3^2\cdot 5$. Then $P_3$ and $P_5$ are contained in $H$. Moreover, $H/D$ has order $15$ and is thus cyclic. This yields that $P_3$ is normal in $H$ and thus $5$ divides $N_G(P_3)$, which contradicts \eqref{E:Syl3}. 
\end{proof}

\begin{lemma}\label{L:DetermineMfromM*}
 Let $G$ be a finite group such that $Y:=O_2(G)$ is elementary abelian, $G/Y\cong L_3(2)$, $V:=[Y,G]$ is a natural $L_3(2)$-module for $G/Y$, $Y\neq V$, $C_Y(G)=1$ and $Y\cap O^2(G)=V$. Then $G\cong \Aut(O^2(G))$. 
\end{lemma}

\begin{proof}
Set $G^*:=O^2(G)$. As $G/Y\cong L_3(2)$ is non-abelian simple, we have $G=G^*Y$ and $C_G(G^*)=C_Y(G^*)=C_Y(G)=1$, where the second equality uses that $Y$ is abelian. Hence, 
\[G\cong \Aut_G(G^*)\leq \Aut(G^*).\] 
As $Y\cap G^*=V$, we have $G^*/V\cong L_3(2)$ and $V$ is a natural $L_3(2)$-module for $G^*/V$. In particular, $\Aut(V)=\Aut_{G^*}(V)$ and $V:=O_2(G^*)$ is characteristic in $G^*$. Thus, setting 
\[\hat{Y}:=C_{\Aut(G^*)}(V),\]
we have 
\[\Aut(G^*)=\Inn(G^*)\hat{Y}.\]
Observe also that $[G^*,\hat{Y}]\leq C_{G^*}(V)=V$. Moreover, the action of $\Aut(G^*)$ on $G^*$ is naturally equivalent to the action of $\Aut(G^*)$ on $\Inn(G^*)\cong G^*$. In particular, $C_{\Aut(G^*)}(\Inn(G^*))=1$ and, setting $\hat{V}:=\Aut_V(G^*)$, we have  $\hat{Y}=C_{\Aut(G^*)}(\hat{V})$. Hence, it follows from \cite[Lemma~1.4]{Henke:Regular} applied with $(\hat{V},\Inn(G^*),\Aut(G^*))$ in place of $(U,N,G)$ that $O^2(\hat{Y})=1$ and $\hat{Y}'=1$. Thus, $\hat{Y}$ is an abelian $2$-group. 

\smallskip

We want to argue now that $\hat{Y}$ is elementary abelian. For an arbitrary choice of $\alpha\in\hat{Y}$ and $g\in G^*$, it is enough to show that $\alpha^2$ centralizes $g$. Indeed, the commutator formular \cite[1.5.4]{Kurzweil/Stellmacher:2004} gives that 
\[[g,\alpha^2]=[g,\alpha][g,\alpha]^\alpha=[g,\alpha]^2=1,\]
where the second equality and third equality use that $[g,\alpha]\in [G^*,\hat{Y}]\leq V$, the second equality uses moreover the definition of $\hat{Y}$, and the third equality uses also that $V$ is elementary abelian. So we have shown that $\hat{Y}$ is elementary abelian. 

\smallskip

Note that $\hat{V}\cong V$ is a natural $L_3(2)$-module for $\Inn(G^*)/\hat{V}$ and $C_{\hat{Y}}(\Inn(G^*))=C_{\hat{Y}}(G^*)=1$. As $[\hat{Y},\Inn(G^*)]\leq \hat{Y}\cap \Inn(G^*)=\hat{V}$ and $\hat{V}=[\hat{V},\Inn(G^*)]$, we have also $[\hat{Y},\Inn(G^*)]=\hat{V}$. Hence, it follows from \cite[Lemma~1.1(b)]{MStr:2008} that $|\hat{Y}/\hat{V}|=2$. As $V< Y$, we have $\hat{V}<\Aut_Y(G^*)\leq \hat{Y}$ and thus $\hat{Y}=\Aut_Y(G^*)$. As $G=G^*Y\cong \Aut_G(G^*)$, it follows that 
\[G\cong \Aut_G(G^*)=\Inn(G^*)\hat{Y}=\Aut(G^*)\]
as required.
\end{proof}

\begin{lemma}\label{A4Module}
 Let $U$ be an abelian $2$-group on which $G\cong A_4$ acts. Suppose there are $V\leq Y\leq U$ such that $|V|=2^3$, $|Y|=2^4$, $V=[U,G]$, $C_V(O_2(G))$ has order $2$, and $C_Y(t)\leq V$ for every involution $t\in G$. Then $U=YC_U(G)$.
\end{lemma}

\begin{proof}
Observe first that $G$ acts faithfully on $V$ as $[V,O_2(G)]\neq 1$ and $O_2(G)$ is contained in any non-trivial normal subgroup of $G$. Let $t\in G$ be an involution. Then $|V/C_V(t)|=|[V,t]|\leq |C_V(t)|$ and thus, as $|V|=2^3$, $|V/C_V(t)|=2$. In particular, $|U/C_U(t)|=|[U,t]|\leq |C_V(t)|=4$ and $|Y/C_Y(t)|=4$, so $U=YC_U(t)$. Set $V_1:=VC_U(t)$. Note that $V_1$ is a $G$-submodule since $[U,G]=V$. Moreover, $|V_1/C_{V_1}(t)|=|V/C_V(t)|=2$. As all involutions in $G$ are conjugate in $G$, it follows $|V_1/C_{V_1}(s)|=2$ for any involution $s\in G$. Moreover, $O_2(G)$ is generated by two involutions and hence, $|V_1/C_{V_1}(O_2(G))|\leq 4$. By assumption, $|V/C_V(O_2(G))|=4$ and so it follows $V_1=VC_{V_1}(O_2(G))$. Hence,
$$U=YC_U(t)=YV_1=Y(VC_{V_1}(O_2(G)))=YC_U(O_2(G)).$$
Moreover, $[C_U(O_2(G)),G]\leq C_V(O_2(G))$ and, as $|C_V(O_2(G))|=2$, $[C_V(O_2(G)),G]=1$. Using $G=O^2(G)$ we obtain $C_U(O_2(G))=C_U(G)$ and thus $U=YC_U(G)$ as required.
\end{proof}

We say that a group $A$ acts quadratically on a group $V$, if $A$ acts on $V$ such that $[V,A,A]=1$.

\begin{lemma}\label{L:TensorProductModule}
Let $G=G_1\times G_2$ be a finite group and let $V$ be an irreducible $\mathbb{F}_2G$-module. Suppose $G_1\cong L_3(2)$ and let  $H\leq G_1$ such that $H\cong S_4$. Assume furthermore that $Q:=O_2(H)$ acts quadratically on $V$ and $W:=[V,Q]$ is as an $((H/Q)\times G_2)$-module isomorphic to the tensor product of a natural $\SL_2(2)$-module for $H/Q$ and an irreducible $G_2$-module $V_2$. Then the following hold:
\begin{itemize}
\item[(a)] As a $G_2$-module, $V$ is isomorphic to the direct sum of three copies of $V_2$.  
\item[(b)] $[V,O^2(H)]\leq W$.   
\end{itemize}
\end{lemma}

\begin{proof}
By abuse of notation we assume that $V_2$ is an irreducible $G_2$-submodule of $W$. Write $m$ for the $\mathbb{F}_2$-dimension of $V_2$. As $V_2$ is irreducible, we may in particular choose a basis of $V_2$ such that all basis elements are conjugate under $G_2$. Our assumption allows us therefore to choose an $H$-submodule $U_1$ of $W$, such that $U_1$ is a natural $\SL_2(2)$-module for $H/Q$, $|U_1\cap V_2|=2$  and $W=\bigoplus_{i=1}^m U_1^{b_i}$ for some $b_1,\dots,b_m\in G_2$ with $b_1=1$. Choose now an irreducible $G_1$-submodule $V_1$ of $V$ with $U_1\leq V_1$. 

\smallskip

Notice that $Q:=O_2(H)$ acts quadratically on $V_1$ as it acts quadratically on $V$. The group $\SL_3(2)$ has up to isomorphism precisely three irreducible $\mathbb{F}_2$-modules, namely the two natural modules and the Steinberg module of dimension $8$. As there is no fours group acting quadratically on the latter, it follows that $V_1$ must be a natural $\SL_3(2)$-module. Note now that $U_1\leq W\cap V_1\leq C_{V_1}(Q)$. As $C_{V_1}(Q)$ has order at most $4$, it follows that we have equality everywhere, i.e. $U_1=W\cap V_1=C_{V_1}(Q)$. Now $[V_1,Q]\leq C_{V_1}(Q)=U_1$. As $U_1$ is an irreducible $H$-module, it follows indeed that 
\[[V_1,Q]=U_1= W\cap V_1=C_{V_1}(Q).\]

\smallskip

Note that $V_1^b$ is a $G_1$-module isomorphic to $V_1$ for every $b\in G_2$. In particular, $X:=\sum_{i=1}^mV_1^{b_i}$ is a $G_1$-submodule of $V$. Notice that $W=\sum_{i=1}^m U_1^{b_i}\leq X$. 
Thus, for every $b\in G_2$, we have $U_1^b\leq W\cap V_1^b\leq X\cap V_1^b$. As $V_1^b$ is an irreducible $G_1$-module, it follows $V_1^b\leq X$. This shows that $X$ is $G_2$-invariant and thus a $G$-submodule of $V$. As $V$ is irreducible, it follows $V=X=\sum_{i=1}^m V_1^{b_i}$. If this sum is not direct, then there exists $i\leq m$ with $V_1^{b_i}\cap \sum_{j\neq i}V_1^{b_j}\neq 0$ and thus, as $V_1^{b_i}$ is irreducible, $V_1^{b_i}\leq \sum_{j\neq i}V_1^{b_j}$. This implies 
\[U_1^{b_i}=[V_1^{b_i},Q]\leq \sum_{j\neq i}[V_1^{b_j},Q]=\sum_{j\neq i}U_1^{b_j}\]
contradicting $W=\bigoplus_{j=1}^mU_1^{b_j}$. Hence, $V=\bigoplus_{i=1}^m V_1^{b_i}$. 
In particular, $\dim(V)=m\cdot \dim(V_1)=3m$.

\smallskip

As $U_1$ has index $2$ in $V_1$, we have $[V_1,O^2(H)]\leq U_1$ and thus $[V_1^{b_i},O^2(H)]\leq U_1^{b_i}\leq W$ for $i=1,2,\dots,m$. As $V=X$, this shows (b). 

\smallskip

Notice that $G_1$ acts transitively on the non-trivial elements of $V_1$. Thus, we may choose $g\in G_1$ such that $(U_1\cap V_2)^g\not\leq U_1$. As $V_1\cap W=U_1$, this implies $V_2^g\not\leq W$. Notice that $V_2^g$ is a $G_2$-module, which is isomorphic to $V_2$ and thus irreducible. Hence, it follows $V_2^g\cap W=0$. Our assumption yields that $W$ as a $G_2$-module is isomorphic to the direct sum of two copies of $V_2$. In particular, $\dim(W)=2m$. As $\dim(V)=3m$, it follows that $V=W\oplus V_2^g$ is as a $G_2$-module isomorphic to a direct sum of three copies of $V_2$. 
\end{proof}

We remark that, with the notation as in the proof above, we could indeed show that the $G$-module $V$ is isomorphic to the tensor product of $V_1$ with $V_2$. This is however not needed later on.

\section{Partial groups and Localities}\label{S:Localities}

For the convenience of the reader we summarize here the basic background on localtities which is needed in this paper. Our notation and terminology follows Chermak \cite{Chermak:2013} and \cite{Chermak:2022}. We write $\W(\L)$ for the set of words in a set $\L$, by $\emptyset$ we denote the empty word, and $v_1\circ v_2\circ\dots\circ v_n$ denotes the concatenation of words $v_1,\dots,v_n\in\W(\L)$. The set $\L$ is identified with the set of words of length one in the natural way.

\subsection{Partial groups} 

The following definition of a partial group is due to Chermak (cf.  \cite[Definition~2.1]{Chermak:2013} or \cite[Definition~1.1]{Chermak:2022}). 

\begin{definition}\label{D:partial}
Suppose $\L$ is a non-empty set and $\D\subseteq\W(\L)$. Let $\Pi \colon  \D \longrightarrow \L$ be a map, and let $(-)^{-1} \colon \L \longrightarrow \L$ be an involutory bijection, which we extend to a map 
\[(-)^{-1} \colon \W(\L) \longrightarrow \W(\L), w = (g_1, \dots, g_k) \mapsto w^{-1} = (g_k^{-1}, \dots, g_1^{-1}).\]
Then $\L$ is called a \emph{partial group} with product $\Pi$ and inversion $(-)^{-1}$ if the following hold for all words $u,v,w\in\W(\L)$:
\begin{itemize}
\item  $\L \subseteq \D$ and
\[  u \circ v \in \D \Longrightarrow u,v \in \D.\]
(So in particular, $\emptyset\in\D$.)
\item $\Pi$ restricts to the identity map on $\L$.
\item $u \circ v \circ w \in \D \Longrightarrow u \circ (\Pi(v)) \circ w \in \D$, and $\Pi(u \circ v \circ w) = \Pi(u \circ (\Pi(v)) \circ w)$.
\item $w \in  \D \Longrightarrow  w^{-1} \circ w\in \D$ and $\Pi(w^{-1} \circ w) = \One$ where $\One:=\Pi(\emptyset)$.
\end{itemize}
\end{definition}

\smallskip

\textbf{For the remainder of this section let $\L$ be a partial group and $\Pi\colon \D\rightarrow \L$ the product on $\L$.}

\smallskip

Write  $f_1f_2\dots f_n$ for the product $\Pi(v)$ of a word $v=(f_1,\dots,f_n)\in\D$, and set $\One=\Pi(\emptyset)$. 

\smallskip

A subset $\m{H}$ of $\L$ is called a \textit{partial subgroup} of $\L$ if $h^{-1}\in \H$ for all $h\in\m{H}$ and $\Pi(v)\in\m{H}$ for all words $v\in\W(\m{H})\cap\D$. If $\H$ is a partial subgroup then, as $\emptyset\in\W(\H)\cap\D$, we have $\One=\Pi(\emptyset)\in\H$ and so $\H$ is non-empty. 

\smallskip

A partial subgroup $\H$ of $\L$ is called a \textit{subgroup} of $\L$ if $\W(\H)\subseteq\D$. If $\H$ is a subgroup of $\L$ then $\H$ forms indeed a group if one takes the restriction of $\Pi$ to $\H\times\H$ as the binary multiplication. In particular, we can consider \textit{$p$-subgroups} of $\L$, i.e. subgroup of $\L$ whose order is a power of $p$. 

\smallskip

A partial subgroup $\N$ of $\L$ is called a \textit{partial normal subgroup} of $\L$ (denoted $\N\unlhd\L$) if $x^f:=\Pi(f^{-1},x,f)\in\N$ for all $x\in\N$ and $f\in\L$ with $(f^{-1},x,f)\in\D$.

\smallskip

For any $g\in\L$, we denote by $\D(g)$ the set of $x\in\L$ with $(g^{-1},x,g)\in\D$. Thus, $\D(g)$ denotes the set of elements $x\in\L$ for which the conjugate $x^g:=\Pi(g^{-1},x,g)$ is defined. If $g\in\L$ and $X\subseteq \D(g)$ we set $X^g:=\{x^g\colon x\in X\}$. If we write $X^g$ for some $g\in\L$ and some subset $X\subseteq \L$, we will always implicitly mean that $X\subseteq\D(g)$.

\smallskip

If $\mathcal{X}$ and $\mathcal{Y}$ are subsets of $\L$, then we set
\[\mathcal{X}\mathcal{Y}:=\{xy\colon x\in\mathcal{X},\;y\in\mathcal{Y}\mbox{ with }(x,y)\in\D\}.\]

\begin{definition}
 Let $\L$ and $\L'$ be partial groups with products $\Pi\colon \D\rightarrow \L$ and $\Pi'\colon \D'\rightarrow\L'$ respectively. Let $\beta\colon\L\rightarrow\L'$ be a map. Let $\beta^*$ be the induced map on words
 \[\beta^*\colon \W(\L)\rightarrow\W(\L'),(f_1,\dots,f_n)\mapsto (f_1\beta,\dots,f_n\beta).\]
 Then $\beta$ is called a \emph{homomorphism of partial groups} if $\D\beta^*\subseteq\D'$ and $\Pi(w)\beta=\Pi'(w\beta^*)$ for all $w\in\D$. We call $\beta$ an \emph{isomorphism of partial groups} if in addition $\beta$ is bijective and $\D\beta^*=\D'$. 
\end{definition}

\subsection{Localities} We are now in a position to define localities. Our definition follows the one given in \cite[Definition~5.1]{Henke:2015}, which is by \cite[Remark~5.2]{Henke:2015} equivalent to the definition of a locality given in \cite[Definition~2.9]{Chermak:2013} and \cite[Definition~2.7]{Chermak:2022}.

\begin{definition}\label{LocalityDefinition}
We say that $(\L,\Delta,S)$ is a \textbf{locality} if $\L$ is a partial group which is finite as a set, $S$ is a $p$-subgroup of $\L$, $\Delta$ is a non-empty set of subgroups of $S$, and the following conditions hold:
\begin{itemize}
\item[(L1)] $S$ is maximal with respect to inclusion among the $p$-subgroups of $\L$.
\item[(L2)] A word $(f_1,\dots,f_n)\in\W(\L)$ is an element of $\D$ if and only if there exist $P_0,\dots,P_n\in\Delta$ such that 
\begin{itemize}
\item [(*)] $P_{i-1}\subseteq \D(f_i)$ and $P_{i-1}^{f_i}=P_i$.
\end{itemize}
\item[(L3)] The set $\Delta$ is closed under passing to $\L$-conjugates and overgroups in $S$.
\end{itemize}
If $(\L,\Delta,S)$ is a locality and $v=(f_1,\dots,f_n)\in\W(\L)$, then we say that $v\in\D$ via $P_0,\dots,P_n$ (or just $v\in\D$ via $P_0$), if $P_0,\dots,P_n\in\Delta$ and (*) holds.

\smallskip

The set $\Delta$ is called the \emph{set of objects} of the locality $(\L,\Delta,S)$.
\end{definition}

Note that, as $\Delta\neq\emptyset$, property (L3) implies in particular $S\in\Delta$. The condition in (L3) that $\Delta$ is \emph{closed under passing to $\L$-conjugates in $S$} means more precisely that, given an element $P\in\Delta$, we have $P^g\in\Delta$ for every $g\in\L$ with $P\subseteq \D(g)$ and $P^g\subseteq S$. Important examples for localities come from finite groups as follows:

\begin{eg}\label{E:LDeltaG}
Let $G$ be a finite group and $S\in\Syl_p(G)$. Suppose $\Delta$ is a non-empty set of subgroups of $S$ such that $\Delta$ is closed under passing to $G$-conjugates and overgroups in $S$. Set
$$\L_\Delta(G):=\{g\in G\colon \exists P\in\Delta\mbox{ such that }P^g\leq S\}.$$
Let $\D$ be the set of words $(f_1,\dots,f_n)\in\W(\L)$ such that there exist $P_0,\dots,P_n\in\Delta$ with $P_{i-1}^{f_i}=P_i$ for $i=1,\dots,n$. Let $\Pi$ be the restriction of the multivariable product on $G$ to $\D$. As an inversion map take the map $\L_\Delta(G)\rightarrow \L_\Delta(G)$ sending every element of $\L_\Delta(G)$ to its inverse in $G$. Then $\L_\Delta(G)$ forms with this product and inversion a partial group, and $(\L_\Delta(G),\Delta,S)$ is a locality.
\end{eg}

Suppose $(\L,\Delta,S)$ is a locality. For any $g\in\L$, write $c_g$ for the conjugation map
$$c_g\colon \D(g)\rightarrow \L,\;x\mapsto x^g.$$
For $g\in\L$ set furthermore 
$$S_g:=\{x\in S\colon x\in\D(g)\mbox{ and }x^g\in S\}.$$ 
For any subset $X$ of $\L$ set
$$N_\L(X):=\{f\in\L\colon X\subseteq\D(f),\;X^f=X\}.$$
We summarize the most important properties of localities in the following lemma.

\begin{lemma}\label{LocalitiesProp}
The following hold:
\begin{itemize}
\item [(a)] $N_\L(P)$ is a subgroup of $\L$ for each $P\in\Delta$.
\item [(b)] Let $P\in\Delta$ and $g\in\L$ with $P\subseteq S_g$. Then $Q:=P^g\in\Delta$, $N_\L(P)\subseteq \D(g)$ and 
$$c_g\colon N_\L(P)\rightarrow N_\L(Q)$$
is an isomorphism of groups.
\item [(c)] Let $w=(g_1,\dots,g_n)\in\D$ via $(X_0,\dots,X_n)$. Then 
$$c_{g_1} \dots c_{g_n}=c_{\Pi(w)}$$
is a group isomorphism $N_\L(X_0)\rightarrow N_\L(X_n)$.
\item [(d)] For every $g\in\L$, $S_g\in\Delta$. In particular, $S_g$ is a subgroup of $S$ and $c_g\colon S_g\rightarrow S_g^g=S_{g^{-1}}$ is a group isomorphism.
\item [(e)] For every $R\leq S$, the set $N_\L(R)$ forms a partial subgroup of $\L$.
\end{itemize}
\end{lemma}

\begin{proof}
Properties (a)-(c) correspond to the statements in \cite[Lemma~2.3]{Chermak:2022} except for the fact stated in (b) that $Q\in\Delta$, which is true by property (L3) in our definition of a locality.  Property (d) is true by Lemma~1.6(c) and Proposition~2.5(a),(b) in \cite{Chermak:2022}, and property (e) is stated in \cite[Lemma~5.5]{Henke:2015}.
\end{proof}

Given a locality $(\L,\Delta,S)$, Lemma~\ref{LocalitiesProp}(d) allows us to construct the fusion system $\F_S(\L)$, which is the fusion system over $S$ generated by the maps $c_g\colon S_g\rightarrow S$ with $g\in\L$. Given a fusion system $\F$, we say that $(\L,\Delta,S)$ is a fusion system \emph{over} $\F$ if $\F=\F_S(\L)$.  

\smallskip

If $\H$ is a partial subgroup of $\L$, then we can more generally construct the fusion system $\F_{S\cap \H}(\H)$ over $S\cap \H$, which is generated by all the conjugation maps $c_h\colon S_h\cap \H\rightarrow S\cap \H,x\mapsto x^h$ with $h\in \H$. In particular, for every $P\leq S$, Lemma~\ref{LocalitiesProp}(e) allows us to consider $\F_{N_S(P)}(N_\L(P))$.

\smallskip

If $(\L,\Delta,S)$ is a locality, then define a subgroup $P\leq S$ to be \textit{weakly closed} in $(\L,\Delta,S)$ if $P^f=P$ for any $f\in\L$ with $P\leq S_f$, or equivalently, if $P$ is weakly closed in $\F_S(\L)$.

\begin{lemma}\label{L:NormalizerLocFus}
 Let $(\L,\Delta,S)$ be a locality over a fusion system $\F$. Then $\F_{N_S(P)}(N_\L(P))\subseteq N_\F(P)$ for every $P\leq S$. If $P\in\Delta$ or $P$ is weakly closed in $(\L,\Delta,S)$, then $\F_{N_S(P)}(N_\L(P))=N_\F(P)$.
\end{lemma}

\begin{proof}
One sees easily that $\F_{N_S(P)}(N_\L(P))\subseteq N_\F(P)$ for every $P\leq S$. Suppose now $P\in\Delta$ or $P$ is weakly closed in $(\L,\Delta,S)$. Given $R\leq N_S(P)$ and a morphism $\phi\in\Hom_{N_\F(P)}(R,N_S(P))$, we need to show that $\phi$ is a morphism in $\F_{N_S(P)}(N_\L(P))$. Without loss of generality, we may assume that $P\leq R$ and so $P\phi=P$. We can now write $\phi$ as a composition of restrictions of conjugation maps by elements $g_1,\dots,g_n\in\L$. 

\smallskip

If $P$ is weakly closed, it follows inductively that $P^{g_i}=P$ for $i=1,2,\dots,n$ and so $g_1,\dots,g_n\in N_\L(P)$, proving that $\phi$ is a morphism in $\F_{N_S(P)}(N_\L(P))$.

\smallskip

Suppose now $P\in\Delta$. As $P\in\Delta$, it follows from the definition of a locality that $(g_1,\dots,g_n)\in\D$ via $P$. Now by Lemma~\ref{LocalitiesProp}(c), $\phi=c_{\Pi(g_1,\dots,g_n)}|_R$. In particular, as $P\phi=P$, we have $\Pi(g_1,\dots,g_n)\in N_\L(P)$. Hence, $\phi$ is a morphism in $\F_{N_S(P)}(N_\L(P))$.  
\end{proof}

For any locality $(\L,\Delta,S)$, we define
\[O_p(\L):=\bigcap_{w\in\W(\L)}S_w.\]
It turns out that $O_p(\L)$ is the largest subgroup of $S$ which is a partial normal subgroup of $\L$, and also the maximal element of the set $\{P\leq S\colon \L=N_\L(P)\}$ (cf. \cite[Lemma~2.13]{Chermak:2022} and \cite[Lemma~3.13]{Henke:Regular}). 

\smallskip

We will need the following weak version of Alperin's Fusion Theorem for localities.
 
\begin{prop}\label{AlperinLocalities}
Let $(\L,\Delta,S)$ be a locality and let $g\in\L$. Then there exist $n\in\mathbb{N}$,  subgroups $R_1,\dots,R_n\in\Delta$ and $g_i\in N_\L(R_i)$ such that 
\begin{itemize}
\item $N_S(R_i)\in\Syl_p(N_\L(R_i))$, 
\item $(g_1,\dots,g_n)\in\D$ and $g=g_1 g_2\dots g_n$,
\item $S_g\leq S_{(g_1,\dots,g_n)}\cap R_1$ and $S_g^{g_1\dots g_{i}}\leq R_{i+1}$ $(1\leq i< n)$.
\end{itemize}
\end{prop}

\begin{proof}
 Let $\L_0$ be the set of all $g\in\L$ for which the assertion holds. Using Lemma~\ref{LocalitiesProp}(c) one observes that the following property holds:
\begin{itemize}
 \item [(*)] Let $f_1,\dots,f_m\in\L_0$ such that $(f_1,\dots,f_m)\in\D$ and $S_{\Pi(f_1,\dots,f_m)}\leq S_{(f_1,\dots,f_m)}$. Then $\Pi(f_1,\dots,f_m)\in\L_0$.
\end{itemize}
Assume now $\L_0\neq \L$ and let $f\in\L\backslash\L_0$ such that $|S_f|$ is maximal. Observe that $N_\L(S)\subseteq \L_0$ since $S\in\Delta$ with $S\in\Syl_p(N_\L(S))$. So $P:=S_f<S$ as otherwise $f\in N_\L(S)\subseteq \L_0$. By \cite[Lemma~2.9]{Chermak:2022}, there exists $g\in\L$ such that $N_S(P^f)\leq S_g$ and $N_S(P^{fg})\in\Syl_p(N_\L(P^{fg}))$. It follows from Lemma~\ref{LocalitiesProp}(c) that conjugation by $fg$ is defined on $N_\L(P)$, and that the conjugation map is an isomorphism of groups $N_\L(P)\rightarrow N_\L(P^{fg})$. In particular, $N_S(P)^{fg}$ is a $p$-subgroup of $N_\L(P^{fg})$. Thus, by Sylow's theorem, there exists $h\in N_\L(P^{fg})$ such that $N_S(P)^{fgh}\leq N_S(P^{fg})$. Notice that $(f,g,h,h^{-1},g^{-1})\in\D$ via $P$. Thus, $(fgh,h^{-1},g^{-1})\in\D$ and  $S_{\Pi(fgh,h^{-1},g^{-1})}=S_f=P\leq S_{(fgh,h^{-1},g^{-1})}$. Moreover, $|S_{fgh}|\geq |N_S(P)|>|P|$ and so $fgh\in\L_0$ by the maximality of $|S_f|=|P|$. Using Lemma~\ref{LocalitiesProp}(d), we get moreover $|S_{g^{-1}}|=|S_g^g|=|S_g|\geq |N_S(P^f)|>|P^f|=|P|$ and again by the maximality of $|P|$, $g^{-1}\in\L_0$. Clearly, $h^{-1}\in N_\L(P^{fg})\subseteq \L_0$ by definition of $\L_0$. Hence, it follows from (*) that $f=\Pi(fgh,h^{-1},g^{-1})\in\L_0$, a contradiction. 
\end{proof}

\subsection{Linking localities and localities of objective characteristic $p$}
If $\F$ is a saturated fusion system, it can be shown that there exist localities with nice properties over $\F$. To make this more precise, we need the following definition.

\begin{definition}\label{D:LinkingLoc}~
\begin{itemize}
\item A finite group $G$ is said to be of \textit{characteristic $p$} if $C_G(O_p(G))\leq O_p(G)$.
\item Define a locality $(\L,\Delta,S)$ to be of \textit{objective characteristic $p$} if, for any $P\in\Delta$, the group $N_\L(P)$ is of characteristic $p$. 
\item A locality $(\L,\Delta,S)$ is called a \textit{linking locality}, if $\F_S(\L)$ is saturated, $\F_S(\L)^{cr}\subseteq \Delta$ and $(\L,\Delta,S)$ is of objective characteristic $p$.
\item Let $\F$ be a fusion system over $S$. A subgroup $P\leq S$ is called $\F$-subcentric, if for some fully $\F$-normalized $\F$-conjugate $Q$ of $P$, we have $O_p(N_\F(Q))\in\F^c$. Write $\F^s$ for the set of $\F$-subcentric subgroups of $S$.
\item If $\F$ is saturated, then a \textit{subcentric locality} for $\F$ is a linking locality $(\L,\F^s,S)$ for $\F$.
\item (Cf. {\cite[Definition~3.5]{Chermak:2013}}) Let $(\L,\Delta,S)$ and $(\L',\Delta,S)$ be localities having the same set of objects. Then an isomorphism of partial groups $\beta\colon\L\rightarrow\L'$ is called a \emph{rigid isomorphism} if $\beta$ restricts to the identity on $S$. 
\end{itemize}
\end{definition}

If $\F$ is a fusion system over $S$ and $\Delta$ is a set of subgroups of $S$, then the existence of a locality  $(\L,\Delta,S)$ over $\F$ implies (according to (L3)) that the set $\Delta$ is closed under passing to $\F$-conjugates and overgroups in $S$. If $(\L,\Delta,S)$ is a linking locality over $\F$, then we have in addition that $\F^{cr}\subseteq\Delta\subseteq \F^s$, where the latter inclusion was shown in \cite[Lemma~6.1]{Henke:2015}.  The first part of the following theorem states that such conditions on $\Delta$ are actually not only necessary, but also sufficient for a linking locality to exist.

\begin{theorem}\label{T:LinkingLocalityExistence}
Let $\F$ be a saturated fusion system over $S$. 
\begin{itemize}
\item If $\Delta$ is a collection of subgroups of $S$ such that $\F^{cr}\subseteq\Delta\subseteq \F^s$ and $\Delta$ is closed under passing to $\F$-conjugates and overgroups in $S$, then there exists a linking locality $(\L,\Delta,S)$ over $\F$, which is unique up to rigid isomorphism. 
\item The set $\F^s$ contains $\F^{cr}$ and is closed under passing to $\F$-conjugates and overgroups in $S$. So in particular, there exists a subcentric locality over $\F$, and such a subcentric locality is unique up to rigid isomorphism.
\end{itemize}
\end{theorem}

\begin{proof}
For $\Delta=\F^c$ this was shown in \cite{Chermak:2013} (where the existence and uniqueness of centric linking systems is shown); a proof which does not rely on the classification of finite simple groups can be given through \cite{Oliver:2013} and \cite{GlaubermanLynd2016}. Using that the theorem is true for $\Delta=\F^c$, a proof of the general statement is given in \cite[Theorem~A]{Henke:2015}.  
\end{proof}


\begin{lemma}\label{lemma:parabolics in local of char p}
	Let $(\L,\Delta,S)$ be a locality of objective characteristic $p$ and $M \le \L$ be a subgroup of $\L$. If $S \le M$, then $M$ has characteristic $p$.
\end{lemma}

\begin{proof}
	By \cite[Proposition 2.10]{Chermak:2022} there exists $P \in \Delta$ such that $M \le N_\L(P)$. As $O_p(N_\L(P))\leq S\leq M$, we have $O_p(N_\L(P)) \le O_p(M)$. Then
	\[
	C_M(O_p(M)) \le C_M(O_p(N_\L(P))) \le O_p(N_\L(P)) \le O_p(M).
	\]
\end{proof}

\subsection{$p$-Residuals of localities}

\begin{definition}\label{D:pResidual}
 Let $(\L,\Delta,S)$ be a locality. We say that a partial normal subgroup $\K$ of $\L$ has $p$-power index in $\L$, if $\K S=\L$.  Define
 \[O^p(\L):=\bigcap\{\K\unlhd \L\colon \K S=\L\}.\]
\end{definition}

\begin{lemma}\label{L:pResidualLoc}
 Let $(\L,\Delta,S)$ be a locality and $\K\unlhd \L$. Then $\K$ has $p$-power index in $\L$ if and only if $O^p(N_\L(P))\subseteq \K$ for every $P\in\Delta$. In particular, $O^p(\L)$ has $p$-power index in $\L$.
\end{lemma}

\begin{proof}
This is a special case of \cite[Lemma~7.2]{Henke:Regular}.
\end{proof}

If $\F$ is a saturated fusion system over $S$, then following \cite[Definition~I.7.3]{Aschbacher/Kessar/Oliver:2011}, we say that a subsystem $\E$ of $\F$ over a subgroup $T$ of $S$ is of \textit{$p$-power index}, if $O^p(\Aut_\F(P))\leq \Aut_\E(P)$ for all $P\leq T$, and $T$ contains the \textit{hyperfocal subgroup} 
\[\hyp(\F):=\<[Q,O^p(\Aut_\F(Q))]\colon Q\leq S\>.\]
By \cite[Theorem~I.7.4]{Aschbacher/Kessar/Oliver:2011}, there exists a unique saturated subsystem of $\F$ over $\hyp(\F)$ of $p$-power index. It is denoted by $O^p(\F)$ and turns out to be normal in $\F$. 

\begin{prop}\label{P:pResidualLocFus}
 Let $(\L,\Delta,S)$ be a linking locality and $\F=\F_S(\L)$. Then $\hyp(\F)=O^p(\L)\cap S$ and $O^p(\F)$ is the smallest normal subsystem of $\F$ containing $\F_{S\cap O^p(\L)}(O^p(\L))$. 
\end{prop}

\begin{proof}
This follows combining Theorem~A(a) and Theorem~E(b) in \cite{Chermak/Henke}.
\end{proof}

\section{Large subgroups of groups, fusion systems and localities}


\subsection{Large subgroups of groups}

Let $p$ always be a prime. Recall the definition of a large subgroup from Definition~\ref{D:LargeGroup} in the introduction.

\begin{definition}
A finite group $G$ is said to be of \textbf{parabolic characteristic $p$} if every $p$-local subgroup of $G$, which contains a Sylow $p$-subgroup of $G$, is of characteristic $p$.

\smallskip

A $p$-subgroup $Q$ of $G$ is called \emph{weakly closed} in $G$ if every Sylow $p$-subgroup of $G$ contains exactly one $G$-conjugate of $Q$.
\end{definition}

In the following lemma we list some important properties of large subgroups.

\begin{lemma}\label{L:LargeSubgroupsOfGroups}
 Let $G$ be a finite group and $Q$ be a large $p$-subgroup of $G$. Then the following hold:
\begin{itemize}
 \item [(a)] $O_p(N_G(Q))$ is large in $G$.
 \item [(b)] $Q$ is weakly closed in $G$.
 \item [(c)] If $R$ is a large subgroup of $G$ normalizing $Q$, then $QR$ is large in $G$. 
 \item [(d)] If $P$ is a non-trivial $p$-subgroup of $G$ normalized by $Q$, then $N_G(P)$ has characteristic $p$. In particular, $G$ has parabolic characteristic $p$.
\end{itemize}
\end{lemma}

\begin{proof}
 \textbf{(a,b)} Property (a) is easy to check and stated in \cite[Lemma~1.52(e)]{MSS}, and property (b) is \cite[Lemma~1.52(b)]{MSS}. 
 
 \smallskip
 
\textbf{(c)} Let $R\leq N_G(Q)$ be a large subgroup of $G$. Then $QR$ is a $p$-subgroup of $G$, which is large in $G$ since $Q$ and $R$ are large in $G$, $C_S(QR)=C_S(Q)\cap C_S(R)=Z(Q)\cap Z(R)=Z(QR)$ and $N_G(QR)=N_G(Q)\cap N_G(R)$.
 
\smallskip
 
\textbf{(d)} Property (d) is a reformulation of \cite[Lemma~1.55(a),(c)]{MSS}. We give however a new proof of (d) here which fits better into our framework: 

\smallskip

If $P$ is as in (d) then $PQ$ is a $p$-group and thus there exists a Sylow $p$-subgroup $S$ of $G$ with $PQ\leq S$. The key property we use is that the set
$$\Delta(S,G):=\{X\leq S\colon N_G(X)\mbox{ has characteristic $p$}\}$$
is by \cite[Lemma~9.1]{Henke:2015} closed under taking overgroups. As $P$ is a non-trivial $p$-group which is normalized by $Q$, we have $1\neq U:=C_P(Q)\leq C_G(Q)=Z(Q)$. Thus, the property (\ref{Q!}) yields $N_G(U)\leq N_G(Q)$ and thus $Q\unlhd N_G(U)$. Since $C_G(Q)\leq Q$, it follows that $N_G(U)$ has characteristic $p$, i.e. $U\in\Delta(S,G)$. As $\Delta(S,G)$ is closed under passing to overgroups in $S$, this implies $P\in\Delta(S,G)$ as required.   
\end{proof}

\subsection{Large subgroups of fusion systems} The reader might want to recall the definition of a large subgroup of a fusion system from Definition~\ref{D:LargeFusLoc} in the introduction. If $G$ is a finite group with Sylow $p$-subgroup $S$ and $\F=\F_S(G)$, then $N_\F(P)=\F_{N_S(P)}(N_G(P))$ for every $P\leq S$. Moreover, factoring out a normal $p^\prime$-subgroup of $G$ does not change the fusion system of $G$ up to isomorphism. Using these properties one sees easily that we have the following example.

\begin{eg}\label{E:LargeGLargeFSG}
 Let $G$ be a finite group, $S\in\Syl_p(G)$ and $Q\leq S$. If $Q$ is large in $G$, then $Q$ is large in $\F_S(G)$. More generally, if $N$ is a normal $p^\prime$-subgroup of $G$ and $Q$ is large in $G/N$, then $Q$ is large in $\F_S(G)\cong \F_{SN/N}(G/N)$. 
\end{eg}

We will now consider properties of large subgroups of abstract fusion systems.

\smallskip

\textbf{For the remainder of this section let $\F$ be a (not necessarily saturated) fusion system over a finite $p$-group $S$.}

\smallskip

We will establish some results on large subgroups of fusion systems which resemble the ones listed in Lemma~\ref{L:LargeSubgroupsOfGroups}. Moreover, we give in Lemma~\ref{L:LargeFusEquiv} an equivalent characterization of large subgroups of fusion systems.

\begin{lemma}\label{L:LargeFusMain}
Let $Q\leq S$ be a large subgroup of $\F$.
 \begin{itemize}
 \item [(a)] The subgroup $O_p(N_\F(Q))$ is large in $\F$.
 \item [(b)] For every $1\neq U\leq Z(Q)$, we have $Q\unlhd N_\F(U)$.
 \item [(c)] We have $Q\unlhd S$.
 \item [(d)] If $R\leq S$ is large in $\F$, then $QR$ is large in $\F$.
 \end{itemize}
\end{lemma}

\begin{proof}
\textbf{(a)} We have $Q\leq Q^\bullet:=O_p(N_\F(Q))$. In particular, $C_S(Q^\bullet)\leq C_S(Q)\leq Q\leq Q^\bullet$ and $Z(Q^\bullet)\leq Z(Q)$. Moreover, $N_\F(Q)\subseteq N_\F(Q^\bullet)$ by definition of $Q^\bullet$. This implies (a).

\smallskip

\textbf{(b)} Let $1\neq U\leq Z(Q)$. Let $X,Y\leq N_S(U)$ and $\phi\in\Hom_{N_\F(U)}(X,Y)$. We need to show that $\phi$ extends to a morphism in $N_\F(U)$ which is defined on $Q$ and leaves $Q$ invariant. By definition of $N_\F(U)$, we may assume that $U\leq X$ and $U\phi=U$. By \eqref{Q!F}, we have $N_\F(U)\subseteq N_\F(Q)$, i.e. $\phi$ extends to a morphism $\hat{\phi}\in\Hom_\F(XQ,YQ)$ with $Q\hat{\phi}=Q$. As $U\hat{\phi}=U\phi=U$, the map $\hat{\phi}$ is a morphism in $N_\F(U)$. This proves the assertion.

\smallskip

\textbf{(c)} Note that $1\neq Z(S)\leq C_S(Q)=Z(Q)$. Hence, $N_\F(Z(S))\subseteq N_\F(Q)$. As $N_\F(Z(S))$ is a subsystem on $S$, this implies $Q\unlhd S$.

\smallskip

\textbf{(d)} Let $R\leq S$ be large in $\F$. Then $C_S(QR)=C_S(Q)\cap C_S(R)=Z(Q)\cap Z(R)$ and so in particular, $Z(QR)=Z(Q)\cap Z(R)$. Hence, for $1\neq U\leq Z(QR)$, it follows from (b) that $Q$ and $R$ are normal in $N_\F(U)$. Thus, $QR$ is normal in $N_\F(U)$. In particular, the property \eqref{Q!F} holds for $QR$ in place of $Q$. 
\end{proof}

\begin{lemma}\label{L:LargeFusEquiv}
Let $\F$ be saturated and $Q\leq S$. Then $Q$ is large in $\F$ if and only if $C_S(Q)\leq Q$ and the following condition holds:
\begin{equation}\label{Q!Ff}\tag{$Q!\F^f$}
N_\F(U)\subseteq N_\F(Q)\mbox{ for all fully $\F$-normalized }U\leq Z(Q)\mbox{ with }U\neq 1
\end{equation} 
Moreover, if so, then $Q$ is weakly closed in $\F$.
\end{lemma}

\begin{proof}
Note that property \eqref{Q!F} implies \eqref{Q!Ff}. Moreover, if $Q$ is weakly closed, then it follows from Lemma~\ref{L:ConjFulNormZQ} that \eqref{Q!Ff} implies \eqref{Q!F}. Hence, assuming from now on that $C_S(Q)\leq Q$ and that the property \eqref{Q!Ff} holds, we only need to show that $Q$ is weakly closed in $\F$. 
 
\smallskip

It is a consequence of Alperin's fusion theorem (see e.g. \cite[Theorem~I.3.6]{Aschbacher/Kessar/Oliver:2011}), \cite[Proposition~2.10]{Diaz/Glesser/Mazza/Park:2009} and Lemma~\ref{L:ZRfullynorm} that every morphism in $\F$ can be written as a composition of restrictions of $\F$-automorphisms of subgroups $R$ of $S$ such that $Z(R)$ is fully $\F$-normalized. Thus, for $Q\leq R\leq S$ with $Z(R)$ fully $\F$-normalized, we need to argue that $Q$ is $\Aut_\F(R)$-invariant. As $C_S(Q)\leq Q$, we have $1\neq Q\leq R$ and $1\neq Z(R)\leq Z(Q)$. So the property \eqref{Q!Ff} yields $N_\F(R)\subseteq N_\F(Z(R))\subseteq N_\F(Q)$ and $Q$ is $\Aut_\F(R)$-invariant as required.
\end{proof}

If the fusion system $\F$ is saturated, then it is said to be of \emph{parabolic characteristic $p$} if $N_\F(P)$ is constrained for every non-trivial normal subgroup $P$ of $S$.

\smallskip

For a saturated fusion system $\F$, a fully $\F$-normalized subgroup $P\leq S$ is by \cite[Lemma~3.1]{Henke:2015} subcentric if and only if $N_\F(P)$ is constrained. In particular, $\F$ is of parabolic characteristic $p$ if and only if $P\in\F^s$ for every non-trivial normal subgroup $P$ of $S$.

\begin{lemma}\label{L:LargeSatFusMain}
 Let $\F$ be saturated and let $Q\leq S$ be large in $\F$.
 \begin{itemize}
\item [(a)] The subgroup $Q$ is weakly closed in $\F$. In particular, we have $Q\in\F^c$.
\item [(b)] We have $P\in\F^s$ for every $1\neq P\leq S$ with $Q\leq N_S(P)$. In particular, $\F$ is of parabolic characteristic $p$.
 \end{itemize}
\end{lemma}

\begin{proof}
 \textbf{(a)} That $Q$ is weakly closed in $\F$ is just a restatement of Lemma~\ref{L:LargeFusEquiv}. In particular,  $Q\in\F^c$ as $C_S(Q)\leq Q$.
 
\smallskip

\textbf{(b)} Let $1\neq P\leq S$ with $Q\leq N_S(P)$. Set $U:=C_P(Q)$ and note that $1\neq U\leq C_S(Q)=Z(Q)$. Then by (a) and Lemma~\ref{L:ConjFulNormZQ}, a fully $\F$-normalized $\F$-conjugate $V$ of $U$ is also contained in $Z(Q)$. Note that $C_S(Q)\leq Q$ and $Q\unlhd N_\F(V)$ by Lemma~\ref{L:LargeFusMain}(b). This implies that $N_\F(V)$ is constrained and so $U\in\F^s$ by \cite[Lemma~3.1]{Henke:2015}. As the set $\F^s$ is by \cite[Proposition~3.3]{Henke:2015} closed under passing to overgroups in $S$, it follows $P\in\F^s$ as required.
\end{proof}

\subsection{Large subgroups of localities}

The reader might want to recall the definition of a large subgroup of a locality from Definition~\ref{D:LargeFusLoc} in the introduction. 

\begin{eg}\label{E:LargeGLargeLDeltaG}
 Let $G$ be a finite group, $S\in\Syl_p(G)$ and let $\Delta$ be a set of subgroups of $S$ which is closed under passing to $G$-conjugates and overgroups in $S$. Consider the locality $\L_\Delta(G)$ constructed in Example~\ref{E:LDeltaG}. If $Q\leq S$ is a large subgroup of $G$, then $Q$ is large in $\L_\Delta(G)$.
\end{eg}

\begin{lemma}\label{L:LargeLocMain}
 Let $(\L,\Delta,S)$ be a locality over $\F$ and let $Q\leq S$ be a large subgroup of $\L$. Then the following hold.
 \begin{itemize}
  \item [(a)] $Q$ is weakly closed in $\L$. In particular, $Q\unlhd S$, the triple $(N_\L(Q),\Delta,S)$ is a locality over $N_\F(Q)$, and $O_p(N_\L(Q))$ is well-defined.
  \item [(b)] $O_p(N_\L(Q))$ is large in $(\L,\Delta,S)$.
  \item [(c)] If $R\leq S$ is large in $\L$, then $QR$ is large in $\L$.
 \end{itemize}
\end{lemma}

\begin{proof}
\textbf{(a)} Let $Q\leq P\in\Delta$. To prove that $Q$ is weakly closed in $\L$, it is by Proposition~\ref{AlperinLocalities} sufficient to see that $N_\L(P)\subseteq N_\L(Q)$. As $C_S(Q)\leq Q$, we have $P\neq 1$ and $1\neq Z(P)\leq C_S(Q)=Z(Q)$. Hence, by \eqref{Q!L}, $N_\L(P)\subseteq N_\L(Z(P))\subseteq N_\L(Q)$. This shows that $Q$ is weakly closed in $\L$. 

\smallskip

We have shown in particular that $Q\unlhd S$. So by \cite[Lemma~2.12]{Chermak:2022}, the triple $(N_\L(Q),\Delta,S)$ is a locality and $O_p(N_\L(Q))$ is well-defined. As $Q$ is weakly closed, it follows moreover from Lemma~\ref{L:NormalizerLocFus} that $\F_S(N_\L(Q))=N_\F(Q)$. 

\smallskip

\textbf{(b)} Set now $Q^\bullet:=O_p(N_\L(Q))\leq S$. As $Q\leq Q^\bullet$, we have $C_\L(Q^\bullet)\subseteq C_\L(Q)\subseteq Q\leq Q^\bullet$ and $Z(Q^\bullet)\leq Z(Q)$. As $N_\L(Q)\subseteq N_\L(Q^\bullet)$, the property \eqref{Q!L} holds thus for $Q^\bullet$ as it holds for $Q$. 

\smallskip

\textbf{(c)} Let $R \le S$ be large in $\L$. It follows from (a) that $QR$ is a subgroup of $S$. Note that $C_\L(QR) \le C_\L(Q) \cap C_\L(R) = Z(Q) \cap Z(R)=Z(QR)$. If $1 \ne U \le Z(QR)$, then \eqref{Q!L} applied to $Q$ and $R$ yields
\[
N_\L(U) \subseteq N_\L(Q) \cap N_\L(R) \le N_\L(QR),
\]
where the last inclusion follows from Lemma~\ref{LocalitiesProp}(d).
\end{proof}

\subsection{Connections between large subgroups of fusion systems and of localities}

In this section we want to obtain connections between large subgroups of fusion systems and large subgroups of associated localities. If $Q$ is a large subgroup of a given locality $(\L,\Delta,S)$ or of the fusion system $\F_S(\L)$, then in view of the properties \eqref{Q!F} and \eqref{Q!L}, we want to ensure for many situations that every non-trivial subgroup of $C_S(Q)=Z(Q)$ is an object. This is important as it allows us to apply Lemma~\ref{L:NormalizerLocFus}. The following lemma is interesting in this context.

\begin{lemma}\label{L:QReplete}
 Let $(\L,\Delta,S)$ be a locality and $Q\leq S$. Then the following two conditions are equivalent:
 \begin{itemize}
\item [(i)] Every non-trivial subgroup $P$ of $S$ with $Q\leq N_S(P)$ is an element of $\Delta$.
\item [(ii)] Every non-trivial subgroup of $C_S(Q)$ is an element of $\Delta$. 
\end{itemize}
\end{lemma}

\begin{proof}
As every subgroup of $C_S(Q)$ is normalized by $Q$, property (i) implies property (ii). 

\smallskip

Assume that (ii) holds. To verify (i), fix a non-trivial subgroup $P$ of $S$ with $Q\leq N_S(P)$. Then $1\neq U:=C_P(Q)\leq C_S(Q)$. Hence, $U\in\Delta$ by assumption, and so $P\in\Delta$, since $\Delta$ is closed under passing to overgroups in $S$. This proves that (ii) implies (i).
\end{proof}

\begin{definition}
 Let $(\L,\Delta,S)$ be a locality and $Q\leq S$. We call $(\L,\Delta,S$) \textit{$Q$-replete} if one and thus each of the two equivalent conditions (i) and (ii) in Lemma~\ref{L:QReplete} holds.
\end{definition}

\begin{eg}\label{E:SubcentricQReplete}
If $\F$ is saturated, $Q\leq S$ is large in $\F$, and $(\L^s,\F^s,S)$ is a subcentric locality over $\F$, then $(\L^s,\F^s,S)$ is $Q$-replete by Lemma~\ref{L:LargeSatFusMain}(b). We show in Lemma~\ref{L:LargeLF}(b) below that $Q$ is large in $(\L^s,\F^s,S)$.
\end{eg}

\begin{lemma}\label{L:QRepleteSaturated}
Let $(\L,\Delta,S)$ be a locality over $\F$ which is $Q$-replete for some subgroup $Q\unlhd S$. Then $\F^c\subseteq\Delta$ and $\F$ is saturated.
\end{lemma}

\begin{proof}
Let $P\in\F^c$. As $Q\unlhd S$, the $p$-group $P$ acts on the $p$-group $Z(Q)\neq 1$ and so $X:=C_{Z(Q)}(P)\neq 1$. As $X$ is normalized by $Q$ and $(\L,\Delta,S)$ is $Q$-replete, it follows that $X\in\Delta$. As $P\in\F^c$, we have $X\leq C_S(P)\leq P$. Since $\Delta$ is closed under passing to overgroups, it follows that $P\in\Delta$. This shows $\F^c\subseteq\Delta$. In particular, by  \cite[Proposition 2.17(a)]{Chermak:2013}, $\F$ is saturated. 
\end{proof}

\begin{eg}\label{E:DeltaSGLarge}
Let $G$ be a finite group, $S\in\Syl_p(G)$ and 
\[\Delta:=\{P\leq S\colon N_G(P)\mbox{ is of characteristic }p\}.\]
By \cite[Lemma~10.2]{Henke:2015}, the set $\Delta$ is closed under passing to $G$-conjugates and overgroups in $S$. Hence, we have the locality $(\L_\Delta(G),\Delta,S)$ defined as in Example~\ref{E:LDeltaG}. By definition of $\Delta$, this is a locality of objective characteristic $p$.

\smallskip

Suppose now $Q\leq S$ is large in $G$. Then $Q$ is large in $\L_\Delta(G)$ as mentioned before in Example~\ref{E:LargeGLargeLDeltaG}. Moreover, $(\L_\Delta(G),\Delta,S)$ is a $Q$-replete linking locality over $\F_S(G)$.
\end{eg}

\begin{proof} We show that $(\L_\Delta(G),\Delta,S)$ is a $Q$-replete linking locality over $\F_S(G)$ with $G,S,\Delta,Q$ as in Example~\ref{E:DeltaSGLarge}. 

\smallskip

It follows from Lemma~\ref{L:LargeSubgroupsOfGroups}(b) that $Q\unlhd S$ and from Lemma~\ref{L:LargeSubgroupsOfGroups}(d) that $(\L_\Delta(G),\Delta,S)$ is $Q$-replete. Hence, $\F_S(\L_\Delta(G))^c\subseteq\Delta$ and $\F_S(\L_\Delta(G))$ is saturated by Lemma~\ref{L:QRepleteSaturated}. 

\smallskip

As $\F_S(\L_\Delta(G))\subseteq\F_S(G)$, we have $\F_S(G)^c\subseteq \F_S(\L_\Delta(G))^c\subseteq\Delta$. By \cite[Theorem~I.2.3]{Aschbacher/Kessar/Oliver:2011}, $\F_S(G)$ is saturated, and so by Alperin's fusion theorem for fusion systems (see e.g. \cite[Theorem~I.3.5]{Aschbacher/Kessar/Oliver:2011}), $\F_S(G)$ is generated by the automorphism groups $\Aut_G(P)$, where $P$ runs over all $P\in\F_S(G)^c$. As $P\in\Delta$ for $P\in\F_S(G)^c$, we have however $N_G(P)=N_{\L_\Delta(G)}(P)$ and $\Aut_G(P)\subseteq \F_S(\L_\Delta(G))$. Hence, $\F_S(G)=\F_S(\L_\Delta(G))$. In particular, as $\F_S(G)$ is saturated and $\F_S(G)^{cr}\subseteq\F_S(G)^c\subseteq\Delta$, the locality $(\L_\Delta(G),\Delta,S)$ is a linking locality over $\F_S(G)$. 
\end{proof}

\begin{lemma}\label{L:LargeLF}
Let $(\L,\Delta,S)$ be a locality over $\F$ and $Q\leq S$. 
\begin{itemize}
 \item [(a)] Assume that $Q$ is large in $(\L,\Delta,S)$ and $(\L,\Delta,S)$ is $Q$-replete. Then $Q$ is large in $\F$, we have $\F^c\subseteq\Delta$, and $\F$ is saturated. 
 \item [(b)] Suppose $(\L,\Delta,S)$ is $Q$-replete and of objective characteristic $p$. Then $Q$ is large in $(\L,\Delta,S)$ if and only if $Q$ is large in $\F$. Moreover, if so, then $\F$ is saturated and $(\L,\Delta,S)$ is a linking locality over $\F$.
 \item [(c)] If $\F$ is saturated and $Q$ is large in $\F$, then $Q$ is large in any linking locality over $\F$.
\end{itemize}
\end{lemma}

\begin{proof}
\textbf{(a)} Assume that $Q$ is large in $\L$. Clearly $ C_S(Q) \le C_\L(Q) = Z(Q)$.\\
Let now $1 \ne U \le Z(Q)$ so that $N_\L(U) \subseteq N_\L(Q)$ by \eqref{Q!L}. As $(\L,\Delta,S)$ is $Q$-replete, we have $U,Q\in\Delta$. Using Lemma~\ref{L:NormalizerLocFus} we obtain therefore
\[
N_\F(U) = \F_{N_S(U)}(N_\L(U)) \subseteq \F_{N_S(Q)}(N_\L(Q)) = N_\F(Q).
\]
Thus \eqref{Q!F} holds and $Q$ is large in $\F$. By Lemma~\ref{L:LargeLocMain}(a), we have $Q\unlhd S$ and so the remaining assertions follows from Lemma~\ref{L:QRepleteSaturated}.

\smallskip



\textbf{(b)} Assume now the hypothesis of (b). If $Q$ is large in $(\L,\Delta,S)$, then by (a), $Q$ is large in $\F$, $\F$ is saturated and $\F^{cr}\subseteq\F^c\subseteq\Delta$. Note that the latter two properties imply that $(\L,\Delta,S)$ is a linking locality. It remains thus only to show that $Q$ is large in $(\L,\Delta,S)$ if it is large in $\F$. 

\smallskip

Assume now that $Q$ is large in $\F$.  Then $Q\unlhd S$ by Lemma~\ref{L:LargeFusMain}(c) and so $\F$ is saturated by Lemma~\ref{L:QRepleteSaturated}. 

\smallskip

Since $(\L,\Delta,S)$ is $Q$-replete, we have $Z(Q) \in \Delta$. Thus $Q \in \Delta$ as $\Delta$ is overgroup-closed in $S$. In particular, $N_\L(Q)$ is a group of characteristic $p$. As $Q\unlhd S$ and $S$ is a maximal $p$-subgroup of $\L$, it follows that $S$ is a Sylow $p$-subgroup of $N_\L(Q)$. Hence, $N_\L(Q)$ is a model for $N_\F(Q)$ by Lemma~\ref{L:NormalizerLocFus}. As $C_S(Q)\leq Q$, if follows from \cite[Theorem~2.1(b)]{Henke:2015} that $C_\L(Q)=C_{N_\L(Q)}(Q)\leq Q$.

\smallskip

Let now $1 \ne U \le Z(Q)$ to verify the property \eqref{Q!L}. By \cite[Lemma~2.9]{Chermak:2022}, there exists $f\in\L$ such that $N_S(U) \le S_f$ and $N_S(U^f) \in \Syl_p(N_\L(U^f))$. Recall that $\F$ is saturated. Hence, $Q$ is by Lemma~\ref{L:LargeSatFusMain}(a) weakly closed in $\F$. As $Q \le N_S(U)$, it follows $Q = Q^f$. In particular, $U^f\leq Z(Q)$ and $f\in N_\L(Q)$. Moreover, if $N_\L(U^f) \le N_\L(Q)$, then $ N_\L(U) = N_\L(U^f)^{f\inv} \le N_\L(Q)$ by Lemma~\ref{LocalitiesProp}(b). Hence up to replacing $U$ by $U^f$, we may suppose that $N_S(U)\in\Syl_p(N_\L(U))$. In particular, we have then $U^*:= O_p(N_\L(U)) \le S$. As $\L$ is of objective characteristic $p$, we obtain $C_\L(U^*) = C_{N_\L(U)}(U^*) \le U^*$. 

\smallskip

For any $h \in N_\L(U)$ the map $c_h|_{U^*}$ is a well-defined morphism in $N_\F(U)$. Note that $N_\F(U) \subseteq N_\F(Q)$ by \eqref{Q!F} and that $N_\F(Q)=\F_S(N_\L(Q))$ by Lemma~\ref{L:NormalizerLocFus}. Hence, there exists $g \in N_\L(Q)$ such that $c_h|_{U^*} = c_g|_{U^*}$. Then $(h,g\inv)$ and $(h,g\inv,g)$ are in $\D(\L)$ via $U^*$ and $hg\inv \in C_\L(U^*) \le U^* \le S \le N_\L(Q)$. As $g\in N_\L(Q)$, it follows that $h=\Pi(h,g\inv,g) =\Pi(hg\inv,g)\in N_\L(Q)$, proving that $N_\L(U)\subseteq N_\L(Q)$. Thus, \eqref{Q!L} holds and $Q$ is large in $\L$.

\smallskip

\textbf{(c)} Suppose $\F$ is saturated, $Q$ is large in $\F$, and $(\L,\Delta,S)$ is a linking locality over $\F$. It follows from Proposition~3.3 and Theorem~7.2(a) in \cite{Henke:2015} that there exists a subcentric linking locality $(\L^s,\F^s,S)$ over $\F$ such that $\L=\L^s|_{\Delta}$. Here the ``restriction'' $\L^s|_\Delta$ of $\L^s$ is defined as in \cite[Definition~2.20]{Chermak:2013}. We refer the reader also to the discussion on restrictions of localities in \cite[Definition~2.22, Lemma~2.23]{Henke:2021}. Using \cite[Lemma~2.23]{Henke:2021} one observes easily that
\[N_\L(P)=N_{\L^s}(P)\cap \L\mbox{ for every }P\leq S.\]
As pointed out in Example~\ref{E:SubcentricQReplete}, the subcentric linking locality $(\L^s,\F^s,S)$ over $\F$ is $Q$-replete. In particular, it follows from (b) that $Q$ is large in $\L^s$. Since $\L \subseteq \L^s$ we get $C_\L(Q) \subseteq C_{\L^s}(Q) \subseteq Q$. Let $1 \ne U \le Z(Q)$. As the property \eqref{Q!L} holds for $\L^s$ in place of $\L$, we have 
\[N_\L(U)= N_{\L^s}(U)\cap \L\subseteq N_{\L^s}(Q)\cap \L=N_\L(Q).\]
Hence, the property \eqref{Q!L} holds for $\L$ and so $Q$ is large in $(\L,\Delta,S)$.
\end{proof}

\subsection{Large subgroups of group fusion systems}

Let $G$ be a finite group and $Q\leq S\in\Syl_p(G)$. If $N$ is  a normal $p^\prime$-subgroup of $G$ and the image of $Q$ in $G/N$ is large in $G/N$, then we have seen in Example~\ref{E:LargeGLargeFSG} that $Q$ is large in $\F_S(G)$. However, the converse is not true in general as either of the two Examples \ref{Ex:G23G2q} or \ref{Ex:AutG23Extended} below shows.

\smallskip

In the introduction we have already considered the case that $G$ is of the form $G=H\times R$, where $S\in\Syl_p(H)$, the group $R\neq 1$ has order prime to $p$, and $Q\leq S$ is large in $H$. Then $Q$ is large in $\F_S(G)=\F_S(H)$, but not in $G$. More precisely, the $Q$-uniqueness property \eqref{Q!} holds in $G$, but $C_G(Q)\not\leq Q$. We provide examples of finite groups $G$ below where a $2$-group $Q$ is large in the $2$-fusion system of $G$, and either just the $Q$-uniqueness property fails and the self-centralizing property holds, or both the self-centralizing and the $Q$-uniqueness property fail.

\smallskip

For any two groups $G_1$ and $G_2$ with $|Z(G_1)|=|Z(G_2)|=2$, we write $G_1\circ G_2$ for the central product of $G_1$ and $G_2$.

\begin{example}\label{Ex:G23G2q}
The group $\G_2(3)$ has a maximal subgroup $X$ of shape $(\SL_2(3)\circ\SL_2(3)):2$. Note that 
\[Q:=O_2(X)\cong Q_8\circ Q_8.\]
The maximality of $X$ yields $N_{\G_2(3)}(Q)=X=C_{\G_2(3)}(Z(Q))$. In particular, $C_{\G_2(3)}(Q)\leq Q$ and $Q$ is a large subgroup of $\G_2(3)$. Hence, for a Sylow $2$-subgroup $T$ of $\G_2(3)$ with $Q\leq T\leq X$, the subgroup $Q$ is by Example~\ref{E:LargeGLargeFSG} a large subgroup of the fusion system $\F_T(\G_2(3))$. 

\smallskip

Let $q$ be an odd prime power with $q\equiv 3,5\mod 8$ and $q>3$. Then the $2$-fusion system of $\G_2(q)$ is isomorphic to the $2$-fusion system of $\G_2(3)$. So identifying $T\in\Syl_2(\G_2(3))$ in a suitable way with a Sylow $2$-subgroup of $\G_2(q)$, we have $\F_T(\G_2(q))=\F_T(\G_2(3))$ and thus $Q$ is large in $\F_T(\G_2(q))$. It turns out then that $X_q:=C_{\G_2(q)}(Z(Q))$ is a maximal subgroup of $\G_2(q)$ of shape $(\SL_2(q)\circ \SL_2(q)):2$ and $Q$ is a Sylow $2$-subgroup of $O^2(X_q)\cong \SL_2(q)\circ \SL_2(q)$. Note that $Q$ is not normal in $O^2(X_q)$ as $q>3$. Hence, $N_{\G_2(q)}(Z(Q))=X_q\not\leq N_{\G_2(q)}(Q)$ and so the $Q$-uniqueness property \eqref{Q!} fails in $\G_2(q)$. In particular, $Q$ is not large in $\G_2(q)$. At the same time, we have $C_{\G_2(q)}(Q)=C_{X_q}(Q)\leq Q$. Note also that $\G_2(q)$ is simple and thus $O_{2^\prime}(\G_2(q))=1$. 
\end{example}

\begin{proof}
 It follows from \cite[Theorem A]{BMO} that the $2$-fusion systems of $\G_2(3)$ and $\G_2(q)$ are isomorphic if $q\equiv 3,5\mod 8$; explicitly this is stated in \cite[p.9, Table~0.1]{BMO2}. Indeed, a Sylow $2$-subgroup of $\G_2(3)$ or $\G_2(q)$ has order $2^6$ (cf. \cite[Section~4.3.3]{Wilson:2009}). 
 
 \smallskip
 
We use now that, for any odd prime power $q^*$, the group $\G_2(q^*)$ has a maximal subgroup of shape $(\SL_2(q^*)\circ\SL_2(q^*)):2$  (cf. Section~4.3.6 and Table~4.1 in \cite{Wilson:2009}). Note that a maximal subgroup $X$ of $\G_2(3)$ of shape $(\SL_2(3)\circ \SL_2(3)):2$ must contain a Sylow $2$-subgroup of $\G_2(3)$. Choose now $X$, $Q$ and $T\in\Syl_2(X)\subseteq \Syl_2(\G_2(3))$ as above and identify $T$ with a Sylow $2$-subgroup of $\G_2(q)$.  Observe that $Z(Q)=Z(T)$. As $\G_2(q)$ contains a maximal subgroup of shape $(\SL_2(q)\circ \SL_2(q)):2$  and such a maximal subgroup must contain a Sylow $2$-subgroup of $\G_2(q)$, we can conclude that $X_q:=C_{\G_2(q)}(Z(Q))$ is maximal and of this shape.
\end{proof}

Note that in the example above the $Q$-uniqueness property fails, while the self-centralizing property holds. We now provide also a series of examples where, for a large subgroup $Q$ of the $2$-fusion system of a group $G$, both the $Q$-uniqueness property and the self-centralizing property fail. Example~\ref{Ex:AutG23Extended} seems moreover interesting in connection with Theorem~\ref{mainMS} and the results in Section~\ref{S:MS}. Recall e.g. from \cite[Section~4.3.9]{Wilson:2009} that $\G_2(q)$ has a non-trivial graph automorphism only in characteristic $3$, which is why we restrict attention to this case below.

\begin{example}\label{Ex:AutG23Extended}
If $Q\cong Q_8 \circ Q_8$ is the subgroup of $\G_2(3)$ introduced in Example~\ref{Ex:G23G2q}, then $Q$ is not only large in $\G_2(3)$, but also in $\Aut(\G_2(3))$; for details see Lemma~\ref{L:G23}. In particular, if we take $S\in\Syl_2(\Aut(\G_2(3))$ with $Q\leq S$, then $Q$ is a large subgroup of $\F_S(\Aut(\G_2(3)))$ by Example~\ref{E:LargeGLargeFSG}.

\smallskip

Let now $q=3^r$ for a positive odd integer $r$. Identify $\G_2(q)$ with $\Inn(\G_2(q))$ and let $G$ be a subgroup of $\Aut(\G_2(q))$ of odd index such that $\G_2(q)\leq G$. Then the following hold:
\begin{itemize}
 \item [(a)] $O_{2^\prime}(G)=1$.
 \item [(b)] The Sylow $2$-subgroup $S$ of $\Aut(\G_2(3))$ may be identified with a Sylow $2$-subgroup of $G\leq \Aut(\G_2(q))$ such that $\F_S(G)=\F_S(\Aut(\G_2(3)))$. In particular, $Q\leq S$ is large in $\F_S(G)$.
 \item [(c)] For $r>1$, the subgroup $Q$ is not large in $G\cong G/O_{2^\prime}(G)$. More precisely, the $Q$-uniqueness property \eqref{Q!} always fails, and we have the following: 
 \begin{itemize}
 \item If $|G:\G_2(q)|=2$, then $C_G(Q)\leq Q$.
 \item If $G=\Aut(\G_2(q))$ or, more generally, if $|G:\G_2(q)|>2$, then $C_G(Q)\not\leq Q$.
 \end{itemize}
\end{itemize}
\end{example}

\begin{proof}
We use throughout that $Q$ is large in $\Aut(\G_2(3))$. Details are provided in Lemma~\ref{L:G23}

\textbf{(a)} As $\G_2(q)$ is simple and a normal subgroup of $G$ whose centralizer in $G$ is trivial, we have $O_{2^\prime}(G)=1$. 

\smallskip

\textbf{(b)} Since $3^2 \equiv 1 \mod 8$ and $r$ is assumed to be odd, we have $q=3^r \equiv 3 \mod 8$. Thus, the $2$-fusion systems of $\G_2(3)$ and $\G_2(q)$ are isomorphic as seen in Example~\ref{Ex:G23G2q}. 

\smallskip

Note that we may regard $\G_2(3)$ as a subgroup of $\G_2(q)$. As the two groups have isomorphic $2$-fusion systems, a Sylow $2$-subgroup $T$ of $\G_2(3)$ is also a Sylow $2$-subgroup of $\G_2(q)$ and $\F_T(\G_2(3))\subseteq \F_T(\G_2(q))$ are isomorphic fusion systems. Thus, $\F_T(\G_2(3))=\F_T(\G_2(q))$. 

\smallskip

By $\Aut(\bF_q)$ we will denote also the field automorphisms of $\G_2(q)$. Recall that, $q$ being a power of $3$, $\Aut(\G_2(q)) \cong\G_2(q) \rtimes \<\sigma_0 \>$ where $\sigma_0$ is a graph automorphism that squares to a generator $\phi$ of $\Aut(\bF_q)$; see for example \cite[Sections 4.3.9 and 4.5.1]{Wilson:2009} and \cite[]{Carter:1972a} for a construction of $\sigma_0$. In particular, $\Aut(\G_2(3))$ is isomorphic to the semidirect product of $\G_2(3)$ with a graph automorphism of order $2$. At the same time $\G_2(3)$ is the group of fix points of $\phi=\sigma_0^2$ and thus normalized by $\sigma_0$. Indeed, $\sigma_0|_{\G_2(3)}$ is also a graph automorphism of $\G_2(3)$ of order $2$. 

\smallskip

As $o(\phi)=o(\sigma_0^2)=|\Aut(\bF_q)|=r$ is odd, we can write $\<\sigma_0\>=\<\phi\>\times \<\tau\>$ for $\tau=\sigma_0^r$. Since $\phi$ fixes the elements of $\G_2(3)$, $\tau$ acts on $\G_2(3)$ the same way as $\sigma_0$. Hence, we may identify $\Aut(\G_2(3))$ with $\G_2(3)\rtimes \<\tau\>$. Similarly we identify $\Aut(\G_2(q))$ with $\G_2(q) \rtimes \<\sigma_0 \>$ and regard thus $\Aut(\G_2(3))$ as a subgroup of $\Aut(\G_2(q))$. We may choose $T\in\Syl_2(\G_2(3))$ to be $\tau$-invariant and thus $\sigma_0$-invariant. Then $T\rtimes \<\tau\>$ is a Sylow $2$-subgroup of $\Aut(\G_2(3))$ and we may assume $S=T\rtimes \<\tau\>$. Note that $S$ is also a Sylow $2$-subgroup of $\Aut(\G_2(q))$.

\smallskip

Observe that $N:=\G_2(q)\rtimes \<\tau\>=O^{2^\prime}(\Aut(\G_2(q)))$. In particular, by Sylow's Theorem, all Sylow $2$-subgroups of $\Aut(\G_2(q))$ are conjugate under $N=\G_2(q)S$ and thus under $\G_2(q)\leq G$. At the same time, as $G$ has odd index in $\Aut(\G_2(q))$, the subgroup $G$ contains a Sylow $2$-subgroup of $\Aut(\G_2(q))$. Hence, $G$ contains every Sylow $2$-subgroup of $\Aut(\G_2(q))$. In particular, $S\in\Syl_2(G)$ and $N\leq G$. 

\smallskip

The elements of $\Aut(\bF_q)=\<\phi\>$ act trivially on $T\leq \G_2(3)$. At the same time, $\tau$ and $\phi$ are both elements of the cyclic group $\<\sigma_0\>$ and therefore centralize each other. Thus, the elements of $\Aut(\bF_q)=\<\phi\>$ centralize $S$. In particular $\F_S(N)=\F_S(G)=\F_S(\Aut(\G_2(q))$. As $\G_2(q)$ has index $2$ in $N$, it follows that $\E:=\F_T(\G_2(q))$ is a normal subsystem of $\F_S(N)=\F_S(G)$ of $2$-power index. Moreover, $\E=\F_T(\G_2(3))\subseteq \F_S(\Aut(\G_2(3)))\subseteq \F_S(\Aut(\G_2(q)))=\F_S(G)$. Thus, it follows from Lemma~\ref{L:AlmostSimpleFusIsoHelp} applied with $\F_S(G)$ and $\F_S(\Aut(\G_2(3)))$ in place of $\F$ and $\mG$ that $\F_S(G)=\F_S(\Aut(\G_2(3)))$. This shows (b).

\smallskip

\textbf{(c)} We keep the notation introduced in the proof of (b). As argued in Example~\ref{Ex:G23G2q}, the $Q$-uniqueness property fails to hold in $\G_2(q)$. Hence, it fails also in $G\geq \G_2(q)$. 

\smallskip

As $Q$ is large in $\Aut(\G_2(3))$, we have in particular $C_S(Q)\leq Q$. Note moreover that $C_S(Q)$ is a  Sylow $2$-subgroup of $C_G(Q)$ since $Q\unlhd S$. We have furthermore seen in Example~\ref{Ex:G23G2q} that $C_{\G_2(q)}(Q)\leq Q$. Hence, if $|G:\G_2(q)|=2$, then $C_G(Q)$ contains no element of odd order and so $C_G(Q)=Z(Q)$. 

\smallskip

If $|G:\G_2(q)|>2$, then $N<G\leq \Aut(\G_2(q))=N\Aut(\bF_q)$ and so $G\cap \Aut(\bF_q)\neq 1$. As $\Aut(\bF_q)$ centralizes $T$ and thus $Q$, it follows $C_G(Q)\not\leq Q$.
\end{proof}

In the following lemma we state some group-theoretical properties which are equivalent to $Q$ being large in $\F_S(G)$.

\begin{lemma}\label{L:QlargeFSG}
 Let $G$ be a finite group, $S\in\Syl_p(G)$ and $Q\leq S$. Write $\mathcal{U}$ for the set of non-trivial subgroups of $Z(Q)$, and $\mathcal{U}^f$ for the fully $\F_S(G)$-normalized subgroups in $\mathcal{U}$. If $C_S(Q)\leq Q$, then the following conditions are equivalent:
 \begin{itemize}
  \item [(i)] Property \eqref{Q!F} holds with $\F=\F_S(G)$.
  \item [(ii)] For every $U\in\mathcal{U}$, the subgroup $Q$ is normal in $\F_{N_S(U)}(N_G(U))$.
  \item [(ii')] For every $U\in\mathcal{U}^f$, the subgroup $Q$ is normal in $\F_{N_S(U)}(N_G(U))$.
  \item [(iii)] For every $U\in \mathcal{U}$, every characteristic subgroup of $Q$ is strongly closed in $N_G(U)$. 
  \item [(iii')] For every $U\in \mathcal{U}^f$, every characteristic subgroup of $Q$ is strongly closed in $N_G(U)$.
  \item [(iv)] For every $U\in\mathcal{U}$, there exists a series $1=Q_0\leq Q_1\leq\cdots \leq Q_n=Q$ of subgroups of $Q$ such that $Q_i$ is strongly closed in $N_G(U)$ and $[Q_i,Q]\leq Q_{i-1}$ for all $i=1,2,\dots,n$.
  \item [(iv')] For every $U\in\mathcal{U}^f$, there exists a series $1=Q_0\leq Q_1\leq\cdots \leq Q_n=Q$ of subgroups of $Q$ such that $Q_i$ is strongly closed in $N_G(U)$ and $[Q_i,Q]\leq Q_{i-1}$ for all $i=1,2,\dots,n$.
 \end{itemize}
In particular, $Q$ is large in $\F_S(G)$ if and only if $C_S(Q)\leq Q$ and one of the above conditions (ii)-(iv) or (ii')-(iv') holds.
\end{lemma}

\begin{proof}
Set $\F:=\F_S(G)$. We use throughout that, for every subgroup $U$ of $S$, we have $N_\F(U)=\F_{N_S(U)}(N_G(U))$. In particular, by Lemma~\ref{L:LargeFusMain}(b), (i) implies (ii).

\smallskip

Clearly (ii) implies (ii'), (iii) implies (iii') and (iv) implies (iv').

\smallskip

Note that, for every $U\in \mathcal{U}^f$, $N_S(U)\in\Syl_p(N_G(U))$ and $N_\F(U)=\F_{N_S(U)}(N_G(U))$ is saturated. In particular, for $U\in\mathcal{U}^f$, a subgroup of $N_S(U)$ is strongly closed in $N_\F(U)$ if and only if it is strongly closed in $N_G(U)$. Hence, the properties (ii')-(iv') are equivalent by Lemma~\ref{L:FusNormEquivalent}.

\smallskip

If (ii') holds, then $Q\unlhd N_\F(U)$ and thus $N_\F(U)\subseteq N_\F(Q)$ for every $U\in\mathcal{U}^f$. This means that the property \eqref{Q!Ff} holds and thus $Q$ is large in $\F$ by Lemma~\ref{L:LargeFusEquiv}. Hence, (ii') implies (i). This shows that all the stated conditions are equivalent.
\end{proof}

\begin{rmk}\label{R:QLargeFSGQrepleteLinkingLoc}
As before let $G$ be a finite group and $Q\leq S\in\Syl_p(G)$. If $Q$ is large in $G$, then we have seen in Example~\ref{E:DeltaSGLarge} that a $Q$-replete linking locality over $\F_S(G)$, in which $Q$ is large, can be constructed directly from the group $G$. This is not necessarily true anymore if we assume the weaker assumption that $Q$ is large in $\F_S(G)$. Setting
\[\Delta^*:=\{P\leq S\colon N_G(P)/O_{p^\prime}(N_G(P))\mbox{ is of characteristic $p$}\},\]
it is described in \cite[Lemma~10.7(b)]{Henke:2015} how a linking locality $(\L,\Delta^*,S)$ over $\F_S(G)$ with object set $\Delta^*$ can be constructed directly from the group $G$. By Lemma~\ref{L:LargeLF}(c), $Q$ is large in $\L$ if $Q$ is large in $\F_S(G)$. However, $(\L,\Delta^*,S)$ is not necessarily $Q$-replete. Indeed, we do not see any way to construct a $Q$-replete linking locality over $\F_S(G)$ directly from the group $G$. However, \cite[Theorem~7.2]{Henke:2015} together with \cite[Proposition~3.3]{Henke:2015} allows to extend the linking locality $(\L,\Delta^*,S)$ to a subcentric linking locality over $\F_S(G)$ using only relatively elementary constructions and not the existence and uniqueness of centric linking systems. Such a subcentric linking locality is $Q$-replete as pointed out in Example~\ref{E:SubcentricQReplete}. Thus, if $Q$ is large in $\F_S(G)$, then the existence of a $Q$-replete linking locality over $\F_S(G)$ can be shown without using the existence and uniqueness of centric linking systems.
\end{rmk}

\section{The proof of Theorem~\ref{mainMS} and a related theorem on localities}\label{S:MS}

In this section we will prove Theorem~\ref{mainMS}. However, we start by formulating the hypothesis in a locality setting and analyze first the structure of the locality. 

\subsection{Basic Setup}

The reader might want to recall the definition of $Y_G$ from Definition~\ref{D:YG}. In this section, we focus on the prime $2$. Thus, for any finite group $G$, we will denote by $Y_G$ the largest elementary abelian normal $2$-subgroup of $G$ which is a $2$-reduced $G$-module.

\smallskip

A finite group $G$ is called a $\m{K}$-group if every simple section of $G$ is a known finite simple group. For most of this section we will work under the following hypothesis.

\begin{hypo}\label{MainHyp}
 Suppose that $S$ is a $2$-group and let $(\L,\Delta,S)$ be a linking locality with a large subgroup $Q\leq S$. Assume furthermore that $(\L,\Delta,S)$ is $Q$-replete. 
 
 \smallskip
 
Set $\tilde{C}=N_\L(Q)$ and suppose that $\tilde{C}$ is a $\m{K}$-group. Let $M$ be a subgroup of $\L$ containing $S$ and set 
 \[V:=[Y_M,M].\]
Assume that the following conditions hold:
\begin{itemize}
 \item [(i)] $Q=O_2(\tilde{C})$.
 \item [(ii)] $M/O_2(M)\cong \SL_3(2)$ and $V$ is a natural $\SL_3(2)$-module for $M/O_2(M)$.
 \item [(iii)] $Y_M\not\leq Q$ and $V\leq Q$.
\end{itemize}
\end{hypo}

Recall that $Q\in\Delta$ as $(\L,\Delta,S)$ is $Q$-replete. Hence, $\tilde{C}$ is a subgroup and property (i) makes sense. Note also that, by Lemma~\ref{lemma:parabolics in local of char p}, the group $M$ is of characteristic $2$.

\smallskip

Our goal is to prove the following result.

\begin{theorem}\label{unique local}
Suppose that $(\L,\Delta,S)$ is a locality satisfying Hypothesis \ref{MainHyp} and let $\F=\F_S(\L)$. Then $\F$ is isomorphic to the $2$-fusion system of $\Aut(\G_2(3))$.
\end{theorem}

Recall that, by Theorem~\ref{T:LinkingLocalityExistence}, the linking locality $(\L,\Delta,S)$ is uniquely determined up to rigid isomorphism, provided $\F$ is uniquely determined up to isomorphism and the set $\Delta$ is known. Moreover, if $\F$ is known, then there are only finitely many choices for $\Delta$, with $\Delta=\F^s$ being the largest for which there exists a linking locality. Working out $\F^s$ and all the possible choices for $\Delta$ if $\F$ is the $2$-fusion system of $\Aut(\G_2(3))$ is however not a task which we consider worthwhile. Thus, strictly speaking, the locality $(\L,\Delta,S)$ will not be classified. Nevertheless, when showing Theorem~\ref{unique local}, we prove along the way many results about the locality $(\L,\Delta,S)$. In particular, the structure of $M$ is described in Lemma~\ref{L:M*C*}(a) and the structure of $\tilde{C}$ is determined in Lemma~\ref{L:StructuretildeC}. This is then used to show in Proposition~\ref{P:AmalgamsIso} that the amalgam formed by $M$, $\tilde{C}$, their intersection and the corresponding inclusion maps is isomorphic to a certain amalgam appearing inside of $\Aut(\G_2(3))$, which is specified in Lemma~\ref{L:G23} and the notation introduced thereafter.

\vspace*{0,4cm}

Throughout this section, except in the proof of Theorem~\ref{mainMS} at the very end, we assume Hypothesis~\ref{MainHyp} and use the notation introduced below.\\
Set 
\[Q_M:=O_2(M),\;\tilde{M}=N_\L(V),\;M^\circ=\<Q^M\>\mbox{ and }Z=\Omega_1 Z(S).\]

As $S\leq M$, we have $V \unlhd S$. Thus, since $(\L,\Delta,S)$ is $Q$-replete, we obtain $V \in \Delta$, so that $\tilde{M}$ is a subgroup of $\L$. By Lemma~\ref{L:LargeLocMain}(a) $Q$ is weakly closed in $\L$, so also $Q\unlhd S$ and $S\leq \tilde{C}$. By the definition of a locality, $S$ is maximal in the poset of $p$-subgroups of $\L$. Hence
$$S\in \Syl_2(M)\cap \Syl_2(\tilde{C}).$$
Note that there exists $L\leq \tilde{C}$ such that $L$ is $(M\cap \tilde{C})$-invariant and $Y_M\not\leq O_2(LY_M)$, since these properties hold for example for $L=\tilde{C}$. Choose such $L$ minimal with respect to inclusion. Set
$$W:=\<V^L\>,\;B:=(M\cap\tilde{C})(L\cap\tilde{M}),\;M_1:=MB,\;M_2:=LB,\;H:=LY_M,\;T:=O^2(M\cap\tilde{C}).$$
Clearly $L$ is $B$-invariant. Thus, we obtain $L \norm M_2 \le \tilde{C}$. Since $V \norm M$, we have $M_1 \subseteq \tilde{M}$. We will show in Lemma~\ref{Lemma2.1}(d) below that $M_1$ is a subgroup of $\tilde{M}$ and thus of $\L$.  Hence, it makes sense to define
\[Q_i=O_2(M_i)\mbox{ for }i=1,2.\]
We also set 
\[\ov{X}:=XQ_2/Q_2\mbox{ for any }X\leq M_2\] 
and 
\[\widehat{X}=XZ(W)/Z(W)\mbox{ for any }X\leq W.\]

Note that $W \le Q \le Q_2$ since $\<V^L\> \le \< Q^L \> = Q$ and since every Sylow $2$-subgroup of $M_2$ is also a Sylow $2$-subgroup of $\tilde{C}\geq M_2$. Similarly, we have $Q_1 \le Q_M$ as every Sylow $2$-subgroup of $M$ is also a Sylow $2$-subgroup of $M_1\geq M$.

\subsection{Some first reductions}

\begin{lemma}[Cf. {\cite[Lemma~2.1]{MStr:2008}}]\label{Lemma2.1}
\noindent \begin{itemize}
 \item [(a)] $Q\not\leq Q_M$, $C_\L(M^\circ)=1$, $Z(M)=1$ and $Y_M=\Omega_1Z(Q_M)$.
 \item [(b)] $Z=C_V(Q)$, $|Z|=2$, $|Y_M/V|=2$, $M\cap\tilde{C}=C_M(Z)$, $QQ_M=O_2(M\cap\tilde{C})$, $M=M^\circ Q_M$ and $[Y_M,Q]=V$. In particular, $|\ov{Y_M}| = 2$.
 \item [(c)] $\tilde{M}=M^\circ C_\L(V)=M^\circ C_{\tilde{M}}(V)$, $[M^\circ,C_\L(V)]\leq O_2(M^\circ)\leq O_2(\tilde{M})\leq Q_M$ and $M^\circ\unlhd \tilde{M}$.
 \item [(d)] $M_1=M^\circ B$ is a subgroup of $\tilde{M}$.
 \item [(e)] We have $Y_M\unlhd \tilde{M}$ and $C_\L(V)=C_\L(Y_M)$.
 \item [(f)] We have $O_2(M^\circ)=M^\circ\cap Q_1=M^\circ\cap Q_M$, $B=(M^\circ\cap B)C_B(V)$, $B\cap M^\circ=C_{M^\circ}(Z)$, $C_{M_1}(V)=C_B(V)$; moreover
 \[
M_1/Q_1=M^\circ Q_1/Q_1\times C_B(V)/Q_1 \quad \text{with } M^\circ Q_1/Q_1 \cong \SL_3(2).
 \]
  \item [(g)] We have $O_2(B)=Q_1Q_2=Q_1Q$, and $B$ acts irreducibly on $O_2(B)/Q_1$. Moreover, $O_2(B)$ is the transvection group to the point $Z\leq V$ and $[Y_M,O_2(B)]=V$. 
  \item [(h)] $C_{Q_2}(V)\leq Q_1$.
\end{itemize}
\end{lemma}

\begin{proof}
\textbf{(a)} Since $Y_M$ is $2$-reduced, if $Q\leq Q_M$ then $Y_M\leq C_\L(Q)\leq Q$ contradicting hypothesis (iii). Hence, $Q\not\leq Q_M$. Observe that $C_\L(M^\circ)\leq C_\L(Q)=Z(Q)$ as $Q$ is large. In particular, $C_\L(M^\circ)=C_M(M^\circ)\unlhd M$. So if $C_\L(M^\circ)\neq 1$, by (\ref{Q!L}), $M\leq N_\L(C_\L(M^\circ))\subseteq N_\L(Q)$, contradicting $Q\not\leq Q_M$. Hence, $C_\L(M^\circ)=1$ and consequently $Z(M)=1$. Clearly, $Y_M\leq \Omega_1Z(Q_M)$. Since $Q_M\leq C_M(\Omega_1 Z(Q_M))\unlhd M$, $M/Q_M\cong \SL_3(2)$ is simple and $[Y_M,M]\neq 1$, we have $C_M(\Omega_1 Z(Q_M))=Q_M$ and so $O_2(M/C_M(\Omega_1 Z(Q_M)))=1$. Hence, by definition of $Y_M$, $Y_M=\Omega_1 Z(Q_M)$.

\smallskip

\textbf{(b)} As $V=[Y_M,M]$, $V_1:=VZ$ is an $M$-submodule of $Y_M$ with $[V_1,M]=V$ and $V_1=VC_{V_1}(S)$. Hence, by Lemma~\ref{L:ApplyGaschutz}(c) applied with $V_1$ in place of $V$, we have $V_1=VC_{V_1}(M)$ and then $V_1=V$ by (a). 
This shows $Z\leq V$ and, in particular, $Z=C_V(S)$ has order $2$. As $M/Q_M$ is simple and $Q\not\leq Q_M$, we have $M=M^\circ Q_M$. Since $Z\leq C_G(Q)=Z(Q)$, it follows from (\ref{Q!L}) that $C_M(Z)\leq M\cap \tilde{C}$. So $C_M(Z)$ normalizes $C_V(Q)$. Note $Z\leq C_V(Q)$ and $C_V(Q)<V$ as $Q\not\leq Q_M$. By the structure of the natural $\SL_3(2)$-module, $C_M(Z)$ acts irreducibly on $V/Z$. Hence, $C_V(Q)=Z$ and thus $M\cap\tilde{C}\leq N_M(Z)$. As $|Z|=2$, it follows $M\cap\tilde{C}=C_M(Z)$. Again by the structure of the natural $\SL_3(2)$-module, $(M\cap\tilde{C})/Q_M\cong S_4$ and $O_2(M\cap\tilde{C})$ is the transvection group to the point $Z\leq V$. As $1 \neq (QQ_M)/Q_M\unlhd (M\cap\tilde{C})/Q_M$, it follows $QQ_M=O_2(M\cap\tilde{C})$.
By hypothesis (iii), $V<Y_M$. So (a) together with \cite[Lemma 1.1(b)]{MStr:2008} gives $|Y_M/V|=2$ and $[Y_M,Q]=V$.

\smallskip

If $Y_M \le Q_2$, then  $Y_M \le (LY_M)\cap Q_2 \le O_2(LY_M)$, contradicting the choice of $L$. Thus $Y_M \not \le Q_2$ and, since $V \le Q \cap M_2 \le Q_2$, we have $|\ov{Y_M}|=2$.

\smallskip

\textbf{(c)} As $M/Q_M\cong \SL_3(2)=\Aut(V)$ and $M=M^\circ Q_M$, we have $\Aut(V)=\Aut_{M^\circ}(V)$ and $\tilde{M}=M^\circ C_\L(V)=M^\circ C_{\tilde{M}}(V)$. Since $Q$ is large, $C_\L(V)\subseteq C_\L(C_V(Q))\subseteq N_\L(Q)=\tilde{C}$. Thus $C_\L(V)=C_{\tilde{C}}(V)$ is a subgroup of $\tilde{C}$ and $[Q,C_\L(V)]\leq Q$. As $[Q,C_\L(V)]$ is normalized by $C_\L(V)$, it follows $[Q,C_\L(V)]\leq O_2(C_\L(V))\cap M^\circ\leq O_2(M^\circ)$. As $C_\L(V)$ and $O_2(M^\circ)$ are normalized by $M$, it follows from the definition of $M^\circ$ that $[M^\circ,C_\L(V)]\leq O_2(M^\circ)$. In particular, $M^\circ\unlhd \tilde{M}=M^\circ C_\L(V)$ and thus $O_2(M^\circ)\unlhd \tilde{M}$. Hence, $O_2(M^\circ)\leq O_2(\tilde{M})$. As $O_2(\tilde{M})\leq S\leq M\leq \tilde{M}$, we have furthermore $O_2(\tilde{M})\leq O_2(M)=Q_M$. This proves (c).

\smallskip

\textbf{(d)} By (b) and (c), we have $M=M^\circ Q_M$ and $M^\circ\unlhd \tilde{M}$. As $Q_M\leq S\leq M\cap\tilde{C}\leq B$, it follows $M_1=M^\circ B$. In particular, $M_1$ is a subgroup of $\tilde{M}$. 



\smallskip

\textbf{(e)} Set $D:=\<Y_M^{\tilde{M}}\>$. As $\tilde{M}$ normalizes $M^\circ$ and $V=[Y_M,M^\circ]$, it follows $V=[D,M^\circ]$. Since $[D,V]=1$, it follows $[D,M^\circ,D]=1=[M^\circ,D,D]$ and by the three-subgroups lemma, $[D,D,M^\circ]=1$. Now by (a), $[D,D]\leq C_\L(M^\circ)=1$. Hence, $D$ is abelian and thus elementary abelian. As $[D,M^\circ]=V$, it follows now from \cite[Lemma~1.1(b)]{MStr:2008} that $|D/V|= 2$. By hypothesis (iii), $V<Y_M\leq D$, so $Y_M=D\unlhd \tilde{M}$. Since $|Y_M/V|=2$, $[Y_M,\tilde{M}]\leq V$, hence $[Y_M,O^2(C_\L(V))]=1$. Note $C_\L(V)=C_{\tilde{M}}(V)\unlhd \tilde{M}$ and $S\in\Syl_2(\tilde{M})$. Hence, $Q_M=C_S(V)\in \Syl_p(C_\L(V))$. As $[Y_M,Q_M]=1$, it follows $C_\L(V)=Q_MO^2(C_\L(V))\leq C_\L(Y_M)$. This completes the proof.  

\smallskip

\textbf{(f)} We will use throughout that, by (c) and (d), $M^\circ\unlhd \tilde{M}$ and $M^\circ\unlhd M_1\leq \tilde{M}$. Thus, $O_2(M^\circ)\unlhd M_1$ and $Q_1\cap M^\circ\unlhd M^\circ$, so $O_2(M^\circ)=Q_1\cap M^\circ$. As $M^\circ\unlhd M$, a similar argument yields $O_2(M^\circ)=Q_M \cap M^\circ$. As we have shown in (b) that $M=M^\circ Q_M$, it follows  
\[M^\circ Q_1/Q_1\cong M^\circ/O_2(M^\circ)\cong M/Q_M\cong \SL_3(2).\]
By (c), $\tilde{M}=M^\circ C_\L(V)$, so $M_1=M^\circ C_{M_1}(V)$. Recall that, by (b), $C_V(Q)=Z$ has order $2$ and $M\cap\tilde{C}=C_M(Z)$. Hence, $Z$ is normalized and thus centralized by $B$ and $C_{M^\circ}(Z)=M^\circ\cap\tilde{C}$. Therefore, $B\cap M^\circ=C_{M^\circ}(Z)=M^\circ\cap\tilde{C}$ and 
\[B\leq C_{M_1}(Z)=C_{M^\circ}(Z)C_{M_1}(V)=(M^\circ\cap\tilde{C})C_{M_1}(V)=(M^\circ\cap B)C_{M_1}(V).\]
So $B=(M^\circ\cap B)C_B(V)$. In particular, by (d), $M_1=M^\circ B=M^\circ C_B(V)$ and $C_{M_1}(V)=C_B(V)$ as $C_{M^\circ}(V)\leq Q_M\leq B$. In particular, $Q_1 \le Q_M\leq  C_{M_1}(V) = C_B(V)$. Note also that $M^\circ\cap C_B(V)\leq C_{M^\circ}(V)\leq O_2(M^\circ)\leq Q_1$. Moreover, by (c), $[M^\circ,C_B(V)] \le Q_1$. This implies (f).

\smallskip

\textbf{(g)} By (f), $B/Q_1=(B\cap M^\circ)Q_1/Q_1 \times C_B(V)/Q_1$ and $O_2(C_B(V))=O_2(C_{M_1}(V))\unlhd M_1$. This implies $O_2(C_B(V))\leq Q_1$ and then $O_2(B)=O_2(B\cap M^\circ)Q_1$. By (f), $M^\circ\cap B = C_{M^\circ}(Z)$ corresponds to a point stabilizer and so $(M^\circ\cap B)Q_1/Q_1\cong S_4$. In particular, $M^\circ\cap B$ acts irreducibly on $O_2(B)/Q_1$, and $O_2(B)/Q_1$ is the transvection group to the point $Z$. By hypothesis (iii), $V<Y_M$ and thus $[Y_M,O_2(B)]=V$ by \cite[1.1(b)]{MStr:2008}. 

\smallskip

As $M\cap\tilde{C}\leq \tilde{M}\cap B$ normalizes $O_2(B\cap M^\circ)$, it follows from (b) that $O_2(B\cap M^\circ)\leq O_2(M\cap \tilde{C})=QQ_M$. As $Q\unlhd B$, this implies $O_2(B\cap M^\circ)=Q(Q_M\cap M^\circ)\leq QO_2(M^\circ)\leq QQ_1$ and $O_2(B)\leq QQ_1$. As $B\leq \tilde{C}\cap M_1$, it follows $O_2(B)=QQ_1$. Since $B\leq M_2\leq \tilde{C}$, we have $Q\leq Q_2\leq O_2(B)$ and obtain $O_2(B)=QQ_1=Q_1Q_2$. 

\smallskip

\textbf{(h)} Note that $C_Q(V)\leq C_{M^\circ}(V)=M^\circ\cap C_M(V)=M^\circ\cap Q_M\leq Q_1$ by (f). Moreover, $Q_1\leq Q_M=C_M(V)$ and, by (g), $Q_2\leq QQ_1$. Hence, $C_{Q_2}(V)\leq C_{QQ_1}(V)=C_Q(V)Q_1=Q_1$.
\end{proof}

\begin{lemma}[Cf. {\cite[Lemma~2.2]{MStr:2008}}]\label{Lemma2.2}~
 \begin{itemize}
  \item [(a)] $L=O^2(L)=[L,Y_M]$ and $H=\<Y_M^L\>=\<Y_M^{M_2}\>$.
  \item [(b)] $W\neq V$, $[W,L]\neq 1$ and $C_{Q_2}(L)=C_{Q_2}(H)=C_{Q_2}(W)\leq Q_1$.
  \item [(c)] $[Q_2^\prime,L]=1$ and $[Q_2,L]\leq W$.
  \item [(d)] $W\unlhd M_2$, $W\leq Q\leq Q_2$, $WQ_1=O_2(B)$, $[Y_M,W]=V$, $V\cap Z(W)=Z$ and $WQ_M=QQ_M$.
  \item [(e)] $[W,Q_2]=W^\prime=Z=\Phi(W)$.
  \item [(f)] $[W,L]=W$ and $C_W(L)=Z(W)$; in particular, $Y_M \not \le W$ and $Y_M \not \le Q_2$.
  \item [(g)] $\widehat{W}$ is a selfdual, irreducible $\mF_2M_2$-module, which is homogeneous as an $\mF_2H$-module (i.e. a direct sum of irreducible isomorphic $\mF_2 H$-modules). 
 \end{itemize}
\end{lemma}

\begin{proof}
\textbf{(a)} Since $L$ is $M \cap \tilde{C}$-invariant and $S \le M \cap \tilde{C}$, we get  $O^2(L)Y_M\unlhd O^2(L)Y_M(S \cap L)=LY_M$ and $[L,Y_M]Y_M\unlhd LY_M$. We obtain
\[
O_2(O^2(L)Y_M)\leq O_2(LY_M) \quad \text{and} \quad O_2([L,Y_M]Y_M)\leq O_2(LY_M).
\]
Hence, $Y_M$ is neither contained in $O_2(O^2(L)Y_M)$ nor in $O_2([L,Y_M]Y_M)$. As $O^2(L)$ and $[L,Y_M]$ are also $M\cap \tilde{C}$-invariant, the minimal choice of $L$ yields $L=O^2(L)=[L,Y_M]$. In particular, $\<Y_M^L\>=[L,Y_M]Y_M=LY_M=H$. By Lemma~\ref{Lemma2.1}(e), $B\leq N_\L(Y_M)$ and so (a) follows.

\smallskip

\textbf{(b)} Assume $W=V$. Then $L\leq N_\L(V)=\tilde{M}\leq N_\L(Y_M)$ by Lemma~\ref{Lemma2.1}(e). Hence, $Y_M\leq O_2(LY_M)$ contradicting the choice of $L$. Thus $W\neq V$ and, in particular, $[W,L]\neq 1$. Set $D:=C_{Q_2}(L)$. Suppose $D\not\leq Q_1$. Since $B$ normalizes $D$ and acts irreducibly on $O_2(B)/Q_1$ by Lemma~\ref{Lemma2.1}(g), we get $DQ_1=O_2(B)$. Now Lemma~\ref{Lemma2.1}(g) together with \cite[1.1(b)]{MStr:2008} yields $[Y_M,D]=V$. So $V\leq D$ is centralized by $L$, contradicting $V\neq W$. Hence, $D\leq Q_1\leq Q_M$ and $D=C_{Q_2}(LY_M)=C_{Q_2}(H)=C_{Q_2}(\<Y_M^L\>)$ using (a). By Lemma~\ref{Lemma2.1}(e), $C_{Q_2}(V)=C_{Q_2}(Y_M)$ and so $C_{Q_2}(W)=C_{Q_2}(\<V^L\>)=C_{Q_2}(\<Y_M^L\>)=D$. This shows (b). 

\smallskip

\textbf{(c)} By Lemma~\ref{Lemma2.1}(g) $Q_2/(Q_1 \cap Q_2) \cong Q_2Q_1/Q_1$ is abelian; it therefore follows that $Q_2^\prime\leq Q_1\leq Q_M$ and hence $[Q_2^\prime,Y_M]=1$. By (a), $L\leq\<Y_M^L\>$ and hence $[Q_2^\prime,L]=1$. By Lemma~\ref{Lemma2.1}(b), $|Y_M/V|=2$. Hence, $[Y_M,Q_2]\leq V\leq W$ and thus $[Q_2,L]\leq [Q_2,\<Y_M^L\>]\leq W$. 

\smallskip

\textbf{(d)} Since $B$ normalizes $V$ and $L$, $B$ normalizes $W$. In particular, $W\unlhd M_2$. By Hypothesis (iii), $V\leq Q$ and so $W\leq Q\leq Q_2$. If $[W,V]=1$ then $W$ is abelian and thus, by (b), $W\leq C_{Q_2}(W)=C_{Q_2}(L)$ contradicting $W\neq V$. Hence, $[W,V]\neq 1$ and $W\not\leq Q_1$. Recall that $B$ normalizes $W$ and thus also $WQ_1$. Using Lemma~\ref{Lemma2.1}(g), the irreducible action of $B$ on $Q_2Q_1/Q_1$ gives $O_2(B)=WQ_1=Q_1Q_2=Q_1Q$; moreover $[Y_M,W]=V$ and $Z(W)\cap V=C_V(W)=Z=[V,W]$. In particular, we have $QQ_1=WQ_1 \le WQ_M$ and so $QQ_M \le WQ_M$ and so (d) is proven.

\smallskip

\textbf{(e)} By Lemma~\ref{Lemma2.1}(g) and part (d), $[W,V]=[Q_2,V]=Z$. Moreover, by hypothesis (iii), $V\leq Q\leq Q_2$. In particular, $Z\leq Q_2^\prime$ and so, by (c), $Z$ is centralized by $L$. Hence, as $W=\<V^L\>$, it follows $[W,W]=[Q_2,W]=Z$. As $W$ is generated by involutions, $W^\prime=\Phi(W)$ and so (e) holds. 

\smallskip

\textbf{(f)} By (b), $Z(W)=C_W(L)=C_W(H)$. By (a), $L=O^2(L)$ and $LY_M=H=\<Y_M^L\>$. Using that $Z=W'$ by (e), we see therefore that $Z(W)/Z=C_{W/Z}(L)=C_{W/Z}(\<Y_M^L\>)$. It follows from hypothesis (ii) and Lemma~\ref{Lemma2.1}(b) that $V/Z$ is a simple $M\cap\tilde{C}$-module. By (d), $[W,Y_M]=V$ and, by (e), $Z \norm M_2$. So it follows from Lemma~\ref{Lemma1.3} applied with $(M_2,M\cap\tilde{C},Y_M,W/Z,V/Z)$ in place of $(H,B,A,W,Y)$ that every proper $\mF_2M_2$-submodule of $W/Z$ lies in $Z(W)/Z$. If $[W,L]\leq Z(W)$, then $W=\<V^L\>=VZ(W)$ and $W$ is abelian, contradicting (d). Thus, $W=[W,L]Z$. So by (e), $Z=[W,W]=[W,[W,L]]\leq [W,L]$ and thus $W=[W,L]$. If $Y_M \le Q_2$, then (d) and (e) imply $V = [Y_M,W] \le [Q_2, W] = Z$, a contradiction. Thus, $Y_M\not\leq Q_2$. In particular, by (d), $Y_M\not\leq W$. So (f) is proved. 

\smallskip

\textbf{(g)} As remarked in the proof of (f), every proper $\mF_2M_2$-submodule of $W/Z$ lies in $Z(W)/Z$. This means that $\widehat{W}=W/Z(W)$ is an irreducible $\mF_2M_2$-module. Recall also that $W'=Z$ by (e). Thus, the commutator map 
\begin{center}
	\begin{tikzcd}
		 \widehat{W}\times\widehat{W} \ar[r] & Z \cong \bF_2 \\[-25pt]
		(xZ(W),yZ(W)) \ar[r,mapsto] & \left[x,y\right]
	\end{tikzcd}
\end{center}
is a non-degenerate bilinear form on $\widehat{W}$ and so $\widehat{W}$ is a selfdual $\mF_2M_2$-module. By (a), $H$ is normal in $M_2$. Thus, by Clifford's Theorem, $\widehat{W}$ is semisimple as $\bF_2 H$-module and a direct sum of homogenous $\bF_2H$-modules $\widehat{W}_i$, $1\leq i\leq n$ such that $M_2$ acts transitively on $\{\widehat{W}_1,\dots,\widehat{W}_n\}$. As $\widehat{W}=\bigoplus_{i=1}^n\widehat{W}_i$, we have then $\widehat{V}=[\widehat{W},Y_M]=\bigoplus_{i=1}^n[\widehat{W}_i,Y_M]$. If $i>1$, then the action of $B$ on $\widehat{V}^\#$ is imprimitive. However, $\widehat{V}\cong V/V\cap Z(W)=V/Z$ as a $B$-module. It follows thus from Lemma~\ref{Lemma2.1}(f) that $|\widehat{V}|=4$ and $B$ is transitive on $\widehat{V}^\#$. In particular, the action of $B$ on $\widehat{V}^\#$ is primitive. Hence, $i=1$ and $\widehat{W}$ is a homogeneous $\mF_2H$-module.
\end{proof}

\begin{lemma}[Cf. {\cite[Lemma~2.3]{MStr:2008}}]\label{Lemma2.3}~
 \begin{itemize}
  \item [(a)] $W$ acts quadratically on $Q_M/V$. In particular, any non-trivial composition factor for $M$ on $Q_M/V$ is a natural $\SL_3(2)$-module. 
  \item [(b)] $N_{M_2}(Y_MQ_2)=N_{M_2}(V)=B$.
  \item [(c)] If $g\in L$ is such that $[Y_M,Y_M^g]\leq Q_2$, then $[Y_M,Y_M^g]=1$ and $Y_MY_M^g$ acts quadratically on $Q_2$ and $\widehat{W}$.
  \item [(d)] $V$ is the unique minimal normal 2-subgroup of $M$ and of $M_1$.
  \item [(e)] $C_{M_2}(\widehat{W})=Q_2$.
  \item [(f)] Let $U$ be an irreducible $H$-submodule of $\hat{W}$. Then $\hat{W}=U$ or $\hat{W}\cong U\oplus U$ as an $H$-module. Moreover, we have $\hat{W}\cong U\oplus U$ if and only if $|U/C_U(\ov{Y_M})|=2$. 
 \end{itemize}
\end{lemma}

\begin{proof}
\textbf{(a)} Note that $Q_M \le S\leq  M \cap \tilde{C}$, and so $Q_M$ normalizes both $V$ and $L$. Thus, $Q_M$ normalizes $W$ and, using Lemma~\ref{Lemma2.2}(e), we obtain $[Q_M,W,W]\leq [W,W]=Z\leq V$. So $W$ acts quadratically on $Q_M/V$ and thus also on every non-trivial composition factor for $M$. 

\smallskip

By Lemma~\ref{Lemma2.1}(f),(g), $Q_M\cap M^\circ=Q_1\cap M^\circ$ and $O_2(B)/Q_1$ has order $4$. As $W\leq Q\leq M^\circ$ and $O_2(B)=WQ_1$ by Lemma~\ref{Lemma2.2}(d), it follows that $W\cap Q_M=W\cap Q_1$ and $WQ_M/Q_M\cong WQ_1/Q_1$ has order $4$. Therefore, $W$ does not act quadratically on the Steinberg module for $\SL_3(2)$. Since the only simple $\mF_2\SL_3(2)$-modules are the trivial module, the Steinberg module, and the two natural modules (which are dual to each other), we have (a). 

\smallskip

\textbf{(b)} By Lemma~\ref{Lemma2.2}(d),(e), $N_{M_2}(Y_MQ_2)$ normalizes $[W,Y_MQ_2]=[W,Y_M]=V$. As $L\cap \tilde{M}\leq B\leq \tilde{M}=N_\L(V)$ and $M_2=LB$, it follows that $N_{M_2}(V)=M_2\cap \tilde{M}=LB \cap \tilde{M} = (L \cap \tilde{M})B=B$. Using Lemma~\ref{Lemma2.1}(e), we have $B\leq N_{M_2}(Y_MQ_2)$. This implies (b).

\smallskip

\textbf{(c)} Let $g\in L$  and note that $Y_M^g$ is defined in $M_2\leq \tilde{C}$. Suppose that $[Y_M,Y_M^g]\leq Q_2$. Then by (b), $Y_M^g\leq \tilde{M}$. By Lemma~\ref{Lemma2.1}(c),(e), $\tilde{M}=MC_\L(V)=MC_\L(Y_M)$. Thus, $[Y_M,Y_M^g]\leq [Y_M,M]=V$. Note that $[Y_M^{g^{-1}},Y_M]\leq Q_2^{g^{-1}}=Q_2$, so the same argument with $g^{-1}$ in place of $g$ gives $[Y_M,Y_M^{g^{-1}}]\leq V$. Hence conjugation by $g$ in the group $\tilde{C}$ gives $[Y_M^g,Y_M]\leq V^g$. Thus, $R:=[Y_M,Y_M^g]\leq V\cap V^g$. Suppose that $R\neq 1$. Then by Lemma~\ref{Lemma2.1}(e), $[V,Y_M^g]\neq 1$ and $g\not\in\tilde{M}=N_\L(V)$. In particular, $R\leq V\cap V^g$ is a proper subgroup of $V$. By \cite[1.1(b)]{MStr:2008} (applied to $V_1=Y_M$), $Y_M^g$ lies in the transvection group to a hyperplane of $V$. Then the proof of \cite[1.1(b)]{MStr:2008} shows that $[Y_M,t]=C_V(t)$ has order $4$ for every element $t\in Y_M^g\backslash Q_M$. Hence, $|R|=4$. Note that $V^g\leq W\leq S$ and thus $V^g\in\Delta$. Hence, by Lemma~\ref{LocalitiesProp}(b), $c_g\colon \tilde{M}\rightarrow \tilde{M}^g=N_\L(V^g)$ is a group isomorphism. In particular, by Lemma~\ref{Lemma2.1}(c),(d), $\tilde{M}^g=M^gC_\L(V^g)$, $V^g$ is a natural $\SL_3(2)$-module for $M^g/Q_M^g$ and $C_\L(V^g)=C_\L(Y_M^g)$. The latter property together with $[V,Y_M^g]\neq 1$ yields $[V,V^g]\neq 1$. Considering the action of $\tilde{M}^g$ on $V^g$, one sees that there exists $x\in R^\sharp$ such that $V$ is not in the transvection group to the point $\<x\>\leq V^g$, that is
\begin{equation*}
	\label{eq:lemma 4.6:V in transv gr}
	\tag{$\dagger$}
	V\not\leq O_2(C_{\tilde{M}^g}(x)).
\end{equation*}
Note that $x\in V$, and recall that, by Lemma~\ref{Lemma2.1}(g), $Q$ induces on $V$ the transvection group to a point in $V$. So looking now at the action of $M$ on $V$, it follows that there exists $m\in M$ with $[x,Q^m]=1$. As $C_M(Q)\leq C_\L(Q)\leq Q$, it follows $C_M(Q^m)=Z(Q^m)$. Then $\<x^{m^{-1}}\>\in Z(Q)$. As $(\L,\Delta,S)$ is $Q$-replete, we have $\<x^{m^{-1}}\>\in\Delta$. In particular, $C_\L(x^{m^{-1}})\leq N_\L(\<x^{m^{-1}}\>)$ is a subgroup of $\L$. The property \eqref{Q!L} gives $Q\leq O_2(C_\L(x^{m^{-1}}))$. Hence, $\<x\>\in\Delta$ and $Q^m\leq O_2(C_\L(\<x\>))$ by Lemma~\ref{LocalitiesProp}(b). By hypothesis (iii), $V\leq Q$ and thus $V\leq Q^m$. This yields $V\leq O_2(C_\L(x))\cap\tilde{M}^g\leq O_2(C_{\tilde{M}^g}(x))$, a contradiction to (\ref{eq:lemma 4.6:V in transv gr}). Hence, $R=1$. In particular, $Y_MY_M^g$ is abelian. Since $Q_2$ normalizes $Y_MY_M^g$ it follows that $[Q_2,Y_MY_M^g,Y_MY_M^g]\leq [Y_MY_M^g,Y_MY_M^g]=1$. This implies (c).

\smallskip

\textbf{(d)} Recall that $\Omega_1(Z(S))=Z$ has order $2$ by Lemma~\ref{Lemma2.1}(b). If $K$ is a normal $2$-subgroup of $M$ or of  $M_1$, then $K\cap \Omega_1(Z(S))=1$ and so $Z\leq K$. Hence, $V=\<Z^M\>\leq K$.

\smallskip

\textbf{(e)} As $\widehat{W}$ is an irreducible $M_2$-module by Lemma \ref{Lemma2.2}(g), we have $Q_2 \le C_{M_2}(\widehat{W})$. 
Set  
\[E:=O^2(C_{M_2}(\widehat{W})).\]
Recall that $L \norm M_2$, so that $F:=L \cap E \norm M_2$. Note that, by Lemma \ref{Lemma2.2}(f), $L \not \le E$ and so $F < L$. Hence, the minimality of $L$ yields $Y_M \le O_2(FY_M)$. In particular
\[
[F,Y_M] \le O_2(FY_M) \cap F = O_2(F) \le Q_2.
\]
By Lemma \ref{Lemma2.2}(a), we have $H=\<Y_M^{M_2}\>$. Hence, we get $[F, H] \le Q_2$, so that $\ov{F} \le Z(\ov{H})$. As $H$ is normal in $M_2$, we have $O_2(\ov{H})=1$ and thus $Z(\ov{H})$ has odd order. This implies $O_2(\ov{FY_M}) = \ov{Y_M}$. Now, since $H/L$ is abelian we have $[H,Y_M] \le [H,H] \le L$, which yields
\[
[H \cap E, Y_M] \le E \cap [H,Y_M] \le E \cap L = F.
\]
Another application of Lemma \ref{Lemma2.2}(a) provides $[H \cap E, H] \le F$, showing that $H \cap E$ normalizes $FY_M\leq H$ and thus also $O_2(\ov{FY_M}) = \ov{Y_M}$. This implies that $H \cap E$ (and thus also $F\leq H\cap E$) normalizes $Y_MQ_2$. So
\begin{equation*}
[H \cap E, Y_M] \le F \cap Y_MQ_2 \le O_2(F) \le Q_2.
\end{equation*}
Applying again Lemma \ref{Lemma2.2}(a), it follows $[H \cap E, H] \le Q_2$. In particular, $\ov{H \cap E} \le Z(\ov{H})$ has odd order and  $O_2(\overline{(H \cap E)Y_M}) = \overline{Y_M}$. As $E$ and $H$ are normal in $M_2$, we get
\[ [(H \cap E)Y_M,E] \le [H,E] \le H\cap E.
\]
Thus, $E$ normalizes $(H \cap E)Y_M$ and thus also $\ov{Y_M}=O_2(\ov{(H \cap E)Y_M})$. Using part (b) we obtain thus $E\leq N_{M_2}(Y_MQ_2) = B$. In particular, $E$ normalizes $V$.

\smallskip

Using Lemma \ref{Lemma2.2}(d) and the fact that $E$ acts by definition trivially on $\widehat{W}$, we get now $[V,E] \le V \cap Z(W) = Z$. As $E = O^2(E)$, this yields $[V,E]=[V,E,E]=1$. Hence, using $C_M(V)=Q_M$ and Lemma~\ref{Lemma2.1}(g), we see that
\[
[M^\circ,E] \le C_{M^\circ}(V)=Q_M\cap M^\circ \le Q_1.
\]
Hence, $M^\circ$ normalizes $E Q_1$. Note that $Q_1\leq S$ normalizes $E$ and so $O^2(EQ_1) = O^2(E) = E$. Thus, $M^\circ$ normalizes $E$.

\smallskip

Seeking a contradiction, suppose $E \ne 1$. As $Q$ is large in $\L$ we have $C_\L(Q) \le Q$, so $1 \ne [E,Q] \le E \cap Q \le O_2(E)$. Since  $M_1=M^\circ B$ normalizes $E$, we get $O_2(E) \le Q_1$. As $V$ is by (d) the unique minimal normal 2-subgroup in $M_1$, we have $V \le Z(O_2(E))$. Since $M_2$ normalizes $E$, it follows $W =\< V^L \>\le Z(O_2(E))$. So $W$ is abelian, a contradiction to Lemma \ref{Lemma2.2}(e). Thus $E = 1$, which implies $C_{M_2}(\widehat{W}) = Q_2$.

\smallskip

\textbf{(f)} Note that, by Lemma \ref{Lemma2.2}(d), $|[\widehat{W},Y_M]| = |\widehat{V}| =4$. Recall also that $\widehat{W}$ is $H$-homogeneous by Lemma~\ref{Lemma2.2}(g). Suppose now that $\widehat{W} = W_1 \oplus \dots \oplus W_k$ is a splitting into irreducible isomorphic $H$-modules. Then 
\[
[[\widehat{W},Y_M]=W_1 \oplus \dots \oplus W_k, Y_M] = [W_1,Y_M] \oplus \dots \oplus [W_k,Y_M]
\]
has order $|[W_1,Y_M]|^k$. Hence, $k\leq 2$, and $k=2$ if and only if $|[W_1,Y_M]|=2$. As $Y_M$ acts by Lemma~\ref{Lemma2.1}(b) and part (g) as an involution on $\hat{W}$, we have $|W_1/C_{W_1}(Y_M)|=|[W_1,Y_M]|$. As every irreducible $H$-submodule of $\hat{W}$ is isomorphic to $W_1$, this implies (f).	
\end{proof}

\begin{lemma}\label{Lemma2.4}
	Set $E= [Q_1,O^2(M)]$. Then also $E=[Q_M, O^2(M)]$. Moreover, if $E \le Y_MQ_2$ or $[E,W] \le V$, then $Y_M = Q_M$.
\end{lemma}

\begin{proof}
By Lemma \ref{Lemma2.1}(b) we have $M=M^\circ Q_M$. Hence, $O^2(M^\circ) = O^2(M)$. By Lemma \ref{Lemma2.1}(f) we get
\[
[Q_M, O^2(M)] = [Q_M, O^2(M^\circ)] \le Q_M\cap M^\circ \le Q_1.
\]
As $Q_1 \le Q_M$, it follows now from the properties of coprime action that
\[
E \le [Q_M, O^2(M)] = [Q_M, O^2(M), O^2(M)] \le [Q_1, O^2(M)] = E.
\]
So all the inequalities in the statement above are equalities. This proves the first statement.

\smallskip

If $E \le Y_MQ_2$, then Lemma \ref{Lemma2.2}(d),(e) yields $[E,W] \le [Y_MQ_2, W] \le V$. Thus, for the proof of the second statement, we may assume that $[E,W] \le V$ holds. Our argument above gives in particular that $E\leq Q_M$. It follows from Lemma \ref{Lemma2.3}(a) that $W$ acts non-trivially on every non-trivial composition factor of $M$ in $Q_M/V$. Hence, there is no such non-trivial composition factor in $EV/V$ and thus $O^2(M)$ acts trivially on $EV/V$. In particular 
\[E = [Q_M, O^2(M)] = [Q_M, O^2(M), O^2(M)] = [E, O^2(M)] \le V.\]
Note that $O^2(M)$ cannot act trivially on $Q_M/\Phi(Q_M)$, otherwise it would act trivially on $Q_M$, a contradiction. This shows that $V \not \le \Phi(Q_M)$. Since $V$ is by Lemma \ref{Lemma2.3}(d) the unique minimal normal $2$-subgroup of $M$, we get $\Phi(Q_M) = 1$. So $Q_M$ is elementary abelian. Then $C_M(Q_M)=Q_M$, showing that $Q_M$ is $2$-reduced for $M$. Thus, $Q_M = Y_M$.
\end{proof}

The next two lemmas were not stated in a similar way in \cite{MStr:2008}, but are nevertheless useful when following a similar line of argument.

\begin{lemma}\label{Minimality of L}
Let $L_0 < L$ and suppose that $L_0$ is $M \cap \tilde{C}$-invariant. Then the following hold.
\begin{itemize}
\item[(a)] $\ov{Y_M} \le O_2(\ov{L_0Y_M})$.
\item[(b)] If $\ov{L_0} = E(\ov{L_0})$ or $\ov{L_0}$ is odd, then $[\ov{L_0}, \ov{Y_M}] = 1$.
\end{itemize}
\end{lemma}

\begin{proof}
By the minimality of $L$ we have $Y_M \le O_2(L_0Y_M)$, implying $\ov{Y_M} \le O_2(\ov{L_0Y_M})$, which is (a). Now note that $\ov{L_0} \norm \ov{L_0Y_M}$. Since every component of $\ov{L_0}$ is a component of $\ov{L_0Y_M}$, if $\ov{L_0} =E(\ov{L_0})$ we get $[\ov{L_0}, \ov{Y_M}]=1$. Similarly, if $\ov{L_0}$ is odd, then $[\ov{Y_M}, \ov{L_0}]  = 1$. This shows (b).
\end{proof}

\begin{lemma}\label{L:F*ovL}
Either $\ov{L} = E(\ov{L})$ or $\ov{L}=F(\ov{L})$ is an $r$-group for some odd prime $r$. Moreover, $F^*(\ov{H})=\ov{L}$.
\end{lemma}

\begin{proof}
By Lemma \ref{Lemma2.2}(a) $\ov{H}=\ov{LY_M} \norm \ov{M_2}$ and so $O_2(\ov{LY_M}) = 1$. This implies $F^*(\ov{L}) = F^*(\ov{H})$ as $\ov{L}$ is a normal subgroup of $\ov{H}$ of $p$-power index. In particular, $[\ov{Y_M}, F^*(\ov{L})] \neq 1$, as we would other have that $\ov{Y_M} \le C_{\ov{H}}(F^*(\ov{L}))=Z(F^*(\ov{L}))$, a contradiction to $O_2(\ov{L})=1$.

\smallskip

Assume now that both the layer $E(\ov{L})$ and the fitting subgroup $F(\ov{L})$ are proper subgroups of $\ov{L}$. Note that $F(\ov{L})$ has odd order, as $L \norm M_2$ and we are quotienting over $Q_2=O_2(M_2)$. Thus, by Lemma~\ref{Minimality of L}(b), $[\ov{Y_M},E(\ov{L})] = 1 = [\ov{Y_M},F(\ov{L})]$ contradicting $[\ov{Y_M}, F^*(\ov{L})] \neq 1$.  Thus, $\ov{L} = E(\ov{L})$ or $\ov{L} = F(\ov{L})$. In particular, $\ov{L}=F^*(\ov{L})=F^*(\ov{H})$. 

\smallskip

If $\ov{L}=F(\ov{L})$ and $\ov{L}$ is not an $r$-group for some odd prime $r$, then Lemma~\ref{Minimality of L}(b) applied with $\ov{L_0}=O_r(\ov{L})$ yields that $[O_r(\ov{L}),\ov{Y_M}]=1$ for every odd prime $r$. Thus $[F^*(\ov{L}),\ov{Y_M}]=[\ov{L},\ov{Y_M}]=1$, a contradiction as above.
\end{proof}

\subsection{The proof that $L$ is solvable}

We will now show in a series of lemmas that the group $L$ is solvable.

\begin{lemma}[{Cf. \cite[Lemma~2.5]{MStr:2008}}]\label{Lemma2.5}
Suppose that $L$ is non-solvable and let $U$ be an irreducible $L$-submodule of $\widehat{W}$. Then the following hold:
\begin{itemize}
    \item[(a)] $\ov{L}$ is quasisimple and $\ov{L}=F^*(\ov{H})$;
    \item[(b)] $H$ normalizes $U$ and $U$ is a self-dual $H$-module.
\end{itemize}
\end{lemma}

\begin{proof}
\textbf{(a)} By Lemma \ref{L:F*ovL}, we have $\ov{L} = E(\ov{L})=F^*(\ov{H})$. Write $\mathfrak{C}$ for the set of components of $\ov{L}$ 
and pick $L_1 \in \mathcal{C}$. Note that $\< L_1 ^B \>$ is $M \cap \tilde{C}$-invariant. Hence, if $\< L_1 ^B \>$ is proper in $\ov{L}$, we may apply Lemma \ref{Minimality of L}(b) and obtain $[\< L_1^B \>, \ov{Y_M}] =1$. The same argument applies then to any component of $\ov{L}$ which is not a $B$-conjugate of $L_1$, so that $[\ov{L}, \ov{Y_M}]=1$. This contradicts Lemma~\ref{Lemma2.2}(a). Thus $\ov{L} = \< L_1 ^B \>$.

\smallskip

Note that $\ov{Y_M} \le Z(\ov{S})$  as $|\ov{Y_M}|=2$ by Lemma~\ref{Lemma2.1}(b). Hence, if there is $t \in \ov{Y_M}$ such that $L_1^t \ne L_1$, then $L_1$ and $L_1^t$ share the 2-Sylow $L_1 \cap \ov{S}$. As the intersection of any two distinct components is central in both and $|L_1/Z(L_1)|$ is even by the Feit-Thompson theorem, this yields a contradiction. Hence, $L_1$ is normalized by $Y_M$. As $L_1$ was arbitrary, this means that $Y_M$ normalizes every element of $\mathfrak{C}$. As $M_2$ acts on $\mathfrak{C}$ by conjugation, we obtain that every $M_2$-conjugate of $Y_M$ normalizes every element of $\mathfrak{C}$. By Lemma \ref{Lemma2.2}(a) $H=\<(Y_M)^{M_2} \>$. Hence, all the components of $\ov{L}$ are normal in $\ov{H}$.

\smallskip

Recall that by Lemma \ref{Lemma2.2}(g) and \ref{Lemma2.3}(e) $\widehat{W}$ is an irreducible and faithful $\bF_2 \ov{M_2}$-module. Now let $U$ be a non-trivial irreducible $\bF_2 L_1$-submodule of $\widehat{W}$. Then $|U| > 4$ since $L_1$ is non-solvable and acts faithfully on $\widehat{W}$. Take $ y \in Y_M$. Note that $C_W(y) \ge W \cap Q_M$ and $|W /(W \cap Q_M)|=|WQ_M / Q_M|  = 4$ by Lemma~\ref{Lemma2.1}(g) and Lemma~\ref{Lemma2.2}(d). Hence, $|\widehat{W} / C_{\widehat{W}}(y)| \leq  |W / C_W(y)| \le 4$ and so $1\neq C_U(y) \leq U \cap U^y$. As we have shown that $Y_M$ normalizes $L_1$, it follows that $L_1^y = L_1$ normalizes $U \cap U^y$. Since $U$ is an irreducible $\bF_2 L_1$-module, this implies $U = U^y$ showing that $Y_M$ normalizes $U$. Since $L_1\unlhd H$, every $H$-conjugate of $U$ is an irreducible $\bF_2 L_1$-submodule. Thus, it follows likewise that $Y_M$ normalizes every $H$-conjugate of $U$, and thus every $H$-conjugate of $Y_M$ normalizes $U$. Since $H = \< Y_M^H \>$ by Lemma~\ref{Lemma2.2}(a), we can conclude that $H$ normalizes $U$ and thus every non-trivial irreducible $\bF_2 L_1$-submodule of $\widehat{W}$.

\smallskip

As $U$ is an irreducible $L_1$-module, the ring $\End_{L_1}(U)$ of $\bF_2 L_1$-linear endomorphisms is a division ring by Schur's lemma, and then a field as it is finite. Note that we obtain a map $j:C_H(L_1) \fto \End_{L_1}(U)^*$, where $\End_{L_1}(U)^*$ denotes the set of invertible elements. As multiplication in a field is commutative, we have $C_H(L_1)^\prime \le ker(j)$, which means that $C_H(L_1)^\prime$ centralizes $U$. Since $\widehat{W}$ is a homogeneous $H$-module and $H$ normalizes all the non-trivial, irreducible $L_1$-submodules, it follows that $C_H(L_1)^\prime$ centralizes $\widehat{W}$. In particular, by Lemma \ref{Lemma2.3}(e), $C_{\ov{H}}(L_1)^\prime=1$, so $C_{\ov{L}}(L_1)$ is abelian. As a consequence, $L_1$ is the only component of $\ov{L}$ and $\ov{L}=E(\ov{L})=L_1$. This shows (a).



\smallskip

\textbf{(b)} By Lemma~\ref{Lemma2.2}(g), $\hat{W}$ is a self-dual $M_2$-module and thus also a self-dual $H$-module; this was shown using the non-degenerate $\bF_2$-bilinear form on $\hat{W}$ induced by the commutator map. It follows that $U$ is also a self-dual $\bF_2H$-module. 
\end{proof}

\begin{lemma}\label{Lemma2.6}
	Suppose that $L$ is non-solvable and let $U$ be an irreducible $H$-submodule of $\widehat{W}$. Then, up to isomorphism, one of the following holds.
	\begin{itemize}
		\item[(1)] $\ov{H} \in \{ \SL_n(2), \SL_n(4), \Sp_{2n}(2), \Sp_{2n}(4), \SU_{n}(2), \Omega_{2n}^{\pm}(2),\G_2(2)^\prime \}$ and $U$ is the corresponding natural module.
		\item[(2)] $\ov{H} \cong \Sp_6(2)$, $U$ is the spin module and $\ov{Y_M}$ is a short root subgroup of $\ov{H}$.
		\item[(3)] $\ov{H} \in \{ \Sym(n), \Alt(n) \}$ for $n \ge 5$ and $U$ is the natural permutation module; moreover $\ov{L} \cong \Alt(n)$. If $\ov{H} \cong \Sym(n)$, then $\ov{Y_M}$ is generated by a transposition, else if $\ov{H} \cong \Alt(n)$, then $\ov{Y_M}$ is generated by a double transposition.
		\item[(4)] $\ov{H} \cong \Alt(7)$ and $U$ is a spin module.
		\item[(5)] $\ov{H} \cong 3 \cdot \Alt(6)$ and $U$ is a module of dimension 6.
	\end{itemize}
\end{lemma}

\begin{proof}
By Lemma \ref{Lemma2.5}, $\ov{L} = F^*(\ov{H})$ is quasisimple. Wishing to apply \cite[Lemma 1.2]{MStr:2008}, we check that the hypothesis are satisfied.

\smallskip

As $\widehat{W}$ is a faithful and homogeneous $\ov{H}$-module, $C_{\ov{H}}(U)=1$. Since $Y_M\unlhd S$ and $|\ov{Y_M}| = 2$ by Lemma~\ref{Lemma2.1}(b), we have  $\ov{Y_M} \le Z(\ov{S} \cap \ov{H})$.

\smallskip

Recall that $[Y_M,L] = L$ by Lemma~\ref{Lemma2.2}(a). As $O(\ov{H})=O(\ov{L})=Z(\ov{L})$, it follows that $\ov{Y_M} \not \le Z^*(\ov{H})$. Glauberman's $Z^*$-theorem now implies the existence of an element $g \in L$ such that $\ov{Y_M} \not = \ov{Y_M}^g\leq \ov{S} \cap \ov{H}$ and $[\ov{Y_M},\ov{Y_M}^g] =1$. Now Lemma \ref{Lemma2.3}(c) shows that $Y_MY_M^g$ induces a quadratic fours group on $U$. At the same time, by Lemma \ref{Lemma2.2}(d), $[U, Y_M] \le \widehat{V}$ has order at most 4.	Hence we can apply \cite[Lemma 1.2]{MStr:2008}, which provides us directly with the statement.
\end{proof}

\begin{lemma}\label{Lemma2.7}
Let $U$ be an irreducible $H$-submodule of $\widehat{W}$. Suppose that $\ov{L} \cong \Alt(n)$ for $n \ge 5$ and $U$ is the natural permutation module for $\ov{L}$. Then $n=6$ and $\ov{H}\cong \Sym(6)$.
\end{lemma}

\begin{proof}
By contradiction, suppose $U$ is the natural permutation module for $\ov{L}$. Then we fall in case (3) of Lemma \ref{Lemma2.6}. So, up to a suitable conjugation, $\ov{Y_M}$ corresponds either to $(12)(34)$ or to $(12)$. Recall that, by Lemma \ref{Lemma2.3}(c), if $g \in L$ is such that $[Y_M,Y_M^g] \le Q_2$, then $Y_MY_M^g$ induces a quadratic fours group on $\widehat{W}$. Since this holds for all choices of $g$ and $\<(12)(34),(13)(24)\>$ does not act quadratically on the natural permutation module, $\ov{Y_M}$ is indeed conjugate to $\<(12)\>$. In particular, $\ov{H}\cong \Sym(n)$ and $\ov{L}<\ov{H}$. 

\smallskip

Since $\ov{Y_M}$ is 2-central in $\ov{H}$ and the only 2-central elements in $\Alt(5)$ are double-transpositions, we have $n \ne 5$. Assume now $n > 6$.

\smallskip

It follows from Lemma~\ref{Lemma2.1}(e) that $Y_M\leq Q_1$ and thus $\ov{L}<\ov{H}\leq \ov{LQ_1}$. Thus Lemma \ref{charact-of-symm-gr} yields $\ov{LQ_1} \cong \Sym(n)$ and so $\ov{LQ_1}=\ov{H}$. Using Lemma \ref{Lemma2.3}(b) we obtain
\[\ov{B\cap LQ_1} = C_{\ov{LQ_1}}(\ov{Y_M}) \cong \ov{Y_M} \times \Sym(n-2).\]
From $n-2 > 4$ we get $O_2(\Sym(n-2)) =1$, so $Q_1 \le O_2(B) \cap Q_1L \le O_2(B \cap Q_1L) \le Y_MQ_2$ using Lemma~\ref{Lemma2.1}(g). It follows that $[Q_1,W] \le [Y_MQ_2, W] \le V$ by Lemma~\ref{Lemma2.2}(d),(e). In particular we can apply Lemma \ref{Lemma2.4} and obtain $Y_M=Q_M$. Then also $Y_M=Q_1$ by Lemma~\ref{Lemma2.1}(e). Recall that, by Lemma~\ref{Lemma2.1}(a),(g) we have that $M \cap \tilde{C} = C_M(Z)$ and $Q_1Q_2$ corresponds to the transvection group to the point $Z$ in the action on $V$. As $C_M(Z) / Q_M \cong \Sym(4)$, we obtain $|S/Y_MQ_2|  = 2$, which contradicts $\ov{B \cap LQ_1}/\ov{Y_M} \cong \Sym(n-2)$. This shows $n=6$.
\end{proof}

\begin{definition}
	Let $H$ be a finite group, $U$ an $\mathbb{F}_pH$-module and $T \in \Syl_p(H)$. Then the subgroup
	\[
	P_H(T,U) := O^{p^\prime}(C_H(C_U(T)))
	\]
	is called a \emph{point stabilizer} of $H$ on $U$ with respect to $T$.
\end{definition}


\begin{lemma}\label{Lemma2.8}
Let $U$ be a simple $H$-submodule of $\widehat{W}$.	Suppose that $\ov{H} \in \{\Omega_{2n}^\pm (2), \Sp_{2n}(2) \}$ and that $U$ is the corresponding natural module. Then $\ov{H} \cong \Sp_{2n}(2)$ 
and $\ov{Y_M}$ induces a transvection on $U$.
\end{lemma}

\begin{proof}
Set $\ov{P} := P_{\ov{H}}(\ov{S} \cap \ov{H},U)$ and let $P$ be the preimage of $\ov{P}$ in $H$. Then one sees from the structure of $(\ov{H},U)$ that $O_2(\ov{P})$ is abelian and $C_{\ov{P}}(O_2(\ov{P}))\leq O_2(\ov{P})$. Note that $[\ov{Y_M},\ov{S}]=1$. Hence $\ov{Y_M} \le \ov{P}$ and then $\ov{Y_M}\leq C_{\ov{P}}(O_2(\ov{P}))\leq O_2(\ov{P})$. Thus, $\<(\ov{Y_M})^{\ov{P}}\>\leq O_2(\ov{P})$ is abelian and Lemma \ref{Lemma2.3}(c) implies that $A:=\<Y_M^P\>$ is abelian. Note that $W$ normalizes $Y_M$ and thus $A$. Hence, $[W,A,A]\leq [A,A]=1$ and $A$ acts quadratically on $\hat{W}$. This is only possible if $\ov{H} \cong \Sp_{2n}(2)$ and $\ov{Y_M}$ induces a transvection on $U$. 
\end{proof}

\begin{lemma}\label{Lemma2.9}
The group $\ov{H}$ is not isomorphic to any of the groups $\SL_n(2)$ for $n\geq 3$, $\SL_n(4)$ for $n\geq 3$, $3\cdot \Alt(6)$, $\Alt(7)$.
\end{lemma}

\begin{proof}
Suppose this is false. Then $\ov{H}$ does not contain a subgroup of index $2$ and so $\ov{L}=\ov{H}$. Let $U$ be, as usual, a simple $H$-submodule of $\widehat{W}$. By Lemma \ref{Lemma2.5}, $U$ is a self-dual $H$-module.

\smallskip

If $n \ge 3$, the natural module for $\SL_n(q)$ is not self-dual. Similarly, also the 6-dimensional module for $3 \cdot \Alt(6)$ and the spin module for $\Alt(7)$ are not self-dual. 

\smallskip

Inspecting now the list in Lemma~\ref{Lemma2.6} and considering isomorphisms between the groups listed there, the only possibility that remains is that $\ov{H}\cong \SL_4(2) \cong \Omega_6^+(2)$ and that $U$ is the natural module for $\Omega_6^+(2)$. However, in this case, Lemma \ref{Lemma2.8} yields a contradiction.
\end{proof}


We now wish to use the gathered information about $\ov{H}$ to study the structure of $M_1/Q_1$, $\ov{M_2}$ and $B/O_2(B)$ in the case that $L$ is not solvable. Recall that we defined
\[T = O^2(M \cap \tilde{C}).\]

\begin{lemma}[{Cf. \cite[Lemma~2.10]{MStr:2008}}]\label{Lemma2.10}~
\begin{itemize}
\item [(a)] We have $T=O^2(B\cap M^\circ)$. If $F \norm B$ such that $[V/Z,F] \ne 1$, then $T  \le F$.
\item [(b)] If $[V/Z, L \cap B] \ne 1$, then $T \le L \cap B$ and $M_2 =LS$.
\end{itemize}
\end{lemma}

\begin{proof}
\textbf{(a)} By definition of $B$ we have $B\cap M=M\cap \tilde{C}$ and so $B\cap M^\circ=\tilde{C}\cap M^\circ$.  As $M=M^\circ Q_M$, we obtain $\tilde{C}\cap M=(\tilde{C}\cap M^\circ)Q_M=(B\cap M^\circ)Q_M$. Hence, 
\[T=O^2(B\cap M^\circ).\]
By Lemma \ref{Lemma2.1}(f), we have $B=(B \cap M^\circ )C_B(V)$ with $B \cap M^\circ =C_{M^\circ}(Z)$, where $V$ is a natural $\SL_3(2)$-module for $M^\circ/O_2(M^\circ)=M^\circ/C_{M^\circ}(V)$ and $Z$ is a point. Hence, $B/C_B(V)\cong \Sym(4)$ corresponds to a point stabilizer and 
\[O_2(B/C_B(V))=C_B(V/Z)/C_B(V).\]
Thus, the image of $R:=[F,B \cap M^\circ]$ in $B/C_B(V)$ contains $O^2(B/C_B(V))\cong \Alt(4)$. In particular, $T\leq O^2(B)\leq RC_B(V)$. 
Recall that $M^\circ\unlhd \tilde{M}\geq B$ by Lemma~\ref{Lemma2.1}(c) and so $R\leq M^\circ$. Hence, $T=O^2(B\cap M^\circ)\leq RC_B(V)\cap M^\circ=RC_{M^\circ}(V)=RO_2(M^\circ)$ and so $T\leq R\leq F$. 

\smallskip

\textbf{(b)} Suppose now that $[V/Z, L \cap B] \ne 1$. Applying (a) with $F=L \cap B$, we obtain then $ T \le L \cap B\leq L$. Hence, $M_2 = LB = L(\tilde{C}\cap M) = LTS = LS$.
\end{proof}

The following lemma corresponds basically to \cite[Lemma~2.11]{MStr:2008}. However, we exclude the third case listed there immediately using an argument given in the proof of \cite[Lemma~2.12]{MStr:2008}.

\begin{prop}[{Cf. \cite[Lemma~2.11]{MStr:2008}}]\label{Lemma2.11}
Suppose that $L$ is non-solvable. Then one of the following holds.
\begin{enumerate}[itemsep=3pt]
 \item For some $n\geq 2$, we have $M_1/Q_1\cong \SL_3(2)\times \Sp_{2n-2}(2)$, 
 \[B/O_2(B)\cong \SL_2(2)\times \Sp_{2n-2}(2),\;\ov{M_2}\cong \SL_2(2)\times \Sp_{2n}(2)\]
 and $\hat{W}$ is the tensor product of the corresponding natural modules.
 \item	$M_1/Q_1 \cong \SL_3(2) \times \SL_2(2)$, 
 \[B/O_2(B) \cong \SL_2(2) \times \SL_2(2),\;\ov{M_2}\cong \Gamma\SU_4(2) \sim \SU_4(2).2\] 
 and $\widehat{W}$ is the corresponding natural module. Moreover, $[O_2(\ov{B}),\ov{T}]\neq 1$.
 \item	$M=M_1$, $B/O_2(B) \cong \Sym(3)$, $\ov{M_2} \in \{\G_2(2), \,\G_2(2)^\prime \}$ and $\widehat{W}$ is the corresponding natural module.
 \item	$M_1/Q_1 \cong  \SL_3(2) \times \SL_2(2)$, $B/O_2(B) \cong \SL_2(2) \times \SL_2(2)$, $\ov{M_2} \cong \Sp_6(2)$ and $\widehat{W}$ is the spin module. Moreover, $[O_2(\ov{B}),\ov{T}]\neq 1$.
\end{enumerate}
\end{prop}

\begin{proof}
Let $U$ be an irreducible $\ov{H}$-submodule of $\widehat{W}$. We inspect now the list of modules in Lemma \ref{Lemma2.6} to determine the possibilities for $\ov{H}$ and $U$. We remark that $\Omega^+_2(2) = 1$ and $\Omega^-_2(2) \cong C_3$. Moreover, $\Omega^+_4(2)$ is of order $36$ and thus solvable. Thus, if $\ov{H}\cong \Omega^\pm _{2n} (2)$, then $n\geq 3$ or $\ov{H}\cong \Omega^-(2)\cong A_5$. In particular, $\ov{H}$ is simple and so $\ov{H}=\ov{L}$. Then Lemma \ref{Lemma2.8} shows that $U$ is not a natural module. So we can exclude the possibility that $\ov{H}\cong \Omega^\pm _{2n} (2)$ and $U$ is the natural module.

\smallskip

Note moreover that $S_6\cong \Sp_4(2)$ and the natural permutation module for $S_6$ can be regarded as a natural $\Sp_4(2)$-module. We use also that $\SL_2(2) \cong \SU_2(2) \cong \Sp_2(2)$ and $\SU_3(2)$ are solvable, and that $\SL_2(4) \cong \Sp_2(4)$. So it follows from Lemmas \ref{Lemma2.6} to \ref{Lemma2.9} that one of the following holds.	
\begin{itemize}
    \item[(i)]		$\ov{H} \in \{\Sp_{2n}(2)\colon n\geq 2\}\cup\{\Sp_{2n}(4)\colon n\geq 1\}\cup\{\SU_{n}(2)\colon n\geq 4\}\cup\{\G_2(2)^\prime \}$ and $U$ is the corresponding natural module. If $\ov{H}\cong \Sp_{2n}(2)$, then $\ov{Y_M}$ induces an $\mathbb{F}_2$-transvection on $U$. 
    \item[(ii)]		$\ov{H} \cong \Sp_6(2)$, $U$ is the spin module and $\ov{Y_M}$ is a short root subgroup of $\ov{H}$.
\end{itemize}
In the cases above, unless $\ov{H}\cong \Sp_{2n}(2)$ and $U$ is the natural module, there is no involution in $\ov{H}$ which induces a transvection on $U$. Thus, using Lemma~\ref{Lemma2.3}(f), we see that we are left with the following possibilities.
\begin{itemize}
    \item[(a)]	$\ov{H} \cong \Sp_{2n}(2)$ for $n\geq 2$, $\widehat{W}$ is the direct sum of two isomorphic natural modules and $Y_M$ induces a transvection on these natural modules.
    \item[(b)]	$\ov{H} \cong \SU_{n}(2)$ for $n\geq 4$, $\widehat{W}$ is the natural module and $Y_M$ induces an $\mathbb{F}_4$-transvection on $\widehat{W}$.
    \item[(c)]	$\ov{H} \cong \Sp_{2n}(4)$ for $n\geq 1$, $\widehat{W}$ is the natural module and $Y_M$ induces an $\mathbb{F}_4$-transvection on $\widehat{W}$.
    \item[(d)]	$\ov{H} \cong \G_2(2)^\prime$, $\widehat{W}$ is the natural module and $\ov{Y_M}$ is a long root element.
    \item[(e)]	$\ov{H} \cong \Sp_6(2)$, $\widehat{W}$ is the spin module and $\ov{Y_M}$ is a short root subgroup.
\end{itemize}
When analysing these cases, we use throughout that, by Lemma \ref{Lemma2.3}(b), $B=C_{M_2}(\ov{Y_M})$ and so $\ov{H \cap B} = C_{\ov{H}}(\ov{Y_M})$. This allows us in each case to compute $\ov{H \cap B}$. In addition, by \ref{Lemma2.2}(d), $V/Z$ and $[\widehat{W},Y_M]$ are isomorphic $B$-modules. Thus, we can determine the action of $H \cap B$ on $V/Z$.

\smallskip

Note also that $B\cap H$ is normal in $B$ and thus $H \cap O_2(B) = O_2(B\cap H)$. This yields $(H \cap B)/ O_2(H \cap B)\cong \ov{H \cap B}/O_2(\ov{H\cap B})$. 

\smallskip

We remark also that $M^\circ\cap B$ is normal in $B$ and so $M^\circ \cap O_2(B)=O_2(M^\circ \cap B)$. Moreover, it follows from  Lemma~\ref{Lemma2.1}(f) that $M^\circ \cap B/O_2(M^\circ\cap B)\cong \SL_2(2)$ acts faithfully on $V/Z$. Hence, 
\begin{equation*}\label{eq-2.11}\tag{$*$}
\begin{aligned}
 & (M^\circ \cap B)O_2(B)/O_2(B)\mbox{ is a normal subgroup of }B/O_2(B)\\
& \mbox{ which is isomorphic to }\SL_2(2)\mbox{ and acts faithfully on }V/Z 
\end{aligned}
\end{equation*}
It follows moreover from Lemma~\ref{Lemma2.1}(g) that $M_1/Q_1\cong \SL_3(2)\times C_B(V)/Q_1$ and $B/Q_1\cong S_4\times C_B(V)/Q_1$. In particular, $O_2(C_B(V))=Q_1$ and so 
\begin{equation*}\label{eq-2.11-2}\tag{$**$}
M_1/Q_1\cong \SL_3(2)\times C_B(V)/Q_1\mbox{ and }B/O_2(B)\cong \SL_2(2)\times C_B(V)/Q_1.
\end{equation*}
Before we start with the case analysis, set now $D:=C_{M_2}(\ov{H})$. Note that $D \le N_{M_2}(\ov{Y_M})=B$ and that $D$ is normal $M_2$.

\smallskip

\textbf{Case (a):}	We first argue that $M_2=DH$. This clearly holds if $\Aut(\ov{H}) = \Inn(\ov{H})$, which is the case if $n\geq 3$. If $n=2$, then note that, for every element $g\in \ov{M_2}$, the centralizer $C_{\hat{W}}(\ov{Y_M}^g)$ has the same order as $C_{\hat{W}}(\ov{Y_M})$ and thus index $4$ in $\hat{W}$. As $\hat{W}$ is the sum of two isomorphic natural $\ov{H}$-modules, we can conclude that $\ov{Y_M}^g$ induces a transvection on each of these natural modules. It follows then from the structure of $\Aut(\Sp_4(2))=\Aut(S_6)$ and of the natural $S_6$-module that conjugation by $g$ induces an inner automorphism of $\ov{H}$. Thus, it is also true for $n=2$ that $M_2=DH$.

\smallskip

Since $M_2$ acts irreducibly on $\hat{W}$ and $\ov{H}$ does not, we have $\ov{D}\neq 1$. If $[V/Z,D]=1$, then, as $W=\<V^{M_2}\>$, it follows that $D$ acts trivially on $\hat{W}$, contradicting Lemma~\ref{Lemma2.3}(e). Hence, $[V/Z,D]\neq 1$ and Lemma~\ref{Lemma2.10}(a) implies $T=O^2(B\cap M^\circ)\leq D$. Now \eqref{eq-2.11} implies that $\ov{D}$ is not centralized by $\ov{B}$ and so $\ov{D}\not\leq Z(\ov{M_2})$. As $\ov{M_2}=\ov{H}\ov{D}$, this shows that $\ov{D}$ is not abelian. Since $C_{GL(\hat{W})}(\ov{H})\cong \SL_2(2)$, we can conclude that $\ov{D}\cong \SL_2(2)$. Hence,
\[\ov{M_2}=\ov{D}\times\ov{H}\cong \SL_2(2)\times \Sp_{2n}(2)\]
and $\hat{W}$ is the tensor product of the corresponding natural modules. As $\ov{B}=C_{\ov{M_2}}(\ov{Y_M})$, it follows thus from the structure of $\ov{H}$ that $B/O_2(B)\cong \ov{B}/O_2(\ov{B})\cong \SL_2(2)\times \Sp_{2n-2}(2)$. Comparing this to \eqref{eq-2.11-2}, it follows that $C_B(V)/Q_1\cong \Sp_{2n-2}(2)$ and $M_1/Q_1\cong \SL_3(2)\times \Sp_{2n-2}(2)$. This shows (1).

\smallskip

\textbf{Case (b)} As $\SU_n(2)$ contains no normal subgroup of index 2, we have $\ov{H}=\ov{L}$. We get from the discussion above as well as the structure of $\ov{H}$ and the natural module $\hat{W}$ that 
\[
(H \cap B)/O_2(H \cap B) \cong C_3 \times \SU_{n-2}(2) \qquad \text{ and } (H \cap B)/C_{H \cap B}(V/Z) \cong C_3.\]
Then $L \cap B=H\cap B$ acts non-trivially on $V/Z$, so Lemma \ref{Lemma2.10} yields $M_2=LS$. 

\smallskip

If $\ov{M_2}=\ov{H}$, then $\ov{B} \le \ov{H}$. Using \eqref{eq-2.11}, we obtain in this case a contradiction. Hence, $\ov{M_2}\neq \ov{H}$. As $\ov{H} \norm \ov{M_2}$ and $\Out(\ov{H})$ has order 2, it follows $\ov{M_2} \cong \Gamma\SU_n(2) = \SU_n(2)\<\sigma\>$ for a field automorphism $\sigma$. As $\ov{B}=C_{\ov{M_2}}(\ov{Y_M})$, it follows that 
\[B/O_2(B)\cong \ov{B}/O_2(\ov{B})\cong (C_3 \times \SU_{n-2}(2)) \< \sigma \>.\]
Again by (\ref{eq-2.11}) the only possibility is that $n=4$ and, in such a case, we obtain $B/O_2(B) \cong \SL_2(2) \times \SL_2(2)$. Comparing this to \eqref{eq-2.11-2} we can conclude that $M_1/Q_1\cong \SL_3(2)\times \SL_2(2)$. Moreover, one sees from the structure of $\ov{B}$ that $[O_2(\ov{B}),\ov{D}]\neq 1$ for every subgroup $D\leq B$ such that $DO_2(B)/O_2(B)$ is a normal subgroup of $B/O_2(B)$ of order $3$. Hence (2) holds.

\smallskip

\textbf{Case (c):}	We will show that this case leads to a contradiction. Similarly as in (b), we have $\ov{H} = \ov{L}$ as $\ov{H}$ is simple. Moreover, the structure of $(\ov{H},\hat{W})$ gives in this case
\[\frac{H \cap B}{O_2(H \cap B)} \cong \Sp_{2n-2}(4) \qquad \text{and} \qquad [V/Z, H \cap B] = 1.\]
As $T = O^2(M^\circ \cap B)$ by Lemma~\ref{Lemma2.10}(a), we get in particular from \eqref{eq-2.11} that $T \not \le H$. As $T=O^2(T)$, this implies $T\not\leq HQ_2$ and thus $\ov{T}\not\leq\ov{H}$.

\smallskip

Since $\Out(\ov{H})$ has order $2$, every element of $\ov{M_2}$ of odd order induces an inner automorphism of $\ov{H}$. Hence,
\[\ov{T}\leq O^2(\ov{M_2})\leq \ov{H}\ov{D}.\]
In particular, as $\ov{T}\not\leq \ov{H}$, we have $\ov{D}\neq 1$. 

\smallskip

Note that $C_{\GL(\widehat{W})}(\ov{H}) \cong C_3$, so $\ov{D}\cong C_3$. Since $\ov{D}\leq C_{\ov{M_2}}(\ov{Y_M})=\ov{B}$, we have moreover $\ov{H}\ov{D}\cap \ov{B}=\ov{(H\cap B)}\times \ov{D}$. As $\ov{H\cap B}/O_2(\ov{H\cap B})\cong \Sp_{2n-2}(4)$ is simple or trivial and $T=O^2(T)$ is normal in $B$, it follows that $1\neq \ov{T}\leq \ov{D}$ and then $\ov{T}=\ov{D}$. In particular, by \eqref{eq-2.11}, we have $\ov{D}\not\leq Z(\ov{B})$ and thus $\ov{D}\not\leq Z(\ov{M_2})$. Using again that $\Out(\ov{H})$ has order $2$, we can see now that 
\[\ov{M_2}\cong (\Sp_{2n}(4)\times C_3)\<\sigma\>,\]
where $\sigma$ induces a field automorphism. This implies
\[B/O_2(B)\cong \ov{B}/O_2(\ov{B})\cong (\Sp_{2n-2}(4)\times C_3)\<\sigma\>.\]
Now \eqref{eq-2.11} yields $n=1$. 
In particular, $\ov{H}=\ov{L} \cong  \SL_2(4)$. By the structure of $\SL_2(4)$, we may find an element $h \in L$ such that $\ov{S \cap H} = \ov{Y_M} \ov{Y_M^h}$; thus we have $(S \cap H)Q_2 = Y_M Y_M^h Q_2$. Then $[W,S \cap H] \le [W, Y_M Y_M^h]Z \le Y_M Y_M^h$, where the last inclusion follows from Lemma \ref{Lemma2.2}(d). By Lemma \ref{Lemma2.3}(c) and since the Sylow $2$-subgroups of $\SL_2(4)$ are abelian, $Y_M Y_M^h$ is elementary abelian.

\smallskip

Note that Lemma \ref{Lemma2.2}(f) yields $[W,H] = W$ and $Z(W) = C_W(L)$. Moreover, by Lemma~\ref{Lemma2.2}(e), $\ov{M_2}$ acts on $W/Z$, and $W/Z$ is elementary abelian. Hence, it follows from Lemma~\ref{L:ApplyGaschutz}(b) that
\[Z(W)/Z\leq C_{W/Z}(\ov{H})\leq [W/Z,\ov{H\cap S}].\]
In particular, as $[W,H\cap S]\leq Y_MY_M^h$ is elementary abelian, we can conclude that $Z(W)$ is elementary abelian. The transitivity of the action of $H$ on $\widehat{W}^\#$ together with the fact that $V\not\leq Z(W)$ show that every element of $W$ has order $2$ and thus $W$ is elementary abelian, a contradiction to Lemma \ref{Lemma2.2}(e). Thus, case (c) cannot occur.

\smallskip

\textbf{Case (d):}	Again, we have $\ov{H} = \ov{L}$ as $\ov{H}$ is simple. In this case, one can compute
\[(B \cap H)/O_2(B \cap H) \cong \SL_2(2) \qquad \text{and} \qquad [V/Z, B \cap L] \ne 1.\]
Therefore, Lemma \ref{Lemma2.10} yields $LB=M_2 =LS$ and $T \le L \cap B =L \cap M$. In particular, as $T=O^2(T)$ is normal in $B$, it follows $O^2(B\cap H)=T$. Thus, $L\cap \tilde{M}\leq B\cap H=T(B\cap H\cap S)\leq M$ and so $B\leq M$. As $\Out(\ov{H})$ has order $2$, one sees now that (3) holds.

\smallskip
		
\textbf{Case (e):} In this case, once more, $\ov{H}=\ov{L}$ as $\ov{H}$ is simple. Moreover, one can compute that
\[(B \cap H)/O_2(B \cap H) \cong \SL_2(2) \times \SL_2(2) \qquad \text{and} \qquad [V/Z,B \cap L] \ne 1.\]
Applying Lemma \ref{Lemma2.10}, we get $T \le L\cap B$ and $M_2 = LS$. We have in this case that $\Out(\ov{H})=1$ and so $\ov{M_2} = \ov{D}\ov{H} = \ov{D} \times \ov{H}$. As $|\ov{M_2}:\ov{H}|$ is a power of 2, $\ov{D}$ is a 2-group, so $\ov{D} \le O_2(\ov{M_2}) =1$. Thus, $\ov{M_2}=\ov{H}$, $\ov{B}=\ov{B\cap H}$ and $B/O_2(B)\cong \ov{B}/O_2(\ov{B})\cong \SL_2(2)\times \SL_2(2)$. Comparing this to \eqref{eq-2.11-2}, one sees that $M_1/Q_1\cong \SL_3(2)\times \SL_2(2)$. Furthermore one sees from the structure of $\ov{B}$ that $[O_2(\ov{B}),\ov{T}]\neq 1$. Hence, (4) follows.
\end{proof}

\begin{lemma}[{Cf. \cite[Lemma~2.12]{MStr:2008}}]\label{Lemma2.12}
Suppose that $L$ is non-solvable. Then $Q_2=Q=W$ and $Z(W)=Z$.
\end{lemma}

\begin{proof}
Our first goal is to show that $Z=C_{Q_2}(W)$. Thus, arguing by contradiction, suppose that $Z < C_{Q_2}(W)$. Pick a subgroup $D \norm M_2$ minimal such that $D \le C_{Q_2}(W)$ and $D > Z$. The minimality of $D$ shows that $\Phi(D)\leq Z$ and so $D/Z$ is abelian. 

\smallskip

Lemma \ref{Lemma2.2}(b) yields $[D,L]=1$ and $D \le Q_1$. Note now that $M_2=L(M\cap B)$ and, as a consequence of the action of $M \cap B$ on $V/Z$, $(M \cap B) / O_2(M \cap B) \cong \SL_2(2)$. Using again the minimality of $D$, we see that 
\[[D,Q_M]\leq [D,O_2(M\cap B)]\leq Z\]
and either $|D/Z|=2$ and $[D,M_2] \le Z$, or $O_2(M\cap B)=C_{M\cap B}(D/Z)$, $M_2/C_{M_2}(D/Z)\cong \SL_2(2)$ and $|D/Z|=4$.
	
\smallskip
		
We prove first that $D$ is abelian. Note that this is clear if $|D/Z|=2$. For the proof that $D$ is abelian, we may thus suppose that $|D/Z|=4$ and $C_{M\cap B}(D/Z)=O_2(M\cap B)$. Pick an element $g \in M \setminus B$. Lemma~\ref{Lemma2.1}(f) implies that $B=N_{M_1}(Z)$ and so $Z \ne Z^g$. It follows from Lemmas~\ref{Lemma2.1}(g) and \ref{Lemma2.2}(d) as well as the module structure of $V$ that $W\cap B^g=C_W(Z^g)$ acts non-trivially on $V/Z^g$ and is thus not contained in $O_2(M\cap B^g)$. Set 
\[R:=[D^g,W\cap B^g].\]
Then $R\leq D^g$. Moreover, as $O_2(M\cap B^g)=C_{M\cap B}(D/Z)^g=C_{M\cap B^g}(D^g/Z^g)$, we have $R\not\leq Z^g$. Since $D^g\leq Q_1\leq M_2$ normalizes $W$, we have also $R\leq W$. Thus, by Lemma~\ref{Lemma2.2}(e), $\Phi(R)\leq \Phi(W)=Z$. As $R\leq D^g\leq W^g$, we see furthermore that $\Phi(R)\leq \Phi(W^g)=Z^g$. As $Z\cap Z^g=1$, this implies $\Phi(R)=1$ and $R$ is elementary abelian. Since $M\cap B^g$ acts transitively on the non-trivial elements of $D^g/Z^g$, it follows that all non-trivial elements of $D^g$ are involutions. Thus, $D\cong D^g$ is elementary abelian. So we have shown that $D$ is abelian.

\smallskip
		
Notice now that $[D,D^g] \le [D,Q_1] \cap [Q_1,D^g] \le Z \cap Z^g =1$. Since $M_1=BM$ and $B$ normalizes $D$, it follows thus from the arbitrary choice of $g$ that $[D,D^h]=1$ for any $h\in M_1$, which implies that 
\[E := \< D^{M_1} \>\]
is abelian.

\smallskip

Assume now that $[E,W]\leq V$. As $O^2(M) \le \< W^M \>$, it follows then $[E,O^2(M)] \le V$. Since $M_1 = BO^2(M)$,  we deduce $E = \< D^{O^2(M)} \> \le DV$. As $V=\<Z^M\>\leq E$, it follows $E =DV$. Then $[D,Q_M]=[DV,Q_M]=[E,Q_M]\norm M$ and $\Phi(E)=\Phi(D) \norm M$. We observed at the beginning that $[D,Q_M] \le Z$ and $\Phi(D) \le Z$. Hence, since $V$ is an irreducible $M$-module,  we conclude $[D,Q_M]=1$ and $\Phi(D)=1$. In particular $[Q_M,E]=1$ and thus $E$ is a normal, elementary abelian, $2$-reduced subgroup of $M$. This implies $E \le Y_M$. Thus, using Lemma~\ref{Lemma2.1}(b), we obtain $D \le Y_M \cap Q_2 =V$. As $B$ acts irreducibly on $V/Z$ and normalizes $D$, it follows $D=V$, which contradicts $[D,L]=1$ and Lemma \ref{Lemma2.2}(b).

\smallskip

Thus we have proved $[E,W] \not  \le V$. In particular we get $E \not \le Y_MQ_2$ and $\ov{Y_M} < \ov{EY_M}$. Since $E$ and $Y_M$ are normal in $M_1$, the subgroup $W$ normalizes $EY_M$. Moreover, $E \le Q_1$ implies that $EY_M$ is abelian. Thus, we have $[W, EY_M, EY_M] \le [EY_M, EY_M] = 1$. In particular $EY_M\unlhd B$ acts quadratically on $\widehat{W}$. Note that in all the cases of Lemma \ref{Lemma2.11}, $\ov{Y_M}$ is maximal with respect to inclusion among the normal subgroups of $\ov{B}= C_{\ov{M_2}}(\ov{Y_M})$ acting quadratically on $\widehat{W}$. This provides us with a contradiction. Thus, we have shown that $Z=C_{Q_2}(W)$. In particular, $Z(W)=Z$.
		
\smallskip

Using Lemma~\ref{Lemma2.2}(e), we see that, for every $g\in Q_2$, the map $g^*\colon W/Z\rightarrow Z,wZ\mapsto [w,g]$ is a group homomorphism which can be regarded as an element of the dual space $(W/Z)^*$ of the $\mF_2$-vector space $W/Z$. Moreover, the map $Q_2\rightarrow (W/Z)^*,\;g\mapsto g^*$ is a group homomorphism with kernel $C_{Q_2}(W)=Z$. Hence, $Q_2/Z$ embeds into $(W/Z)^*\cong (W/Z)$. As $W\leq Q\leq  Q_2$, this implies $W=Q=Q_2$. 
\end{proof}

We finally look for contradictions relative to each of the cases (1) to (4), thus providing a contradiction to the non-solvability of $L$.

\begin{lemma}[{Cf. \cite[Lemma~2.13]{MStr:2008}}]\label{Lemma2.13}
Suppose $L$ is non-solvable. Then case (3) of Lemma~\ref{Lemma2.11} does not hold.
\end{lemma}

\begin{proof}
If case (3) of Lemma~\ref{Lemma2.11} holds, then $\ov{M_2}$ acts transitively on $\widehat{W}^\#$. By Lemma \ref{Lemma2.12}, $Z(W) =Z$ is elementary abelian. Since $Z<V\leq W$, it follows that every element of $W$ is an involution and thus $W$ is elementary abelian. This contradicts Lemma \ref{Lemma2.2}(e).
\end{proof}

\begin{lemma}\label{Lemma2.14}
Suppose $L$ is non-solvable. Then the following hold:
\begin{itemize}
 \item [(a)] Case (1) of Lemma~\ref{Lemma2.11} holds with $n=2$, i.e. $M_1/Q_1\cong \SL_3(2)\times \SL_2(2)$, $B/O_2(B)\cong \SL_2(2)\times \SL_2(2)$, $\ov{M_2}\cong \SL_2(2)\times \Sp_4(2)$ and $\hat{W}=Q_2/Z$ is a tensor product of the corresponding natural modules.
 \item [(b)] A Sylow $3$-subgroup of $C_B(V)$ acts fixed point freely on $Q_1/Y_M$.
\end{itemize}
\end{lemma}

\begin{proof}
By Lemma~\ref{Lemma2.13} one of the cases (1),(2) or (4) of Lemma~\ref{Lemma2.11} holds. Set $m:=n-1$ in the first case and $m:=1$ in the other two cases. Then $B/O_2(B)\cong \SL_2(2)\times \Sp_{2m}(2)$ and $C_B(V)/Q_1\cong \Sp_{2m}(2)$ in all cases. Recall that $\ov{B}=C_{\ov{M_2}}(\ov{Y_M})$, $\ov{Q_1}=O_2(\ov{B})$ and $[\hat{W},\ov{Y_M}]=\hat{V}$ by Lemmas~\ref{Lemma2.2}(d) and \ref{Lemma2.3}(b). Hence, in all cases $[\hat{W},Q_1]$  has index $4$ in $\hat{W}$, $\hat{V}=C_{\hat{W}}(Q_1)=[\hat{W},Q_1,Q_1]$,  and the $B/O_2(B)$-module $[\hat{W},Q_1]/\hat{V}$ is the tensor product of a natural $\SL_2(2)$-module and a natural $\Sp_{2m}(2)$-module. 
 
 \smallskip
 
By Lemmas~\ref{Lemma2.2}(d) and \ref{Lemma2.12} we have $Z(W)=Z\leq V=[W,Y_M]\leq [W,Q_1]$. Clearly, $[W,Q_1]\leq W\cap Q_1$. As $|W/W\cap Q_1|=4$ by Lemmas~\ref{Lemma2.1}(g) and Lemmas~\ref{Lemma2.2}(d), it follows from order considerations that $W\cap Q_1=[W,Q_1]$. In particular, 
\[[W\cap Q_1,Q_1]\leq V\]
and $(W\cap Q_1)/V$ is a $B/O_2(B)$-module, which is the tensor product of a natural $\SL_2(2)$-module and a natural $\Sp_{2m}(2)$-module. In particular, $(W\cap Q_1)/V$ is a simple $B$-module and $O^2(M^\circ\cap B)$ acts fixed point freely on $(W\cap Q_1)/V$. 

\smallskip

Set now $E:=\<(W\cap Q_1)^{M_1}\>$. Then $[E,Q_1]\leq V$ so that $E/V$ is elementary abelian and an $M_1/Q_1$-module. Let $V\leq F\leq E$ with $F/V=C_{E/V}(\<W^{M_1}\>)$. It follows from Lemma~\ref{Lemma1.3} applied with $(E/V,(W\cap Q_1)/V,W\unlhd B\leq M_1)$ in place of $(W,Y,A\unlhd B\leq H)$ that $E/F$ is either trivial or an irreducible $M_1$-module.

\smallskip

Notice that $WQ_1=O_2(B)=QQ_1$ by Lemmas~\ref{Lemma2.1}(g) and \ref{Lemma2.2}(d), so $M^\circ Q_1=\<W^{M_1}\>Q_1$. As $B\cap M^\circ$ acts fixed point freely on $(W\cap Q_1)/V$, it follows $(W\cap Q_1)\cap F=V$ and so 
\[(W\cap Q_1)F/F\cong (W\cap Q_1)/V\]
as a $B$-module. In particular, $F\neq E$. Thus, $E/F$ is an irreducible $M_1$-module and $[E,W]\not\leq V$. As $[E,W]\leq [Q_1,W]=Q_1\cap W$ and since $(W\cap Q_1)/V$ is a simple $B$-module, it follows $[E,W]V=W\cap Q_1$ and $[E/F,W]=(W\cap Q_1)F/F$. Recall also that $O_2(B)=Q_1W$ and $[Q_1\cap W,W]\leq W'=Z\leq V\leq F$. Hence, $O_2(B)$ acts quadratically on $E/F$ and 
\begin{equation}\label{E:EmodFO2B}
[E/F,O_2(B)]\cong (W\cap Q_1)/V\mbox{ as a $B/O_2(B)$-module.} 
\end{equation}
It is now a consequence of Lemma~\ref{L:TensorProductModule} that $E/F$ is isomorphic to the direct sum of three copies of the natural $\Sp_{2m}(2)$-module for $C_B(V)/Q_1$, and that $[E,T]\leq (W\cap Q_1)F$. As $M^\circ Q_1=\<W^{M_1}\>Q_1$ and $[Q_1,W]\leq W\cap Q_1\leq E$, we have $[Q_1,O^2(M^\circ)]\leq E$ and $[F,M^\circ]\leq V$. Hence, we can conclude that $[Q_1,T]\leq W\cap Q_1\leq W\leq Q_2$ and so $[O_2(\ov{B}),\ov{T}]=1$. Thus, cases (2) and (4) of Lemma~\ref{Lemma2.11} cannot hold, which means that we are in case (1) of that Lemma.

\smallskip

Recall now that $Z(W)=Z$ and $Q_2=W$ so that $Q_2/Z$ is the direct sum of two natural $\Sp_{2n}(2)$-modules for $\ov{H}$. This means that there are precisely two non-trivial $C_B(V)$-composition factors in a $C_B(V)$-composition series for $Q_2$. Indeed, these non-trivial composition factors are $C_B(V)$-composition factors of $[W,Q_1]=W\cap Q_1$. Moreover, there is precisely one non-trivial $C_B(V)$-composition factor in $O_2(\ov{B})$. Thus, looking at the structure of $M_2$, we see that there are altogether exactly three non-trivial $C_B(V)$-composition factors in $O_2(B)$. Recall now that $C_B(V)$ has also three non-trivial composition factors on $E/F$. Hence, there is no non-trivial $C_B(V)$-composition factor in $F$ and thus $[F,O^2(C_B(V))]=1$. As $(W\cap Q_1)/V$ is the direct sum of two natural $\Sp_{2m}(2)$-modules for $C_B(V)/Q_1$, we can see that $E/V=\<((W\cap Q_1)/V)^{M_1}\>$ is a sum of natural $\Sp_{2m}(2)$-modules and can then be written as a direct sum of such modules. In particular, $C_{E/V}(O^2(C_B(V)))=1$. Thus we obtain $F=V$ and 
\begin{equation}\label{E:EmodV}
E/V\mbox{ is the direct sum of three natural $\Sp_{2m}(2)$-modules for }C_B(V)/Q_1.
\end{equation}
Equation~\eqref{E:EmodFO2B} yields now $[E,Q_2]V=[E,O_2(B)]V=W\cap Q_1$ and so $W\cap Q_1\leq E\cap Q_2$. As $W=Q_2$, we have also $E\cap Q_2=E\cap W\leq Q_1\cap W$. Hence, $W\cap Q_1=E\cap Q_2$ and $\ov{E}\cong E/E\cap Q_2$ is a natural $\Sp_{2m}(2)$-module by \eqref{E:EmodV} and since $(W\cap Q_1)/V$ is the direct sum of two natural $\Sp_{2m}(2)$-modules. It follows from the structure of $\ov{M_2}$ that $O_2(\ov{B})$ is a an extension of $\ov{Y_M}$ by a natural $\Sp_{2m}(2)$-module, which is non-split if $n\geq 3$. Hence, $n=2$ and so (a) holds. In particular, $m=1$ and $|\ov{Q_1}|=|O_2(\ov{B})|=2^3$.

\smallskip

We calculate now that $|Q_1\cap Q_2|=|Q_1\cap W|=|V|\cdot 2^4=2^7$ and so $|Q_1|=|\ov{Q_1}|\cdot |Q_1\cap Q_2|=2^{10}$. At the same time, by \eqref{E:EmodV}, we have $|E/V|=2^6$ and thus $|E|=2^9$. As $[Y_M,M_2]\leq V$, it follows also from \eqref{E:EmodV} that $Y_M\not\leq E$. Hence, $Q_1=EY_M$ and $E\cap Y_M=V$. So $Q_1/Y_M\cong E/V$ as an $M_1$-module and (b) is obtained from another application of \eqref{E:EmodV}. 
\end{proof}

\begin{lemma}[{Cf. \cite[Lemma~2.15]{MStr:2008}}]\label{Lemma2.15}
$L$ is solvable.
\end{lemma}

\begin{proof}
Suppose that $L$ is non-solvable. We show now that the situation described in Lemma~\ref{Lemma2.14} does not occur. By that lemma, we have $M_1/Q_1\cong \SL_3(2)\times \SL_2(2)$. Indeed, according to Lemma~\ref{Lemma2.1}(f), $M_1/Q_1=M^\circ Q_1/Q_1\times C_B(V)/Q_1$ where $M^\circ Q_1/Q_1\cong \SL_3(2)$ and $C_B(V)/Q_1\cong \SL_2(2)$. We note also that $C_B(V)=C_B(Y_M)$ by Lemma~\ref{Lemma2.1}(e).

\smallskip

Let now $D$ be a Sylow $3$-subgroup of $B$. Then $D=D_1D_2$ where $D_1=C_D(V)$ and $D_2=D\cap M^\circ Q_1$. Set 
\[N_1:=N_{M_1}(D_1).\]
As $D_1Q_1$ is normal in $M_1$, it follows from the Frattini argument that $M_1=N_1Q_1$. By Lemma~\ref{Lemma2.14}(b), $D_1$ acts fixed point freely on $Q_1/Y_M$. Hence, $N_1\cap Q_1=Y_M$, 
\[N_1\sim 2^{3+1}(\SL_3(2)\times \SL_2(2))\] 
and $|O_2(N_1/D_1)|=2^5$. By \cite[Lemma~1.1(b)]{MStr:2008} we have $C_{Y_M}(M^\circ)=1$ and so $Y_MD_1/D_1$ intersects trivially with $Z(N_1/D_1)$. As $Z(N_1/D_1)$ does not contain any element of odd order, it follows that $|Z(N_1/D_1)|=2$. Let $E_1$ be the inverse image of $Z(N_1/D_1)$ in $N_1$ and set $F_1=C_{N_1}(E_1)$. Then $E_1\cong \SL_2(2)$. In particular, $Z(E_1)=1$ and $\Aut(E_1)\cong E_1$. Hence, $N_1=E_1\times F_1$. Note also that $E_1\leq C_B(V)=C_B(Y_M)$ and $Y_M D_2\leq F_1$. Moreover, $F_1/Y_M\cong \SL_3(2)$. Set now
\[N:=N_B(D)=N_{N_1}(D_2)\cap B.\]
Then $|Y_M\cap N|=|C_{Y_M}(D_2)|=4$ and 
\[(F_1\cap N)/(Y_M\cap N)\cong (F_1\cap N)Y_M/Y_M\cong \SL_2(2)\]
using a Frattini argument. By \cite[Lemma~1.1(c)]{MStr:2008} we have $[Y_M\cap N,F_1\cap N]\neq 1$. Hence, $(F_1\cap N)/D_2\cong D_8$. Note also that $E_1\leq C_N(D_2)\leq N$, so $N=E_1\times (F_1\cap N)$ and $C_N(D_2)=E_1\times C_{F_1\cap N}(D_2)=E_1\times D_2\times (Y_M\cap N)$. Thus, 
\[N/D\cong (E_1/D_1)\times (F_1\cap N)/D_2\cong C_2\times D_8\] 
and the Sylow $2$-subgroups of $C_N(D_2)$ are elementary abelian of order $8$.

\smallskip

We now consider the embedding of $N$ in $M_2$ to obtain a contradiction. By Lemma~\ref{Lemma2.14} we have $\ov{M_2}\cong \SL_2(2) \times \Sp_4(2)$ and $\hat{W}=Q_2/Z$ is a tensor product of the corresponding natural modules. Here more precisely, $\ov{H}\cong \Sp_4(2)\cong S_6$ and $C_{\ov{M_2}}(\ov{H})\cong \SL_2(2)$. 

\smallskip

As $D_1$ and $D_2$ are the only normal subgroups of $N$ of order $3$ and since $D_1$ centralizes $V/Z$ while $D_2$ does not, we have 
\[\ov{D_1}\leq \ov{H} \mbox{ and } \ov{D_2}=O^2(C_{\ov{M_2}}(\ov{H}))\unlhd \ov{M_2}.\]
In particular, every $2$-subgroup of $\ov{M_2}$ which is normalized by $\ov{D_2}$ is also centralized by $\ov{D_2}$. This yields $[O_2(B\cap F_1),D_2]\leq Q_2$. Looking at the structure of $F_1$ we note now that there are two non-trivial $D_2$-composition factors in $O_2(B\cap F_1)$ which then must be $D_2$-composition factors of $O_2(B\cap F_1)\cap Q_2$. Moreover, $D_2$ centralizes $Z\leq O_2(B\cap F_1)\cap Q_2$. Hence, $|F_1\cap Q_2|\geq 2^5$ and so $|C_{Q_2}(E_1)|\geq 2^5$. Observe also that $\ov{E_1}\leq O^{2^\prime}(C_{\ov{M_2}}(\ov{D_2}))=\ov{H}$. Recall that $\hat{W}=Q_2/Z$ is the direct sum of two natural $\Sp_4(2)$-modules for $\ov{H}$. So the involutions in $\ov{E_1}$ must induce transvections on each of these natural modules. In particular, $\ov{E_1}\not\leq \ov{H}'\cong \Sp_4(2)'\cong A_6$. 

\smallskip

Set now 
\[N_2:=N_{M_2}(D_2)\mbox{ and }U_2:=C_{M_2}(D_2)'.\]
Note that $N_2\cap Q_2=C_{Q_2}(D_2)=Z$. As $\ov{D_2}$ is normal in $\ov{M_2}$, a Frattini argument gives moreover $\ov{N_2}=\ov{M_2}$. Hence
\[N_2/Z\cong \ov{M_2}\cong \Sp_4(2)\times \SL_2(2).\]
As $\Aut(D_2)\cong C_2$, it follows that
\[C_{M_2}(D_2)/Z\cong \Sp_4(2)\times C_3\mbox{ and }U_2Z/Z\cong \Sp_4(2)'\cong A_6.\]
In particular, $U_2Z$ has index $6$ in $C_{M_2}(D_2)$. Note that $C_N(D_2)\leq C_{M_2}(D_2)$ and recall that $C_N(D_2)$ has Sylow $2$-subgroups which are elementary abelian of order $8$. Hence, $U_2Z$ contains a fours group. Since $2.A_6\cong \SL_2(9)$ has quaternion Sylow $2$-subgroups, it follows that the extensions $U_2Z/Z$ must be split. Thus, $U_2\cong \Sp_4(2)'$ and $U_2\cap Z=1$.

\smallskip

We may now choose $D$ such that $D_1\leq U_2$. Then $N\cap U_2\cong \SL_2(2)$ and so $(N\cap U_2)D/D$ has order $2$. As $N\cap U_2 \unlhd N$, it follows $(N\cap U_2)D/D\leq Z(N/D)$. Note also that $ZD/D\leq Z(N/D)$ as $Z$ is centralized by $B$. Moreover, $E_1D/D\leq Z(N/D)$ and $\ov{E_1}\not\leq \ov{H}'=\ov{U_2Z}$. Hence, $|Z(N/D)|\geq 8$ contradicting $N/D\cong C_2\times D_8$. 
\end{proof}

\subsection{The structure of the amalgam}\label{SS:Amalgam}

In this section we first determine the structure of $M$ and $\tilde{C}$. Furthermore we will show that $M=\tilde{M}$ and $B=M\cap \tilde{C}$. In a second step, we will determine the structure of the amalgam formed by $B$, $M$, $\tilde{C}$ and the corresponding inclusion maps showing in particular that this amalgam is isomorphic to an amalgam found in $\Aut(\G_2(3))$.  

\begin{prop}\label{Proposition2.16}
The following hold.
\begin{itemize}
    \item[(a)] $Z=Z(W)$, $Q_M = Y_M$, $\ov{H} \cong \Sym(3)$ and $|S|=2^7$.
    \item[(b)] $Q \cong Q_8 \circ Q_8$ is the central product of two quaternion groups of order 8.
    \item[(c)] $\tilde{C}/Q \cong \Sym(3) \times \Sym(3)$.
    \item[(d)] $Q_2=W=Q$ and $M_2 = \tilde{C}$.
\end{itemize}
\end{prop}

\begin{proof}
We use that $L$ is solvable by Lemma~\ref{Lemma2.15}. Hence, by Lemma \ref{L:F*ovL}, $F^*(\ov{H}) = \ov{L}$ is an $r$-group for some odd prime $r$. As $Y_M\not\leq O_2(LY_M)$, we have $[\ov{L},\ov{Y_M}]\neq 1$ and thus, as the action of $\ov{Y_M}$ on $\ov{L}$ is coprime, $[\ov{L}/\Phi(\ov{L}),\ov{Y_M}]\neq 1$. Using the minimality of $L$ in the form of Lemma~\ref{Minimality of L}, we see moreover that $[\Phi(\ov{L}),\ov{Y_M}]=1$ and the action of  $M \cap B = M \cap \tilde{C}$ on $\ov{L}/\Phi(\ov{L})$ is irreducible. The latter fact implies in particular that $\ov{Y_M}$ acts fixed point freely on $\ov{L}/\Phi(\ov{L})$ and inverts thus every element of $\ov{L}/\Phi(\ov{L})$. As $[\Phi(\ov{L}),\ov{Y_M}]=1$ and $\ov{H}=\<\ov{Y_M}^{\ov{L}}\>$, we obtain moreover
\[\Phi(\ov{L})\leq Z(\ov{H})\leq Z(\ov{L}).\]

\smallskip

Lemma~\ref{Lemma2.2}(e),(f) gives $W=[W,L]$ and $[W,Q_2]=Z=W'=\Phi(W)$. This implies that $W/Z$ is abelian, $\ov{L}$ acts on $W/Z$, and $[W/Z,\ov{L}]=W/Z$. As $\ov{L}$ and $W/Z$ have coprime orders, it follows that $C_{W/Z}(\ov{L})=1$. In particular, by Lemma~\ref{Lemma2.2}(f), $Z(W)=C_W(L)\leq Z$ and thus $Z=Z(W)$. This shows that $W$ is an extra-special 2-group. Moreover, by Lemma~\ref{Lemma2.3}(e), $\ov{M_2}$ acts faithfully in $W/Z$. 

\smallskip

Next we prove that $\ov{L}$ is abelian. By contradiction, suppose it is not. Then $1\neq \ov{L}'\leq \Phi(\ov{L})$ and so $Z(\ov{H}) \ne 1$. Moreover, $Z(\ov{L})$ is a proper subgroup of $\ov{L}$ and thus Lemma~\ref{Minimality of L}(b) yields 
\[Z(\ov{L})\leq C_{\ov{M_2}}(\ov{Y_M})=\ov{B},\]
where the latter equality uses Lemma~\ref{Lemma2.3}(b). In particular, $Z(\ov{L})$ acts on $V/Z$. If an element of $Z(\ov{L})$ acts trivially on $V/Z$ then, as $W=\<V^L\>$, it acts trivially on $W/Z$ and is thus trivial since $\ov{M_2}$ acts faithfully on $W/Z$. Hence, $Z(\ov{L})$ acts faithfully on $V/Z$. This shows that $\ov{L}'=\Phi(\ov{L})=Z(\ov{L})$ has order $3$. Hence, $\ov{L}$ is an extraspecial $3$-group.

\smallskip

Let $Z(\ov{L}) \le A \le \ov{L}$ with $|A|=9$. As $Y_M$ inverts $\ov{L}/Z(\ov{L})$, the subgroup $A$ is $Y_M$-invariant and $A$ is not centralized by $Y_M$. Let $A_1=[A,Y_M]$ so that $A = A_1 \times C_A(Y_M)$. Since $Z(\ov{L}) \le C_A(Y_M)$ and $|A|=9$ we have $Z(\ov{L}) = C_A(Y_M)$ and so $A = A_1 \times Z(\ov{L})$ is elementary abelian. Denote by $A_1,A_2, A_3, Z(\ov{L})$ the subgroups of order $3$ in $A$.

\smallskip

Since $Z(\ov{L})$ acts faithfully and irreducibly on $V/Z$ and $W/Z = \<(V/Z)^L\>$, we can write $W/Z$ as a direct sum of faithful and irreducible $Z(\ov{L})$-modules. In particular, $C_{W/Z}(Z(\ov{L}))=1$. As $A$ is abelian and non-cyclic and the action of $A$ on $W/Z$ is coprime, it follows  
\[W/Z = \<C_{W/Z}(B)\colon 1\neq B\leq A\>=\sum_{i=1}^3 C_{W/Z}(A_i).\]
Since $A$ is a normal non-central subgroup of $\ov{L}$, the action of $\ov{L}$ on $\{A_1,A_2,A_3 \}$ is transitive. Hence $|W/Z| \leq  |C_{W/Z}(A_i)|^3$ for $i=1,2,3$. In addition, as $Z(\ov{L})$ acts non-trivially on $C_{W/Z}(A_i)$, we have $|C_{W/Z}(A_i)| \ge 4$. Note that 
$C_{W/Z}(A_2)\cap C_{W/Z}(A_3)=C_{W/Z}(A)=1$ and $\ov{Y_M}$ swaps $A_2$ and $A_3$. Hence, $|[C_{W/Z}(A_2),\ov{Y_M}]|=|C_{W/Z}(A_2)|$. As $|[W/Z,Y_M]|=|V/Z|=4$ by Lemma~\ref{Lemma2.2}(d), it follows $|C_{W/Z}(A_i)|=4$ for $i=1,2,3$. We obtain then $|W/Z|\leq 2^6$. Recall that $\ov{L}$ acts faithfully on $\widehat{W}=W/Z$ and is thus isomorphic to a subgroup of $\GL_6(2)$. Since a Sylow $3$-subgroup of $\GL_5(2)$ has order $9$, we have indeed
\[|W/Z|=2^6.\]
 A Sylow 3-subgroup of $\GL_6(2)$ has order $3^4$ and is by \cite{Weir:1955} isomorphic to $C_3 \wr C_3$, which is not extra-special. Thus $|\ov{L}|=3^3$. As $A$ is elementary abelian and $A$ was an arbitrary subgroup of $\ov{L}$ of order $9$, $\ov{L}$ is of exponent $3$. Thus, by \cite[Theorem 1]{Winter:1972}, we have
	\[
	C_{\Out(\ov{L})}(Z(\ov{L})) \cong \SL_2(3) \qquad \text{and} \qquad |C_{\GL(W/Z)}(\ov{L})|=3=|Z(\ov{L})|.
	\]
The latter property implies that $C_{\ov{M_2}}(\ov{L})=Z(\ov{L})$ and $\ov{M_2}/\ov{L}$ embeds into $\Out(\ov{L})$. 

\smallskip

Recall that $Z(\ov{L})\leq C_{\ov{M_2}}(\ov{Y_M})=\ov{B}$. In particular, $[Z(\ov{L}),\ov{O_2(B)}]\leq \ov{O_2(B)} \cap \ov{L}=1$. Hence, $\ov{O_2(B)}$ is isomorphic to a subgroup of $\SL_2(3)$. Note that $\ov{Y_M}\leq \ov{O_2(B)}$. As the unique Sylow $2$-subgroup of $\SL_2(3)$ is quaternion of order $8$, we get $\Omega_1(\ov{O_2(B)}) \le \ov{Y_M}$.

\smallskip

Set $E:= \<(W \cap Q_M)^M \>$. By Lemma~\ref{Lemma2.1}(h) we have $W\cap Q_M\leq C_{Q_2}(V)\leq Q_1$ and thus $E\leq Q_1\leq O_2(B)$. As $\Phi(W \cap Q_M) \le Z \le V$, the group $E/V$ is generated by involutions. Since $V \le Q_2$, it follows that $\ov{E}$ is generated by involutions. So $\ov{E} \le \Omega_1(\ov{O_2(B)}) \le \ov{Y_M}$ and hence $E \le Y_MQ_2$.

\smallskip

As $W \not \le Q_M$ and $M/Q_M$ is simple, we have $\<W^M\>Q_M=M$ and thus $O^2(M) \le \<W^M\>$. In particular, $[Q_M,O^2(M)] \le \<[Q_M,W]^M\> \le E\leq Y_MQ_2$. Thus Lemma \ref{Lemma2.4} yields $Q_M=Y_M$ and we get $|S|=2^7=|W|$, a contradiction.

\smallskip
	
Therefore, $\ov{L}$ must be abelian. As we have argued at the beginning of the proof that $\ov{Y_M}$ inverts $\ov{L}/\Phi(\ov{L})$ and centralizes $\Phi(\ov{L})$, it follows that $\ov{L}$ must be elementary abelian and $\ov{Y_M}$ inverts $\ov{L}$. Pick $R$ to be a simple $L$-submodule of $\widehat{W}=W/Z$. Then $C_{\ov{L}}(R)$ is normalized by $LY_M=H$ and centralizes thus $\< R^H \>$. As $\widehat{W}$ is a homogeneous $H$-module by Lemma~\ref{Lemma2.2}(g) and $\ov{M_2}$ acts faithfully on $\widehat{W}$ by Lemma~\ref{Lemma2.3}(e), it follows that $C_{\ov{L}}(R)=1$. Note that $C_R(l)=1$ for every $1\neq l\in \ov{L}$, since $\ov{L}$ is abelian and so $C_R(l)$ is an $\ov{L}$-submodule of $R$. Thus, it follows from \cite[Theorem 8.3.2]{Kurzweil/Stellmacher:2004} that $\ov{L}$ is cyclic. So $\ov{H}=\ov{L}\ov{Y_M}$ is dihedral and thus generated by two conjugates of $\ov{Y_M}$. By Lemma \ref{Lemma2.2}(d) $[W/Z, Y_M] =V/Z$, hence $|[W/Z,Y_M]|=4$. Using Lemma~\ref{Lemma2.2}(f), we see now that $|W/Z|=|[W/Z,H]|\leq |[W/Z,Y_M]|^2 = 16$. It follows from Lemmas~\ref{Lemma2.1}(b) and \ref{Lemma2.2}(d) that $|W/W\cap Q_M|=4$. As $V\leq W\cap Q_M$, this implies $|W|\geq 2^5$.  Thus $W$ is extraspecial of order $2^5$. Since $\ov{L}$ has odd order and does not normalize $V$, there are at least three $\ov{L}$-conjugates of $V$ in $W$. Hence, $W$ contains more than $7+4=11$ involutions. Thus, $W \cong Q_8 \circ Q_8$ and so $\Out(W) \cong \O^+_4(2) \cong \Sym(3) \wr C_2$. Note that $\ov{M_2}$ embeds into $\Out(W)$, as $\Inn(W)=C_{\Aut(W)}(W/Z)$. In particular, as $\ov{L}$ is cyclic, we find that $\ov{L} \cong C_3$ and $\ov{H}\cong \Sym(3)$.

\smallskip

Clearly $[T,Y_M] \le V \le Q_2$, so that $\ov{T} \not \le \ov{L}$. At the same time, $\ov{T}=O^2(\ov{T})$ and $\ov{L}$ embed into $O^2(\Out(W))\cong C_3\times C_3$, so $\ov{TL} \cong C_3 \times C_3$. By the structure of $W$, we have $[W,Q_M] \le C_W(V) = V$. In particular, $[O^2(M),Q_M, W] \le V$ and Lemma \ref{Lemma2.4} gives $Y_M=Q_M$. This in implies $|S|=2^7$ and (a) and (b) are proven. Moreover, $Q_M=Y_M=Q_1$.

\smallskip

Lemmas~\ref{Lemma2.1}(g) and \ref{Lemma2.2}(d),(f) imply $Q_1\cap Q_2=Y_M \cap Q_2 = V = Y_M \cap W=Q_1\cap W$  and $Q_1W=Q_1Q_2=O_2(B)$. Then $Q_2=W=Q$ and $|\ov{S}|=|S/Q_2|=2^2$. It follows from Lemma~\ref{Lemma2.1}(b) that $M\cap \tilde{C}/O_2(M\cap \tilde{C})\cong \Sym(3)$. Hence, there exists $g\in S$ such that conjugation by $g$ invert $T/O_2(T)$. It follows from the above that $|\ov{T}|=3$, so $Q_2\cap T=O_2(T)$ and $\ov{g}$ inverts $\ov{T}$. As $\ov{LY_M} \cong \Sym(3)$ and $\ov{Y_M}$ centralizes $\ov{T}$, it follows that $\ov{M_2}=\ov{TLS} \cong \Sym(3) \times \Sym(3)$. As $Q$ is large, we have $C_\L(Q) \le Q$. So $\Out(Q) \cong \O^+_4(2)$ implies $|\tilde{C}:M_2| \le 2$. However, $S$ is a maximal $2$-subgroup of $\L$, as $(\L,\Delta,S)$ is a locality. Hence, $M_2 = \tilde{C}$. This gives (c) and (d).
\end{proof}

The reader might want to recall now Definition~\ref{D:pResidual}. We set
\[M^*=M\cap O^2(\L),\;\tilde{C}^*:=\tilde{C}\cap O^2(\L)\mbox{ and }S^*=S\cap O^2(\L).\]
By \cite[Lemma~3.1(c)]{Chermak:2022}, $S^*$ is a maximal $2$-subgroup of $O^2(\L)$. Thus, $S^*$ is a Sylow $2$-subgroup of $M^*$ and of $\tilde{C}^*$.

\smallskip

The proofs of (a) and (b) follow essentially the arguments given in \cite{MStr:2008}, whereas the main argument in the proof of (c) was provided to us by Ulrich Meierfrankenfeld in the case of groups to fill in a gap in the proof of \cite[Theorem~1]{MStr:2008} (see Remark~\ref{R:GapMStr} below). 

\begin{lemma}\label{L:M*C*}~
\begin{itemize}
\item [(a)] We have $M = \tilde{M}=N_\L(V) = N_\L(Y_M)$, $B=M\cap \tilde{C}$, $M^*=M'=O^p(M)$ has index $2$ in $M$, $M^*\cap Y_M=V$ and $M^*/V\cong \SL_3(2)$. Moreover, $S^*$ has index $2$ in $S$.
\item [(b)] $\tilde{C}^*$ has index $2$ in $\tilde{C}$ and contains $Q$. Moreover $\tilde{C}^*/Q\cong S_3\times C_3$.
\item [(c)] There exists an involution in $S^*\backslash Q$.
\item [(d)] The two subgroups of $Q$ which are isomorphic to $Q_8$ are normal in $S^*$.
\item [(e)] There is no elementary abelian subgroup of $S^*$ of order $2^4$.
\end{itemize}
\end{lemma}

\begin{proof}
\textbf{(a)} By Lemma \ref{Lemma2.1}(c) and (e) we get $M \le \tilde{M} = N_\L(V) \le N:=N_\L(Y_M)$, where $N$ is a subgroup of $\L$ of characteristic $2$, as $Y_M$ contains $V \in \Delta$ and is thus itself an element of $\Delta$. As $S\leq M$ is a maximal $2$-subgroup of $\L$, it is a Sylow $2$-subgroup of $N$. Recall that $Y_M=Q_M=O_2(M)$ by Proposition~\ref{Proposition2.16}(a). Hence, it follows that $Y_M=O_2(N)$. So $C_\L(V)=C_\L(Y_M)=C_N(Y_M)=Y_M$, where the first equality uses Lemma~\ref{Lemma2.1}(e) and the last equality uses that $N$ has characteristic $2$.

\smallskip

It follows from Proposition~\ref{Proposition2.16}(b) that $Q$ is generated by involutions. Thus there exists an involution in $S \setminus Y_M$. As $\SL_2(7)$ has a centre of order $2$ and the Sylow $2$-subgroups of $\SL_2(7)$ are generalized quaternion of order 16, the group $\SL_2(7)$ has no involution outside of the centre. Hence, we get that $M/V \not \cong \SL_2(7)$. Note now that $M/V$ is a central extension of $\SL_3(2)$ with a centre of order $2$. We use now that the Schur multiplier of $\SL_3(2)$ is cyclic of order 2 and $\SL_2(7)$ is the unique perfect central extension of $\SL_3(2)$ (cf. \cite[Example 6.9.11]{Weibel:1994a}). 
Hence, $M/V$ is not perfect. As $M/Y_M$ is perfect, it follows that the derived subgroup of $M/V$ has index $2$ in $M/V$ and does not contain $Y_M/V$. Since $V=[Y_M,M]\leq M'$, it follows that $M'$ has index $2$ in $M$ and $M'\cap Y_M=V$. In particular, 
\[
M'/V\cong M'Y_M/Y_M=M/Y_M\cong \SL_3(2).
\]
As $\SL_3(2)$ is simple and $V$ is an irreducible $M'$-module, it follows $M'=O^2(M)$. Thus, by Lemma~\ref{L:pResidualLoc}, $M'\subseteq O^2(\L)$ and so $M'\leq M^*$.

\smallskip

As a $7$-Sylow subgroup of $M$ acts coprimely on $Y_M$ and centralizes $Y_M/V$, we can pick an element $y \in Y_M \setminus V$ which is centralized by an element of order $7$. As $Q$ is large and $Z=Z(Q)\neq 1$, we have $N_\L(Z)\subseteq N_\L(Q)=\tilde{C}$. Thus, by Proposition~\ref{Proposition2.16}(c), there is no element of order $7$ centralizing $Z$. Note also that $Z\in \Delta$, since $(\L,\Delta,S)$ is $Q$-replete by Hypothesis~\ref{MainHyp}. Thus, $\<y\>$ is not $\F$-conjugate to $Z$ by Lemma~\ref{LocalitiesProp}(c).

\smallskip

Note now that $M$ acts transitively on $V^\#$. Furthermore, as $\tilde{C}/Q\cong \Sym(3)\times \Sym(3)$ by Lemma~\ref{Proposition2.16}(c), all involutions in $Q\backslash Z$ are $\tilde{C}$-conjugate. So all involutions in $Q$ are $\F$-conjugate to the involution in $Z$. As all involutions in $M'/V\cong \SL_3(2)$ are conjugate and $V\leq Q\leq M'$ with $Q/V\cong C_2\times C_2$, we see moreover that every involution in $M'$ is conjugate into $Q$. Hence, all involutions in $S\cap M'$ are $\F$-conjugate to the involution in $Z$. In particular   $y$ is not $\F$-conjugate to an element of $S\cap M'$. Now by \cite[Theorem~TL]{Lynd2014}, $y\not\in\mathfrak{foc}(\F)$, where $\mathfrak{foc}(\F)$ is the focal subgroup of $\F$ which contains the hyperfocal subgroup $\hyp(\F)$ of $\F$ (cf. \cite[p.32]{Aschbacher/Kessar/Oliver:2011}). Hence, $y\not\in \hyp(\F)=S^*$, where the latter equality uses Proposition~\ref{P:pResidualLocFus}. In particular, $y\not\in M^*$. We have seen above that $M'=O^2(M)\leq M^*$. As $M'$ has index $2$ in $M$, it follows that $M'=M^*$, $S\cap M'=S^*$ and $S=S^*\<y\>$. In particular, $\L=O^2(\L)S=O^2(\L)\<y\>\neq O^2(\L)$. As $V\leq M'=M^*$ and $Y_M\not\leq M^*$, we have also $V=Y_M\cap M^*$. 

\smallskip

Set now $N^*:=N\cap O^2(\L)$. Then $y\not\in N^*$ and $S^*\leq N^*$. Thus, $N^*\cap S=S^*$. By Lemma~\ref{L:pResidualLoc}, we have $O^2(N)\leq N^*$, so $N^*$ has index $2$ in $N$. Note also that $Y_M\cap N^*=Y_M\cap S^*=Y_M\cap M^*=V$ is normal in $N$. Since $Y_M=C_\L(V)=C_N(V)$ is the kernel of the action of $N$ on $V$, it follows that $C_{N^*}(V)=V$ and so $N^* / V$ is isomorphic to a subgroup of $\SL_3(2)$. Since $(M^* / V) \le (N^* / V)$, we get $N^*/V \cong \SL_3(2)$, yielding $|N^*| = |M^*|$ and therefore $|N| = |M|$. This shows $M=N=\tilde{M}$. In particular, $B=(M\cap \tilde{C})(L\cap \tilde{M})=M\cap \tilde{C}$ and the proof of (a) is complete.

\smallskip

\textbf{(b)} Note that $Y_M\cap Q=V$, as $Y_M\not\leq Q$ and $V\leq Q$ by Hypothesis~\ref{MainHyp} and since $V$ has index $2$ in $Y_M$. Moreover, $B$ acts irreducibly on $QY_M/Y_M$ by Lemma~\ref{Lemma2.1}(g) and Proposition~\ref{Proposition2.16}(d). Hence, $QY_M=[Q,B]Y_M$ and $Q=[Q,B](Q\cap Y_M)=[Q,B]V\leq M'=M^*\leq O^2(\L)$. In particular, $Q\leq \tilde{C}^*$. Since $S^*$ is a Sylow $2$-subgroup of $\tilde{C}^*$ and has index $2$ in $S$ by (a), and as $O^2(\tilde{C})\leq \tilde{C}^*$ by Lemma~\ref{L:pResidualLoc}, it follows that $\tilde{C}^*$ has index $2$ in $\tilde{C}$. By Proposition~\ref{Proposition2.16}(b),(c) $Q\cong Q_8\circ Q_8$ and $\tilde{C}/Q\cong S_3\times S_3$. As $Y_M$ centralizes $V\leq Q$, it follows that $y\in Y_M\backslash V$ interchanges the two subgroups of $Q$ isomorphic to $Q_8$ and thus centralizes a subgroup of $O^2(\tilde{C}/Q)$ of order $3$. As $Y_M\not\leq \tilde{C}^*$ by (a), it follows that $\tilde{C}^*/Q\cong \Sym(3)\times C_3$. 

\smallskip

\textbf{(c)} Note that $Q\cong Q_8 \circ Q_8$ is generated by involutions and $Q/V\cong C_2\times C_2$. Furthermore, $S^*/V\cong D_8$ and all involutions of $M^*/V$ are conjugate in $M^*/V$. Therefore, there is an involution in $S^*\backslash Q$.

\smallskip

\textbf{(d)} According to (c), we may pick an involution $t\in S^*\backslash Q$ and have then $S^*=Q\<t\>$. Assume now that the two subgroups of $Q$ which are isomorphic to $Q_8$ are not normal in $S^*$ and thus not normalized by $t$. Then $[Q,t]=C_Q(t)$ is elementary abelian of order $8$ and contains $Z=Z(Q)$. Hence, $A:=\<t^Q\>=[Q,t]\<t\>$ is elementary abelian of order $2^4$. Using that the subgroup $V$ is the unique elementary abelian subgroup of $Q$ of order $8$ which is centralized by a given element $y\in Y_M\backslash V$, one sees that $C_{\tilde{C}}(A)=A$. Since $Z\leq A$ and $C_\L(Z)\subseteq \tilde{C}$ as $\tilde{C}$ is large, we have indeed $C_\L(A)=C_{\tilde{C}}(A)=A$. 

\smallskip

Note that $A/A\cap V$ is isomorphic to the image of $A=[Q,t]\<t\>$ in $M^*/V$ and has thus order $4$. So $A\cap V$ has also order $4$ and equals $C_V(A)$. From this one sees that $N_{M^*}(A)$ acts transitively on $(A\cap V)^\#$ and $(A/A\cap V)^\#$, and that $N_{M^*}(A)/A\cong C_2\times \Sym(3)$. By (b), $\tilde{C}^*/Q\cong \Sym(3)\times C_3$. As $t$ is an involution swapping the two subgroups of $Q$ which are isomorphic to $Q_8$, one can see that $N_{\tilde{C}^*}(A)/A\cong \Alt(4)$ and that $N_{\tilde{C}^*}(A)$ acts transitively on $A\cap Q=[Q,t]$.  Note here that  
\[R:=N_{O^2(\L)}(A)\]
is a group, as $A$ contains $Z\in \Delta$ and is thus an element of $\Delta$. As $Z\leq A\cap V< A\cap Q<A$, we can conclude from the above that $R$ acts transitively on $A^\#$. Using that $C_\L(Z)=\tilde{C}$ as $Q$ is large and $Z=Z(Q)$ is central in $\tilde{C}$, we see moreover that
\[R_2:=C_R(Z)=N_{\tilde{C}^*}(A).\]
Thus, $R_2/A\cong \Alt(4)$ has order $12$. Now by the orbit stabilizer theorem, we have $|R:R_2|=|A^\#|$ and so $|R/A|=|R_2/A|\cdot |A^\#|= 12\cdot 15=180$. Note that $R_2\leq N_R(S^*)$ as $S^*=Q\<t\>=QA$. At the same time, since $C_A(S^*)=Z$, we have $N_R(S^*)\leq R_2$. Thus, $N_R(S^*)=R_2$. This shows that $R/A$ is a group of order $180$, in which the normalizer of a Sylow $2$-subgroup is isomorphic to $\Alt(4)$. However, by Lemma~\ref{L:GroupOrder180}, no such group exists and we have thereby obtained a contradiction. This proves (d).

\smallskip

\textbf{(e)} Recall that $|S^*/Q|=2$ and note that any elementary abelian subgroup of $Q\cong Q_8\circ Q_8$ has order at most $2^3$. Since $C_S(Q)\leq Q$ and $C_{\Aut(Q)}(Q/Z)=\Inn(Q)$, every element of $S\backslash Q$ acts non-trivially on $Q/Z$. Thus, any involution in $S\backslash Q$ which centralizes an elementary abelian subgroup of $Q$ of order $2^3$ cannot normalize the two subgroups of $Q$ which are isomorphic to $Q_8$. Hence, it follows from (d) that $S^*$ cannot contain an elementary abelian subgroup of order $2^4$.
\end{proof}

\begin{rmk}\label{R:GapMStr}
There is a small gap in the proof \cite[Theorem~1]{MStr:2008}, but corrections were kindly provided to us by Ulrich Meierfrankenfeld. Since these corrections are still unpublished, we include a remark on them here. The proof for groups is similar so far to what has been done by us in localities. Then Meierfrankenfeld and Stroth invoke Aschbacher's result \cite{Aschbacher:2002} to show that a certain normal subgroup $G^*$ of $G$ is isomorphic to $\G_2(3)$. Here $G$ is a finite group having similar properties as our locality $\L$, and similar notation for subgroups $M$, $\tilde{C}$, $S$ etc. is used (see the statement of Theorem~1 and the notation introduced at the beginning of Section~2 in \cite{MStr:2008}). It is easy to see that the normal subgroup $G^*$ is in fact equal to $O^2(G)$. One can similarly define $M^*:=M\cap G^*$, $\tilde{C}^*:=\tilde{C}\cap G^*$ and $S^*=S\cap G^*=S\cap M^*$. To prove that the hypothesis of \cite{Aschbacher:2002} holds, one needs to see that $\tilde{C}^*$ has two normal subgroups $H_1$ and $H_2$ with $H_1\cong H_2\cong \SL_2(3)$ which form a central product. It is easy to construct subgroups of $\tilde{C}^*$ which are candidates for $H_1$ and $H_2$, but what is missing in \cite{MStr:2008} is an argument to show that these two subgroups are normal in $\tilde{C}^*$. To see this, one only needs to check that the two subgroup of $Q$ which are isomorphic to $Q_8$ are normal in $S^*$. This can be shown using a similar argument as given above in the proof of Lemma~\ref{L:M*C*}(c),(d). Indeed, this is essentially the argument provided to us by Meierfrankenfeld.
\end{rmk}

Unlike Meierfrankenfeld and Stroth \cite{MStr:2008}, we are not able to use Aschbacher's result \cite{Aschbacher:2002}, since this would presuppose the existence of a finite group $G$ with $\L=\L_\Delta(G)$. Thus, we will need to do a certain amount of extra work in order to determine the structure of $\tilde{M}=N_\L(V)$, of $\tilde{C} = N_\L(Q)$ and then of the fusion system $\F$.




\begin{lemma}\label{L:DT}
 Let $D$ be a Sylow $3$-subgroup of $\tilde{C}$ and set $T:=N_S(D)$. Then the following hold.
 \begin{itemize}
  \item [(a)] $O^2(\tilde{C}) =D \ltimes Q$ and $D$ acts fixed-point freely on $Q/Z(Q)$;
  \item [(b)] $S=QT$ and $\tilde{C}=O^2(\tilde{C})T$;
  \item [(c)] $T\cong D_8$ and $T\cap O^2(\tilde{C})=T\cap Q=Z=Z(T)$;
  \item [(d)] $Y_M\cap T\cong C_2\times C_2$ and $S^*\cap T\cong C_2\times C_2$ are both not contained in $Q$.
 \end{itemize}
\end{lemma}

\begin{proof}
\textbf{(a)} As $Q$ is large, we have $C_{\tilde{C}}(Q)\leq Q$. Moreover, by Proposition~\ref{Proposition2.16}(b), $O^2(\tilde{C}/Q)\cong C_3\times C_3$ and $|S/Q|=4$. It follows now from the structure of $\Aut(Q)$ that $D$ acts fixed-point freely on $Q/Z(Q)$, $Q=[Q,D]\leq O^2(\tilde{C})$ and $O^2(\tilde{C})=D \ltimes Q$. 

\smallskip

\textbf{(b)} As $[D,N_Q(D)]\leq D\cap Q=1$, it follows that $N_Q(D)=C_Q(D)=Z$. Note that $\Syl_3(\tilde{C})=\Syl_3(O^2(\tilde{C}))$. So $O^2(\tilde{C})$ acts transitively on $\Syl_3(\tilde{C})$ and thus $Q$ acts transitively on $\Syl_3(\tilde{C})$. The general Frattini argument \cite[3.1.4]{Kurzweil/Stellmacher:2004} yields therefore that $S=TQ=QT$. This implies that $\tilde{C}=O^2(\tilde{C})S=O^2(\tilde{C})T$.

\smallskip

\textbf{(c)} Since $S=TQ$, $|S/Q|=4$ and $|T\cap Q|=|Z|=2$, it follows that $T$ has order $8$. By Lemma~\ref{Lemma2.1}(b) and Proposition~\ref{Proposition2.16}(a), we have $B/Y_M\cong S_4$. In particular, a Sylow $3$-subgroup $D_0$ of $B$ has order $3$. Now such a Sylow $3$-subgroup $D_0$ of $B$ is conjugate into $D\cap B$ under $Q\leq B$, since all Sylow $3$-subgroups of $\tilde{C}$ are conjugate under $Q$. Thus, $D\cap B$ has order $3$. As $B/Y_M\cong S_4$, we have thus $|N_S(D\cap B)Y_M/Y_M|\leq 2$.

\smallskip

Note that $T\leq N_S(D\cap B)$. Moreover, $T\cap V=Z$, as $T\cap Q=N_Q(D)=Z$ and $Z\leq V\leq Q$. In particular, $|T\cap Y_M|\leq 4$. Since $T$ has order $8$, it follows that
\[TY_M=N_S(D\cap B)Y_M,\;|T/T\cap Y_M|=2\mbox{ and }|T\cap Y_M|=4.\]
As $T\cap Y_M$ is elementary abelian, we see in particular that $T$ is not quaternion. As $T\cap V=Z$, there exists an element in $(T\cap Y_M)\backslash V$. It follows thus from \cite[Lemma~1.1(c)]{MStr:2008} that $C_S(T\cap Y_M)\leq Y_M$ and so $T$ acts non-trivially on $T\cap Y_M$. Thus, $T\cong D_8$ and $|Z(T)|=2$. Since $Z=Z(S)\leq Z(T)$, it follows in particular that $Z=Z(T)$. As $O^2(\tilde{C})/Q$ is a $3$-group, we have moreover $ O^2(\tilde{C})\cap T\leq Q\cap T=Z$. The converse inclusion is clear and so (c) holds.

\smallskip

\textbf{(d)} Since $T\cap Q=Z$ has order $2$, it is sufficient to show that $T\cap Y_M$ and $T\cap S^*$ are fours groups. Indeed, we have seen already in the proof of (c) that $T\cap Y_M$ is elementary abelian of order $4$. It remains thus to show that $S^*\cap T\cong C_2\times C_2$.

\smallskip

As seen in the proof of (c) there exists an element $y\in (Y_M\cap T)\backslash V$. Such an element $y$ centralizes $V\leq Q$ and swaps thus the two subgroups of $Q$ which are isomorphic to $Q_8$. Hence, it follows from Lemma~\ref{L:M*C*}(d) that $S^*\cap T$ is properly contained in $T$. Since $S^*$ has index $2$ in $S$, we can conclude that $S^*\cap T$ has in fact index $2$ in $T$ and  thus order $4$. As $Z=T\cap Q\leq S^*\cap T$, it remains to argue that there is an involution in $(S^*\cap T)\backslash Z$.

By Lemma~\ref{L:M*C*}(c), there is an involution in $S^*\backslash Q$. Since $|S^*/Q|=2$, this means that $Q$ has a complement in $S^*$. Hence, by a Theorem of Gaschu{\"u}tz \cite[]{Kurzweil/Stellmacher:2004}, $Q$ has a complement $K$ in $B\cap M^*$. By Lemma~\ref{L:M*C*}(a), $M^*/V\cong \SL_3(2)$, $V=Y_M\cap M^*$ is a natural $\SL_3(2)$-module for $M^*/V$ and $Q\leq M^*$. Hence, it follows from Lemma~\ref{Lemma2.1}(b) that $(B\cap M^*)/V\cong \Sym(4)$ and $Q=O_2(B\cap M^*)$. Thus, $K\cong (B\cap M^*)/Q\cong \Sym(3)$. Thus, there is a Sylow $3$-subgroup of $B\cap M^*$ which is normalized by an involution in $S^*\backslash Q$. We have seen above that $D\cap B$ is a Sylow $3$-subgroup of $B$. As $M^*\cap B$ has index $2$ in $B$, it follows that $D\cap B$ is also a Sylow $3$-subgroup of $B\cap M^*$. As $Q$ acts transitively on $\Syl_3(B\cap M^*)$, it follows that there is an involution in $N_{S^*}(D\cap B)\backslash Q$. As $Z< S^*\cap T= N_{S^*}(D)\leq N_{S^*}(D\cap B)$ and $N_{S^*}(D\cap B)/Z\cong N_{S^*}(D\cap B)V/V$ has order at most $2$, it follows that $S^*\cap T=N_{S^*}(D\cap B)$ has order $4$ and contains an involution outside of $Q$. This shows the assertion.  
\end{proof}

We fix now $D\in \Syl_3(\tilde{C})$ and set $T:=N_S(D)$. Let $P_1$ and $P_2$ the two different subgroups of $Q$ which are isomorphic to $Q_8$. Thus, 
\[Q=P_1\circ P_2\mbox{ and } O^2(\tilde{C})=D\ltimes Q,\]
where $D$ acts faithfully on $Q$. Then $D=D_1\times D_2$ where $D_i\cong C_3$ acts non-trivially on $P_i$ and centralizes $P_{3-i}$ for $i=1,2$. Set
\[H_i:=D_iP_i\mbox{ for }i=1,2.\]

\begin{lemma}\label{L:StructuretildeC}
\begin{itemize}
 \item [(a)] $H_i=D_i\ltimes P_i\cong \SL_2(3)$ for $i=1,2$ and $O^2(\tilde{C})=H_1\circ H_2$.
 \item [(b)] The structure of $\tilde{C}$ is uniquely determined up to isomorphism. More precisely, if $D_1=\<d_1\>$, $y\in (Y_M\cap T)\backslash Q$ and $t\in (T\cap S^*)\backslash Q$ is an involution, then $t$ inverts $D=D_1\times D_2$ and there exist $i,j\in P_1$ such that
 \[P_1=\<i,j\>,\;P_2=\<i^y,j^y\>,D_2=\<d_1^y\>.\]
 and $t$ swaps $i$ and $j$ as well as $i^y$ and $j^y$.
\end{itemize}
\end{lemma}

\begin{proof}
\textbf{(a)} Clearly, $H_i=D_i\ltimes P_i\cong \SL_2(3)$. Note that $[P_1,P_2]=1$ and that $[D_1,D_2]=1$. Moreovoer, $D_i$ is chosen such that $[D_i,P_{3-i}]=1$ for $i=1,2$. Hence, $H_1$ and $H_2$ centralize each other. In particular, $H_1\cap H_2\leq Z(H_1)=Z$. Clearly also $Z\leq P_1\cap P_2\leq H_1\cap H_2$ and so $H_1\cap H_2=Z$. This proves that $O^2(\tilde{C})=H_1\circ H_2$. 

\smallskip

\textbf{(b)} To determine the structure of $\tilde{C}$, we use now that, by Lemma~\ref{L:DT}(b),(c), $\tilde{C}=O^2(\tilde{C})T$ and $T\cong D_8$ with $Z(T)=Z=O^2(\tilde{C})\cap T=Q\cap T$. In particular, it is enough to determine the action of $T$ on $O^2(\tilde{C})$. By Lemma~\ref{L:DT}(d), $T\cap Y_M$ and $T\cap S^*$ are the two fours groups in $T$, and these are not contained in $Q$. So the elements $y$ and $t$ always exist, and they generate $T$. Thus, we only need to determine the actions of $y$ and $t$ on $O^2(\tilde{C})$.

\smallskip

As $y$ centralizes with $V$ an elementary abelian subgroup of $Q$ of order $8$, $y$ must swap $P_1$ and $P_2$. Thus, $y$ swaps also $D_1=C_D(P_2)$ and $D_2=C_D(P_1)$. So with $d_1$ as in (b), we have $D_2=\<d_1^y\>$.

\smallskip

By Lemma~\ref{L:M*C*}(d), the involution $t$  normalizes $P_1$ and $P_2$. As $T\cong D_8$, we have $[y,t]\leq Z(T)=Z=Z(O^2(\tilde{C}))$. Thus, the action of $t$ on $O^2(\tilde{C})$ commutes with the action of $y$ on $O^2(\tilde{C})$. Since $\Inn(Q)=C_{\Aut(Q)}(Q/Z(Q))$ and $C_{\tilde{C}}(Q)\leq Q$, $t$ acts non-trivially on $Q/Z=Q/Z(Q)$. Hence, $t$ acts also non-trivially on $P_1/Z$ and $P_2/Z$. Thus, we may pick generators $i,j$ of $P_1$ which are swapped by $t$. Then $P_2=P_1^y=\<i^y,j^y\>$. Moreover, as the action of $t$ on $Q$ commutes with the action of $y$ on $Q$, the involution $t$ interchanges also $i^y$ and $j^y$ by conjugation.

\smallskip

As $\Aut(P_i)\cong S_4$ and $O_2(\Aut(P_i))=\Inn(P_i)=C_{\Aut(P_i)}(P_i/Z)$ for $i=1,2$, it follows that the automorphism of $P_i$ induced by $t$ inverts a $3$-Sylow subgroup of $\Aut(P_i)$, and such a $3$-Sylow subgroup has order $3$. Since $t$ normalizes $D_i$, it follows that $t$ must invert $D_i$ for $i=1,2$.  Hence, $t$ inverts $D=D_1\times D_2$. 
\end{proof}

We now consider the amalgam $(M,\tilde{C},B,i,j)$, where $i\colon B\hookrightarrow M$ and $j\colon B\hookrightarrow \tilde{C}$ are the inclusion maps. Indeed, we will show that this amalgam is isomorphic to an amalgam found in $\Aut(\G_2(3))$. We introduce this amalgam and a suitable locality containing it in the following lemma.

\begin{lemma}\label{L:G23}
 Let $G\cong \Aut(\G_2(3))$. Then $G$ has a large $2$-subgroup $Q_G$ with $Q_G\leq F^*(G)$, 
 \[\tilde{C}_G=N_G(Q_G)\mbox{ is a maximal subgroup of $G$},\]
and $Q_G=O_2(\tilde{C}_G)$. Moreover, there exists a maximal subgroup $M_G$ of $G$ such that 
 \[O_2(M_G)=Y_{M_G}\mbox{ has order }2^4,\;M_G/O_2(M_G)\cong \SL_3(2)\] 
and $V_G:=[Y_{M_G},M_G]$ is a natural $\SL_3(2)$-module for $M_G/O_2(M_G)$. Furthermore, $M_G$ and $Q_G$ can be chosen such that there exists $S_G\in\Syl_2(G)$ with 
\[S_G\leq M_G\cap \tilde{C}_G.\]
Then $Y_{M_G}\not\leq Q_G$ and $V_G\leq Q_G$. In particular, setting 
\[\Delta_G:=\{P\leq S_G\colon N_G(P)\mbox{ is of characteristic $2$}\}\] 
and $\L_G:=\L_{\Delta_G}(G)$, Hypothesis~\ref{MainHyp} is fulfilled with $(\L_G,\Delta_G,S_G,M_G,V_G,Q_G,\tilde{C}_G)$ in place of $(\L,\Delta,S,M,V,Q,\tilde{C})$. Moreover, $\F_{S_G}(\L_G)=\F_{S_G}(G)$.  
\end{lemma}

\begin{proof}
According to the ATLAS, a Sylow $2$-subgroup of $G$ has order $2^7$, and there are maximal subgroups $\tilde{C}_G$ and $M_G$ of $G$ such that $\tilde{C}_G\cong 2^{1+4}.(S_3\times S_3)$ and $M_G\cong 2^3.L_3(2)\colon 2$. So we may choose $\tilde{C}_G$ and $M_G$ such that they contain a common Sylow $2$-subgroup $S_G$ of $G$. Setting $Q_G:=O_2(\tilde{C}_G)$, we have $N_G(Q_G)=\tilde{C}_G$ because of the maximality of $\tilde{C}_G$. Note that $Z_G:=Z(Q_G)$ has order $2$ and $\tilde{C}_G\leq N_G(Z_G)$. Hence, the maximality of $\tilde{C}_G$ yields also that $N_G(Z_G)=\tilde{C}_G=N_G(Q_G)$. This proves that $Q_G$ is a large subgroup of $G$ which is extraspecial of order $2^{1+4}$. 
 
\smallskip
 
As $M_G\cong 2^3.L_3(2)\colon 2$, setting $M_G^*:=O^2(M_G)$ and $V_G:=O_2(M_G^*)$, we have $M_G^*/V_G\cong L_3(2)$ and $V_G$ is elementary abelian of order $2^3$. As $\Aut(V_G)\cong L_3(2)$, it follows that $V_G$ is a natural $L_3(2)$-module and $M_G=M_G^*Y$ for $Y:=C_{M_G}(V_G)$. As $Y\cap M_G^*=C_{M_G^*}(V_G)=V_G$, we have $|Y:V_G|=2$ and $|Y|=2^4$. In particular, $Y$ is abelian, since $V_G$ is centralized by $Y$. The fact that $[M_G^*,Y]\leq C_{M_G^*}(V_G)=V_G\leq Y$ implies also that $Y$ is normal in $M_G$. Thus, $Y=O_2(M_G)$. Since $Y$ contains an elementary abelian subgroup of order $2^3$, $\Phi(Y)\leq V_G$ has order at most $2$ and is then trivial since $V_G$ is an irreducible $M_G$-module.  So $Y$ is elementary abelian and thus equal to $Y_{M_G}$. Clearly, $V_G=[Y,M_G]$ is a natural $L_3(2)$-module for $M_G/Y\cong M_G^*/V_G\cong L_3(2)$. 

\smallskip

Note that $Q_G=[Q_G,O^2(\tilde{C}_G)]\leq O^2(G)=F^*(G)$ and $V_G\leq M_G^*\leq O^2(G)$. Moreover, $S_G/Q_G\cong C_2\times C_2$. As $Y$ is elementary abelian and not contained in $O^2(G)$, it follows that $|YQ_G/Q_G|=2$ and $Y\cap Q_G=Y_G\cap O^2(G)$ has order $2^3$. Since $V_G\leq Y\cap O^2(G)$, we can conclude that $V_G=Y\cap Q_G\leq Q_G$. 

\smallskip

By \cite[Lemma~10.2]{Henke:2015}, $\Delta_G$ is closed under passing to $\F_{S_G}(G)$-conjugates and overgroups in $S$. As $N_G(Z_G)=\tilde{C}_G$ is of characteristic $2$, we have $C_{S_G}(Q_G)=Z(Q_G)=Z_G\in\Delta_G$. Thus, property (ii) of Lemma~\ref{L:QReplete} holds and $(\L_G,\Delta_G,S_G)$ is $Q_G$-replete. By Lemma~\ref{L:QRepleteSaturated}, we have thus $\F_{S_G}(\L_G)^c\subseteq \Delta_G$. By definition of $\Delta_G$, for every $P\in\Delta_G$, $N_{\L_G}(P)=N_G(P)$ is of characteristic $2$. Hence, $(\L_G,\Delta_G,S_G)$ is a linking locality. Moreover, as $\F_{S_G}(\L_G)\subseteq \F_G:=\F_{S_G}(G)$ are both fusion systems over $S_G$, we have $\F_G^c\subseteq \F_{S_G}(\L_G)^c\subseteq \Delta$. Hence, it follows from Alperin's fusion theorem that $\F_G=\F_{S_G}(\L_G)$. As $Q_G$ is large in $G$, it is large in $\L_G$ by Example~\ref{E:LargeGLargeLDeltaG}. Thus, Hypothesis~\ref{MainHyp} is indeed fulfilled. 
\end{proof}

From now on we pick $G\cong \Aut(\G_2(3))$ as well as the subgroups $Q_G$, $S_G$, $\tilde{C}_G$, $M_G$, $V_G$ as in  
Lemma~\ref{L:G23}. Moreover, set 
\[B_G:=\tilde{C}_G\cap M_G.\]
Write 
\[i\colon B\rightarrow M,\;j\colon B\rightarrow \tilde{C},\;i_G\colon B_G\rightarrow M_G,\;j_G\colon B_G\rightarrow \tilde{C}_G\]
for the inclusion maps. Set 
\[M_G^*:=O^2(M_G)\mbox{ and }S_G^*:=M_G^*\cap S_G.\]

\smallskip

Our goal will be to show that the amalgam $(B,M,\tilde{C},i,j)$ is isomorphic to the amalgam $(B_G,M_G,\tilde{C}_G,i_G,j_G)$. That amounts to finding isomorphisms
\[
\alpha_B \colon B \longrightarrow B_G, \qquad
\alpha_M \colon M \longrightarrow M_G, \qquad
\alpha_C \colon \wde{C} \longrightarrow \wde{C}_G
\]
that make the two squares in the following diagram commute.
\begin{center}
	\begin{tikzpicture}
		\node (B) at (0,0) {$B$};
		\node (M) at (1.7,1.3) {$M$};
		\node (C) at (2.5,-0.7) {$\wde{C}$};
		\node (BG) at (5,0.5) {$B_G$};
		\node (MG) at (6.7,1.8) {$M_G$};
		\node (CG) at (7.5,-0.2) {$\wde{C}_G$};
		
		\draw[->] (B) to node[above] {$i$} (M);
		\draw[->] (B) to node[below] {$j$} (C);
		\draw[->] (BG) to node[pos=0.3,right] {$\; i_G$} (MG);
		\draw[->] (BG) to node[pos=0.7,above] {$j_G$} (CG);
		\draw[->,dashed] (B) to node[above] {$\alpha_B$} (BG);
		\draw[->,dashed] (M) to node[above] {$\alpha_M$} (MG);
		\draw[->,dashed] (C) to node[below] {$\alpha_C$} (CG);
	\end{tikzpicture}
\end{center}
Since the solid arrows are inclusions, all we need is a pair of isomorphisms $(\alpha_M,\alpha_C)$ as above such that their restrictions to $B$ coincide. The strategy will consist in starting with isomorphisms $\alpha_M \colon M \longrightarrow M_G$ and $\beta: \wde{C} \longrightarrow \wde{C}_G$ and show that there exists a replacement of $\beta$ by an isomorphism that agrees with $\alpha_M$ on $B$, the main tool being Lemma~\ref{L:AutBExtend}.

\begin{lemma}\label{L:IsoC}
 There is an isomorphism $\tilde{C}\rightarrow \tilde{C}_G$ which takes $V$ to $V_G$, $S^*$ to $S_G^*$, and $B$ to $B_G$.
 \end{lemma}

\begin{proof}
We use that, by Lemma~\ref{L:G23}, there exists a locality $(\L_G,\Delta_G,S_S)$ such that Hypothesis~\ref{MainHyp} is fulfilled with $(\L_G,\Delta_G,S_G,M_G,V_G,Q_G,\tilde{C}_G)$ in place of $(\L,\Delta,S,M,V,Q,\tilde{C})$. In particular, everything we have proved for $S$, $M$, $V$, $Q$ and $\tilde{C}$ applies also to $S$, $M$, $V$, $Q$, $\tilde{C}$. Lemma~\ref{L:M*C*}(a) gives thus also that $O^2(M_G)=M_G\cap O^2(\L_G)$, and so everything we proved for $M^*$ and $S^*$ applies similarly to $M_G^*$ and $S_G^*$. Moreover, Lemma~\ref{L:StructuretildeC}(b) gives us the existence of an isomorphism $\tilde{C}\rightarrow \tilde{C}_G$. 

\smallskip

Any isomorphism $\tilde{C}\rightarrow \tilde{C}_G$ will take $Q=O_2(\tilde{C})$ to $Q_G=O_2(\tilde{C}_G)$ and thus $Z=Z(Q)$ to $Z(Q_G)=Z(S_G)$. Using the notation as in Lemma~\ref{L:StructuretildeC}, we have $S^*=Q\<t\>$ and $V=Z\<ii^y,jj^y\>$, as $V$ is the unique elementary abelian subgroup of $Q$ of order $2^3$ which is centralized by $y$. Again, similar results hold accordingly for $S_G^*$ and $V_G$. Thus, Lemma~\ref{L:StructuretildeC} allows us indeed to choose an isomorphism $\tilde{C}\rightarrow \tilde{C}_G$ which takes $V$ to $V_G$ and $S^*$ to $S_G^*$.

\smallskip

By Lemma~\ref{Lemma2.3}(b) and Proposition~\ref{Proposition2.16}(d), we have $B=N_{\tilde{C}}(V)$. Recall that $M_G$ is a maximal subgroup of $G$, and $V_G$ is normal in $M_G$, but not in $G$. Hence, $M_G=N_G(V)$ and $B_G=\tilde{C}_G\cap M_G=N_{\tilde{C}}(V_G)$. Thus, an isomorphism $\tilde{C}\rightarrow \tilde{C}_G$ which takes $V$ to $V_G$ will also take $B$ to $B_G$.
\end{proof}

\begin{lemma}\label{L:McongMG}
 We have $M\cong M_G$.
\end{lemma}

\begin{proof}
Lemma~\ref{Lemma2.1}(a) allows us to apply Lemma~\ref{L:DetermineMfromM*} with $M$ in place of $G$ to conclude that $M\cong \Aut(O^2(M))=\Aut(M^*)$. By Lemma~\ref{L:G23} we have then similarly $M_G\cong \Aut(M_G^*)$. Thus, it is sufficient to argue that $M^*\cong M_G^*$. 

\smallskip

By Lemma~\ref{L:IsoC}, there exists an isomorphism $\gamma\colon S^*\rightarrow S_G^*$ with $V\gamma=V_G$. As $V$ is normal and centric in $\F_{S^*}(M^*)$, $V_G=V\gamma$ is also centric and normal in  $\F_{S^*}(M^*)^\gamma$. Note here that $\F_{S^*}(M^*)^\gamma$ is a fusion system over $S_G^*$. Moreover, $V_G\unlhd M_G^*$, $C_{M_G^*}(V_G)\leq V_G$ and $S_G^*\in\Syl_2(M_G^*)$. As $\Aut(V_G)=\Aut_{M_G^*}(V_G)$ and $\Aut(V)=\Aut_{M^*}(V)$, we have moreover $\Aut_{\F_{S^*}(M^*)^\gamma}(V_G)=\Aut_{M^*}(V)^\gamma=\Aut(V)^\gamma=\Aut(V_G)=\Aut_{M_G^*}(V_G)$. Hence, it follows from \cite[Theorem~III.5.10(b)]{Aschbacher/Kessar/Oliver:2011} that there exists $\beta\in\Aut(S_G^*)$ with $\F_{S_G^*}(M_G^*)=(\F_{S^*}(M^*)^\gamma)^\beta=\F_{S^*}(M^*)^{\gamma\beta}$ and $\beta|_{V_G}=\id$. 

\smallskip

Choose a set $M^+$ containing $S_G^*$ such that $|M^+|=|M^*|$. We can extend $\gamma\beta$ to a bijection $\alpha\colon M^*\rightarrow M^+$. There is now a unique way to turn $M^+$ into group such that $\alpha$ becomes a group isomorphism. Then
\[\F_{S_G^*}(M_G^*)=\F_{S^*}(M^*)^{\gamma\beta}=\F_{S_G^*}(M^+).\]
As $V\unlhd M^*$ and $C_{M^*}(V)=V$, the subgroup $V_G=V_G\beta=V\gamma\beta=V\alpha$ is normal and self-centralizing in $M^+$. Using \cite[Theorem~III.5.10(c)]{Aschbacher/Kessar/Oliver:2011}, one can conclude now that $M_G^*\cong M^+\cong M^*$. As argued above, this implies the assertion. 
\end{proof}

\begin{lemma}\label{ActionOnYM}
 We have $C_{Y_M}(x)\leq V$ for any $x\in S\backslash Y_M$. 
\end{lemma}

\begin{proof}
This follows from Lemma~\ref{Lemma2.1}(a) and \cite[Lemma~1.1(c)]{MStr:2008}.
\end{proof}

\begin{lemma}\label{L:AutBExtend}
 Every element of $\Aut(B)$ extends to an automorphism of $\tilde{C}$.
\end{lemma}

\begin{proof}
 Let $\alpha\in\Aut(B)$. We will show that there exists an automorphism of $\tilde{C}$ whose restriction is $\alpha$. We use throughout that, by Lemma~\ref{Lemma2.1}(b) and Proposition~\ref{Proposition2.16}(a), we have $B/Y_M\cong S_4$ and $O_2(B)=QY_M$, In particular, $\alpha$ normalizes $QY_M$ and thus also $Y_M$, as $Y_M$ is the unique elementary abelian subgroup of $QY_M$ of order $2^4$. Hence, $\alpha$ induces an automorphism of $B/Y_M\cong S_4$. As $\Aut(S_4)=\Inn(S_4)$, replacing $\alpha$ by $\alpha$ composed with a suitable inner automorphism of $B$, we may assume 
 \[[B,\alpha]\leq Y_M.\]
 Notice that $\alpha$ normalizes $V=[Y_M,B]$ and $Z=Z(B)$. Thus, $\alpha$ induces an automorphism of $V/Z$, which is an irreducible $B$-module. As $[B,\alpha]\leq Y_M=C_B(V)$, it follows that the automorphism of $V/Z$ induced by $\alpha$ is a non-zero element of $\End_{B}(V/Z)\cong\mF_2$. Thus, $[V,\alpha]\leq Z$ and so $[V,\alpha,B]=1$. Notice also $[\alpha,B,V]\leq[Y_M,V]=1$. Therefore, the Three-Subgroup-Lemma gives $[B,V,\alpha]=1$. As $[V,B]=V$, this shows
\[[V,\alpha]=1.\]
In particular, $|[Y_M,\alpha]|=|Y_M/C_{Y_M}(\alpha)|\leq 2$. As $[B,\alpha]\leq Y_M$, $B$ normalizes $[Y_M,\alpha]$ and hence, $[Y_M,\alpha]\leq Z(B)=Z$. Let $1\neq z\in Z$ and $y\in Y_M\backslash V$. Then $y^\alpha=y$ or $y^\alpha=yz$. As $\tilde{C}/O^2(\tilde{C})\cong C_2\times C_2$ and $y$ is an involution, there exists a subgroup $C_0\leq\tilde{C}$ of index $2$ with $\tilde{C}=\<y\>\ltimes C_0$. Hence, by Lemma~\ref{AutSemidirect}, there exists $\beta\in\Aut(\tilde{C})$ such that $[C_0,\beta]=1$ and $y^\beta=yz$. Notice that $[\tilde{C},\beta]\leq Z\leq B$ so $B^\beta=B$ and $[B,\alpha\beta|_B]\leq Y_M$. Hence, if $y^\alpha=yz$ we may replace $\alpha$ by $\alpha\beta|_B$ and can so assume in any case that $y^\alpha=y$. So we may suppose that
\[[Y_M,\alpha]=1.\]
Recall $[B,\alpha]\leq Y_M$, so an element in $\<\alpha\>$ of odd order will act trivially on every Sylow $2$-subgroup of $B$ and thus on $B=O^{2^\prime}(B)$. Hence, $\alpha$ is a $2$-element and $U:=\<\alpha\>\times Y_M$ is an abelian $2$-subgroup of $\<\alpha\>\ltimes B$ on which $B/Y_M\cong S_4$ acts. Note also $[U,O^2(B)]=V$ as $[U,B]\leq Y_M$ and $[Y_M,B]\leq V=[V,O^2(B)]$. Moreover, $C_{Y_M}(t)\leq V$ for every involution $t\in O^2(B/Y_M)$ by Lemma~\ref{ActionOnYM}. So applying Lemma~\ref{A4Module} with $G=O^2(B/Y_M)$ gives $U=Y_MC_U(O^2(B))$. Every element of $Y_M$ induces by conjugation an automorphism of $\tilde{C}$ centralizing $B/Y_M$. So we may assume that $\alpha$ centralizes $O^2(B)$. Then $\alpha$ centralizes $O^2(B)Y_M$ and $[B,\alpha]\leq C_{Y_M}(O^2(B))=Z$. Since $B/O_2(B) = C_M(Z) / O_2(B) \cong \Sym(4)$ and $Y_M \norm B$, we may pick an involution $t\in B\backslash O^2(B)Y_M$. Then $[t,\alpha]\leq Z$, i.e. $t^\alpha=t$ or $t^\alpha=tz$. If $t^\alpha=t$ the $\alpha=\id_B$ and we are done. So we may assume $t^\alpha=tz$. Note that $\tilde{C}=\<t\>\ltimes (O^2(\tilde{C})Y_M)$, so by Remark~\ref{AutSemidirect}, there exists an automorphism $\beta\in\Aut(\tilde{C})$ centralizing $O^2(\tilde{C})Y_M$ with $t^\beta=tz$. Then $B^\beta=B$ and $\beta|_B=\alpha$, so the assertion follows. 
\end{proof}

\begin{prop}\label{P:AmalgamsIso}
 The amalgam $(B,M,\tilde{C},i,j)$ is isomorphic to the amalgam $(B_G,M_G,\tilde{C}_G,i_G,j_G)$. 
\end{prop}

\begin{proof}
By Lemma~\ref{L:McongMG} and Sylow's Theorem, we can choose an isomorphism $\alpha_M\colon M\rightarrow M_G$ with $S\alpha=S_G$. Note that $\alpha$ will take $Z=Z(S)$ to $Z(S_G)$. 

\smallskip

Applying first Lemma~\ref{L:M*C*}(a) and then Lemma~\ref{Lemma2.1}(b), we see that  $B=M\cap \tilde{C}=N_M(Z)$. Similarly, Lemma~\ref{L:G23} allows us to apply Lemma~\ref{Lemma2.1}(b) to conclude $B_G=N_{M_G}(Z(S_G))$. Thus, $B\alpha_M=B_G$ and we may define $\alpha\colon B\rightarrow B_G$ to be the isomorphism which is obtained as the restriction of $\alpha_M$.

\smallskip

By Lemma~\ref{L:IsoC}, there is an isomorphism $\beta\colon \tilde{C}\rightarrow \tilde{C}_G$ with $B\beta=B_G$. Note now that $\alpha\beta^{-1}\in\Aut(B)$. Thus, by Lemma~\ref{L:AutBExtend}, $\alpha\beta^{-1}$ extends to an automorphism $\gamma\in\Aut(\tilde{C})$. Then 
\[\alpha_C:=\gamma\beta\colon \tilde{C}\rightarrow \tilde{C}_G\] 
is an isomorphism which restricts to $\alpha$. 

\smallskip

As $i\colon B\rightarrow M$, $j\colon B\rightarrow \tilde{C}$, $i_G\colon B_G\rightarrow M_G$ and $j_G\colon B_G\rightarrow \tilde{C}_G$ are defined as the inclusion maps, we obtain $i\alpha_M=\alpha_M|_B=\alpha i_G$ and $j\alpha_C=\alpha_C|_B=\alpha j_G$. This shows the assertion.
\end{proof}

\subsection{Generation of the fusion system and essential subgroups}

Set $\F:=\F_S(\L)$. Note that $\F$ is saturated as $(\L,\Delta,S)$ is a linking locality. Moreover, $Q$ is large in $\F$ by Lemma~\ref{L:LargeLF}(a). The goal of this section is to prove the following proposition.

\begin{prop}\label{Generate}
 $\F=\<N_\F(Q),N_\F(V)\>$.
\end{prop}

As $\Aut_\F(S)$ normalizes $Z=Z(S)$ and since $Q$ is large in $\F$, we have $\Aut_\F(S)\subseteq N_\F(Z)\subseteq N_\F(Q)$. Moreover, by the Alperin--Goldschmidt Fusion Theorem \cite[Theorem 3.5]{Aschbacher/Kessar/Oliver:2011}, the fusion system $\F$ is generated by the $\F$-automorphism groups of $S$ and of all the essential subgroups of $\F$. Hence, in order to prove Proposition~\ref{Generate}, it is enough to prove the following lemma.

\begin{lemma}\label{EssentialSubgroups}
 For every essential subgroup $R$ of $\F$, we have $\Aut_\F(R)\subseteq N_\F(Q)$ or $\Aut_\F(R)\subseteq N_\F(V)$.
\end{lemma}

We prove Lemma~\ref{Generate} in a series of lemmas. For that let $R$ be a counterexample to Lemma~\ref{EssentialSubgroups}.

\begin{property}\label{L0}
 We have $Z\leq R$ and $Z$ is not $\Aut_\F(R)$-invariant.
\end{property}

\begin{proof}
As $R$ is essential, we have $R\in\F^c$ and thus $Z\leq C_S(R)\leq R$. Since $Q$ is large, we have moreover $N_\F(Z)\subseteq N_\F(Q)$. Therefore, the assertion follows from $R$ being a counterexample.
\end{proof}

\begin{property}\label{L01}
If $Y_M\leq R$ then $Y_M$ is not $\Aut_\F(R)$-invariant. In particular, $R\not\leq Y_M$.
\end{property}

\begin{proof}
 By Lemma \ref{L:M*C*} $N_\F(V)=N_\F(Y_M)$. So the first part holds since $R$ is a counterexample to Lemma~\ref{EssentialSubgroups}. If $R\leq Y_M$ then $R=Y_M$ as $R\in\F^c$ and $Y_M$ is abelian. This proves the assertion.
\end{proof}

\begin{property}\label{L1}
$V\cap R$ is not $\Aut_\F(R)$-invariant.
\end{property}

\begin{proof}
As $V\unlhd S$ and $V$ is abelian, we have $[R,N_V(R)]\leq V\cap R$ and $[V\cap R,N_V(R)]=1$. Assuming $V\cap R$ is $\Aut_\F(R)$-invariant, it follows from Lemma~\ref{CentralSeries}(b) that $V\leq R$. Then $V$ is $\Aut_\F(R)$-invariant and $R$ is not a counterexample.  
\end{proof}

\begin{property}\label{L2}
 Suppose $Z(R)\not\leq Y_M$ and $RY_M/Y_M$ is abelian. Then either $R^\prime=1$ or $R\cap Y_M=R\cap V=C_V(R)$ has order $4$ and $RY_M/Y_M$ is elementary abelian.
\end{property}

\begin{proof}
Assume $R^\prime\neq 1$. As $RY_M/Y_M$ is abelian, $R^\prime\leq Y_M\cap R\leq C_{Y_M}(Z(R))\cap R\leq V\cap R$, where the last inclusion uses Lemma~\ref{ActionOnYM}. Now $R^\prime=R\cap V$ contradicts \ref{L1}, hence, $1\neq R^\prime<R\cap V$. As $Z(R)\not\leq Y_M$, $R\cap V\leq C_V(Z(R))<V$. Hence, $|R\cap V|=4$ and $|R^\prime|=2$. As $|(V\cap R)/Z|=2$, we have $[V\cap R,R]\leq Z$. If $[V\cap R,R]\neq 1$ then, $Z=[V\cap R,R]\leq R^\prime$ and $Z=R^\prime$ as $|Z|=2=|R'|$. This contradicts \ref{L0}. Hence, $V\cap R\leq C_V(R)$, and order considerations give $V\cap R=C_V(R)$. Then $RY_M/Y_M$ is contained in the transvection group belonging to a hyperplane in $V$ and thus $RY_M/Y_M$ is elementary abelian.  
\end{proof}

\begin{property}\label{L5}
 $RY_M/Y_M$ is not cyclic of order $4$.
\end{property}

\begin{proof}
Suppose this is false. Then $RY_M/Y_M$ does not lie in the transvection group belonging to a hyperplane of $V$ and thus $C_V(R)=Z$. If $Z(R)\leq Y_M$, then it follows from Lemma~\ref{ActionOnYM} and $R\in\F^c$ that $Z(R)=C_{Y_M}(R)=C_V(R)=Z$ contradicting \ref{L0}. Hence, $Z(R)\not\leq Y_M$. Thus, $R$ is abelian by \ref{L2}. 

\smallskip

Now $R\cap Y_M=R\cap V\leq C_V(R)=Z$, where we use Lemma~\ref{ActionOnYM}. Hence, $R$ is isomorphic to $C_4$, $C_8$ or $C_2\times C_4$. Thus, $\Aut(R)$ is a $2$-group, contradicting the assumption that $R$ is essential.      
\end{proof}

\begin{property}\label{L6a}
 Let $X$ be the centralizer in $S$ of a hyperplane of $V$. Then there exist at most two elementary abelian subgroups of $X$ of order $2^4$.
\end{property}

\begin{proof}
Let $U$ be a hyperplane of $V$ such that $X=C_S(U)$. Note that $Y_M\leq X$ is elementary abelian of oder $2^4$ and $X/Y_M\cong C_2\times C_2$. Moreover, as $C_M(V)=Y_M$, we have $C_V(x)=U$ for all $x\in X\backslash Y_M$.

\smallskip

Assume now that there is a subgroup $W\neq Y_M$ of $X$, which is also elementary abelian of order $2^4$. As $W\not\leq Y_M$, we have $W\not\leq Y_M$ and thus, by Lemma~\ref{ActionOnYM}, $W\cap Y_M\leq C_{Y_M}(W)=C_V(W)=U$. As $|W|=2^4$, $|U|=4$ and $|W/W\cap Y_M|=|WY_M/Y_M|\leq |X/Y_M|=4$, it follows indeed that 
\[W\cap Y_M=C_{Y_M}(W)=U\mbox{ and }X=WY_M.\]
If $w\in W\backslash Y_M$, Lemma~\ref{ActionOnYM} gives us $C_{Y_M}(w)=C_V(w)=U$. For an arbitrary $w\in W$ and $y\in Y_M$, the product $wy$ is an involution if and only if $[w,y]=1$. For this to be the case, we must either have $w\in Y_M$ or $w\not\in Y_M$ and $y\in C_{Y_M}(w)=U\leq W$. Thus, every involution of $X=WY_M$ lies either in $Y_M$ or in $W$. This shows that $Y_M$ and $W$ are the only elementary abelian subgroups of $X$ of order $2^4$. 
\end{proof}

\begin{property}\label{L6}
 We have $C_V(R)=Z$. In particular, $|RY_M/Y_M|\geq 4$.
\end{property}

\begin{proof}
Note that every involution acting on $V$ acts quadratically on $V$ and thus centralizes a hyperplane of $V$. If $C_V(R)=Z$, we can thus conclude that $|RY_M/Y_M|\geq 4$. 

\smallskip

Assume now $Z<C_V(R)$. It follows from \ref{L01} and $C_M(V)=Y_M$ that $C_V(R)<V$. Hence, $|C_V(R)|=4$ and $RY_M/Y_M$ lies in the transvection group belonging to the hyperplane $C_V(R)$ in $V$, i.e. $R\leq X:=C_S(C_V(R))=C_M(C_V(R))$. 
In particular, by \cite[Lemma~1.1(b)]{MStr:2008}, $[Y_M,R]\leq [Y_M,X]=[V,X]=C_V(R)\leq V\cap R$ where the latter inclusion uses $R\in\F^c$.

\smallskip

Suppose first that $Z(R)\leq Y_M$. Using $R\in\F^c$ and Lemma~\ref{ActionOnYM}, we see that $Z(R)=C_{Y_M}(R)=C_V(R)\geq [R,Y_M]$. Moreover, as $Y_M$ is abelian, we have $[Z(R),Y_M]=1$. Hence, by Lemma~\ref{CentralSeries}(b), $Y_M\leq R$. By Lemma~\ref{L6a}, there are at most two elementary abelian subgroups of $X$ of order $2^4$. Thus, $R$ has at most two elementary abelian subgroups of order $2^4$, and one of them is $Y_M$. Hence, every element of $\Aut_\F(R)$ of odd order normalizes $Y_M$. Recall that $Y_M\unlhd S$. As $R\in\F^f$, we have $\Aut_S(R)\in\Syl_2(\Aut_\F(R))$ and thus $\Aut_\F(R)=\Aut_S(R)O^2(\Aut_\F(R))$ normalizes $Y_M$, contradicting \ref{L01}. This shows that $Z(R)\not\leq Y_M$.

\smallskip

As $Z(R)\not\leq Y_M$, it follows from \ref{L2} that $R^\prime=1$ or $R\cap Y_M=R\cap V=C_V(R)$. If $R$ is abelian then also $R\cap Y_M=C_{Y_M}(R)=C_V(R)=V\cap R$ by Lemma~\ref{ActionOnYM} and as $R\in\F^c$. So in any case, $R\cap Y_M=R\cap V=C_V(R)$ has order $4$ and $|R|\leq 2^4$. As $X/Y_M$ is elementary abelian, $\Phi(R)\leq Y_M\cap R=V\cap R$. By \ref{L1}, $\Phi(R)<Y_M\cap R$. We have seen above that $[Y_M,R]\leq V\cap R$, so in particular, $Y_M\leq N_S(R)$. Note that $Y_M$ centralizes $Y_M\cap R$, so that for $\ov{R} := R / \Phi(R)$ we get $(Y_M \cap R) /\Phi(R) \le C_{\ov{R}}(Y_M)$ and so  $|\ov{R} / C_{\ov{R}}(Y_M)| \le |R/(Y_M\cap R)|\leq |X/Y_M|=4=|Y_MR/R|$. Hence Lemma~\ref{FF} yields $|R|=2^4$, $\Phi(R)=1$, $\Aut_\F(R)=\Out_\F(R)$, $O^{2^\prime}(\Aut_\F(R))\cong \SL_2(4)$, $R$ is a natural $\SL_2(4)$-module for $O^{2^\prime}(\Aut_\F(R))$ and $N_S(R)=Y_MR=X$. 

\smallskip

It follows now from the structure of $\SL_2(4)$ that there exists $\alpha_0\in N_{\Aut_\F(R)}(\Aut_S(R))$ of order $3$ and such $\alpha_0$ acts irreducibly on $\Aut_S(R)\cong C_2\times C_2$. As $R\in\F^f$ and $F$ is saturated, $\alpha_0$ extends to an element $\alpha\in\Aut_\F(X)$ of order $3$ acting irreducibly on $X/R=Y_MR/R$. It follows now from \ref{L6a} that there are precisely two elementary abelian subgroups of $X$ of order $2^4$, and these are $R$ and $Y_M$. So $\alpha$ normalizes $Y_M$. As $\Aut_\F(Y_M)=\Aut_M(Y_M)$, it follows $[Y_M,\alpha]\leq [Y_M,M]=V$, so $\alpha$ leaves $VR/R < X/R$ invariant, contradicting the fact that $\alpha$ is irreducible on $Y_MR/R$.  
\end{proof}

\begin{property}\label{L4}
We have $R\not\leq C_S(V/Z)$.
\end{property}

\begin{proof}
Set $X:=C_S(V/Z)$. By the structure of the natural $\SL_3(2)$-module, $X/Y_M$ is elementary abelian of order $4$ and $C_V(X)=Z$. So by Lemma~\ref{ActionOnYM}, $C_{Y_M}(X)=C_V(X)=Z$. 

\smallskip

Assume $R\leq X$. Recall $|RY_M/Y_M|\geq 4$ by \ref{L6}. Hence, $X=Y_MR$. Assume $Z(R)\leq Y_M$. Using  Lemma~\ref{ActionOnYM} and $R\in\F^c$, we see then that $Z(R)=C_{Y_M}(R)=C_V(R)=C_V(X)=Z$ contradicting \ref{L0}.  This shows $Z(R)\not\leq Y_M$. As $C_V(R)=C_V(X)$ has order $2$, it follows in particular from \ref{L2} that $R$ is abelian. 

\smallskip

By Lemma~\ref{ActionOnYM}, we have $R\cap Y_M=R\cap V=C_V(R)=Z$ as $C_S(R)\leq R$. So $|R|\leq 2^3$. Note that $\Phi(R)\leq Y_M\cap R=Z$ and so, by \ref{L0}, $\Phi(R)=1$. As $R\leq X$, we have $[R,V]\leq Z\leq R$. Thus $V\leq N_S(R)$ and $V$ induces transvections on $R$. Thus by Lemma~\ref{FF}, $2^3\geq |R|\geq |N_S(R)/R|^2\geq |V/(V\cap R)|^2=|V/Z|^2=2^4$, a contradiction. This completes the proof.
\end{proof}

\begin{property}\label{L3}
 We have $S\neq RY_M$.
\end{property}

\begin{proof}
 Assume $S=RY_M$. Then $Z(R)\cap Y_M \le C_{Y_M}(R)\leq Z(S)=Z$ and that $R\cap Y_M\unlhd R$ and so $Z(R)\cap Y_M\neq 1$ as $R$ is a $p$-group. Since $|Z|=2$ it follows $Z(R)\cap Y_M=Z$ and thus $Z(R)\not\leq Y_M$ by \ref{L0}. In particular, by Lemma~\ref{ActionOnYM}, $V\cap R=Y_M\cap R\leq C_V(Z(R))$ has order at most $4$.

 \smallskip
 
 As $S/Y_M$ is dihedral of order 8, it follows that
 \[C_2\cong Z(R)Y_M/Y_M=Z(S/Y_M)=\Phi(S/Y_M)\geq \Phi(R) Y_M/Y_M.\] 
In particular, $\Phi(R)\leq Z(R)Y_M\cap R=Z(R)(Y_M\cap R)=Z(R)(V\cap R)$. As $Z(R)\cap Y_M=Z$, we see now that $|Z(R)|=4$ and $|C_V(Z(R))|=4$. Lemma~\ref{CentralSeries}(b) applied to the series $1 < Z(R) < R$ and to $C_V(Z(R))$ in place of $X$ yields $C_V(Z(R)) \le R$. We conclude that $C_V(Z(R)) = R \cap V$ has order $4$ and 
\[
|Z(R)(V\cap R)|=2\cdot |V\cap R|= 2^3.
\] 

\smallskip

If $Z(R)$ were cyclic of order $4$, then $Z \le Z(R)$ would be $\Aut_\F(R)$-invariant, in contradiction to \ref{L0}. Thus $Z(R) \cong C_2 \times C_2$.

\smallskip

If $\Phi(R)\not\leq Z(R)$, then we obtain $\Phi(R) Z(R)=Z(R)(V\cap R)$ and $1\neq [\Phi(R),R]=[V\cap R,R]\leq Z$. As $|Z|=2$ this would yield $Z=[\Phi(R),R]$, a contradiction to \ref{L0}. Hence, $\Phi(R)\leq Z(R)$. Since $|V\cap R/Z|=2$, we have $1\neq [V\cap R,R]\leq Z$ and thus $Z=[V\cap R,R]\leq R^\prime\leq \Phi(R)$. Hence, \ref{L0} yields $\Phi(R)=Z(R)\cong C_2\times C_2$.

\smallskip

As $R \cap V = R \cap Y_M$ and $8 = |RY_M / Y_M| = |R / (R \cap Y_M)|$, it follows $|R| = 2^5$. Thus,
\[\ov{R}:=R/\Phi(R)=R/Z(R)\mbox{ has order }2^3.\]

\smallskip

Since $|V / (V \cap R)| = 2$ we also get $[V,R] \le V \cap R$ and so $V \le N_S(R)$. Note that $|V/V\cap R|=2=|\ov{R\cap V}|\geq |\ov{[R,V]}|= |[\ov{R},V]|=|\ov{R}/C_{\ov{R}}(V)|$, where the last equality uses \cite[8.4.1]{Kurzweil/Stellmacher:2004} and the fact that $V$ acts as an involution on $\ov{R}$ as $|V/V\cap R|=2$. Hence, by Lemma~\ref{FF},
 \[
 H:=O^{2^\prime}(\Out_\F(R))\cong \SL_2(2) \cong S_3
 \]
 and $\ov{R}/C_{\ov{R}}(H)$ is a natural $\SL_2(2)$-module. Then
 \[
 \ov{R}=[\ov{R},H]\times C_{\ov{R}}(H).
 \]
 Let $R_0$ and $R_1$ be the preimages of $[\ov{R},H]$ and of $C_{\ov{R}}(H)$ respectively. Our goal is now to show that $R_0$ is elementary abelian and obtain a contradiction to Lemma \ref{L:M*C*}(e).
 
 \smallskip
 
 Observe that $\ov{R_0}$ is a natural $\SL_2(2)$-module for $H$. In particular, every element of $C:=C_{\Out_\F(R)}(H)$ acts on $\ov{R_0}$ as a non-zero scalar of $\mathbb{F}_2$ and thus as the identity. As $|\ov{R_1}|=2$, $C$ centralizes also $\ov{R_1}$ and thus $C$ centralizes $\ov{R}$. Hence, by Lemma~\ref{ActionRModPhiR} $C=1$. As $\Aut(S_3)=\Inn(S_3)$, it follows $\Out_\F(R)=H\cong S_3$. 
 
 \smallskip
 
 Note that $H=\Out_\F(R)$ acts also on $Z(R)$. As $R$ is fully $\F$-normalized, $\Aut_S(R)$ is a Sylow $2$-subgroup of $\Aut_\F(R)$. If $O^2(H)$ acts trivially on $Z(R)$, then $O^2(\Aut_\F(R))$ acts also trivially on $Z(R)$ and thus $\Aut_\F(R)=\Aut_S(R)O^2(\Aut_\F(R))$ centralizes $Z$, contradicting \ref{L0}. Hence, as $Z(R)\cong C_2\times C_2$, it follows that $Z(R)=[Z(R),H]$ is a natural $\SL_2(2)$-module for $H$.  
 
 \smallskip
 
 As $1\neq [\ov{R},\ov{V}]\leq [\ov{R},H]=\ov{R_0}$ and $V$ is elementary abelian, there is a coset in $\ov{R_0}$ which contains an involution outside of $Z(R)$. Thus, as $H$ acts transitively on $\ov{R_0}^\sharp$, every coset in $\ov{R_0}^\sharp$ contains an involution outside of $Z(R)$. As $Z(R)\cong C_2\times C_2$ is in the centre of $R_0$, it follows that every element of $R_0$ is an involution and $R_0$ is elementary abelian. 
 
 \smallskip
 
As in Subsection~\ref{SS:Amalgam}, set $S^*:=S\cap O^2(\L)$. By Lemma~\ref{P:pResidualLocFus}, we have $S^*=\hyp(\F)$. Since both $R_0/Z(R)$ and $Z(R)$ are natural $S_3$-modules for $\Out_\F(R)$, it follows that $R_0=[R_0,O^2(\Aut_\F(R))]\leq S^*$. This contradicts Lemma~\ref{L:M*C*}(e), which states that $S^*$ does not contain any elementary abelian subgroup of order $2^4$.
\end{proof}

\begin{proof}[Proof of Lemma~\ref{EssentialSubgroups}]
We show that we reach a contradiction if $R$ is a counterexample to Lemma~\ref{EssentialSubgroups}. By \ref{L6} and \ref{L3}, $RY_M/Y_M$ is a subgroup of $S/Y_M$ of order $4$ and $C_V(R)=Z$. Note that $S/Y_M$ is dihedral of order $8$ and so the only subgroups of $S/Y_M$ of order $4$ are a cyclic group of order $4$ and two fours groups. The two fours groups are the transvection group belonging to a hyperplane in $V$ and the transvection group belonging to the point $Z$ in $V$. As $C_V(R)=Z$, $R$ is not contained in the transvection group belonging to a hyperplane, and by \ref{L5}, $RY_M/Y_M$ is not cyclic of order $4$. Hence, $RY_M/Y_M$ is the transvection group belonging to the point $Z$, which contradicts \ref{L4}. 
\end{proof}

\subsection{The proofs of Theorem~\ref{unique local} and Theorem~\ref{mainMS}}

\begin{proof}[Proof of Theorem~\ref{unique local}]
As before we consider the amalgam $(M,\tilde{C},B,i,j)$ where $i\colon B\rightarrow M$ and $j\colon B\rightarrow \tilde{C}$ are the inclusion maps, and also, for $G\cong \Aut(\G_2(3))$, the amalgam $(M_G,\tilde{C}_G,B_G,i_G,j_G)$ of subgroups of $G$ with inclusion maps, where $M_G$ and $\tilde{C}_G$ are as introduced in Lemma~\ref{L:G23} and $B_G=M_G\cap \tilde{C}_G$. Let 
\[X=M*_B\tilde{C}\mbox{ and }X_G=M_G*_{B_G}\tilde{C}_G\]
be the corresponding free amalgamated products. We identify $M$ and $\tilde{C}$ with subgroups of $X$, and $M_G$ and $\tilde{C}_G$ with subgroups of $X_G$ in the natural way. 

\smallskip

By Proposition~\ref{P:AmalgamsIso}, the amalgams $(M,\tilde{C},B,i,j)$ and $(M_G,\tilde{C}_G,B_G,i_G,j_G)$ are isomorphic. Hence, also the groups $X$ and $X_G$ are isomorphic. We may choose $S\in \Syl_2(B)$ and $S_G\in \Syl_2(B_G)$ such that $S$ is mapped onto $S_G$ under an isomorphism $X\rightarrow X_G$, and we obtain then  $\F_S(X)\cong \F_{S_G}(X_G)$. 

\smallskip

Now by \cite[Theorem~3.1]{Clelland/Parker}, we have 
\[\F_S(X)=\<\F_S(M),\F_S(\tilde{C})\>\mbox{ and }\F_S(X_G)=\<\F_{S_G}(M_G),\F_{S_G}(\tilde{C}_G)\>.\]
By Lemma~\ref{L:M*C*}(a), we have $M=\tilde{M}=N_\L(V)$. Recall also that $\tilde{C}=N_\L(Q)$ by definition. As $Q,V\in\Delta$, it follows thus from Lemma~\ref{L:NormalizerLocFus} that $\F_S(M)=N_\F(V)$ and $\F_S(\tilde{C})=N_\F(Q)$. Hence, by Proposition~\ref{Generate}, $\F=\<N_\F(V),N_\F(Q)\>=\F_S(X)$. If we use Lemma~\ref{L:G23} and choose the notation as in that lemma, then it follows similarly that we can apply Lemma~\ref{L:M*C*}(a), Lemma~\ref{L:NormalizerLocFus} and Lemma~\ref{Generate} with $(\F_{S_G}(G),\L_G,\Delta_G,S_G,Q_G,V_G)$ in place of $(\F,\L,\Delta,S,Q,V)$ to obtain $\F_{S_G}(G)=\F_{S_G}(X_G)$. Hence, $\F=\F_S(X)\cong \F_{S_G}(X_G)=\F_S(G)$.  
\end{proof}

\begin{proof}[Proof of Theorem~\ref{mainMS}]
 Let $\F$ be a fusion system over a finite $2$-group $S$ fulfilling the hypothesis of Theorem~\ref{mainMS} with $Q$ and $M$ chosen as in the theorem. Let $(\L,\Delta,S)$ be the subcentric locality over $\F$. Set $\tilde{C}:=N_\L(Q)$, $R:=O_2(M)$ and 
 \[\hat{M}:=\{g\in N_\L(R)\colon c_g|_R\in\Aut_M(R)\}.\] 
 We will verify Hypothesis~\ref{MainHyp} now with $\hat{M}$ in place of $M$. 
 
 \smallskip
 
 By Lemma~\ref{L:LargeSatFusMain}(b), the subcentric locality $(\L,\Delta,S)$ is $Q$-replete and so, by Lemma~\ref{L:LargeLF}(b), $Q$ is large in $(\L,\Delta,S)$.  In particular, $Q\in\Delta$. Therefore, $N_\L(Q)$ is a group of characteristic $2$ and thus a model for $N_\F(Q)$ by Lemma~\ref{L:NormalizerLocFus}. Hence, it follows from \cite[Theorem~III.5.10]{Aschbacher/Kessar/Oliver:2011} that $O_2(N_\F(Q))=O_2(\tilde{C})$. Thus, as we assume that $O_2(N_\F(Q))=Q$, we have $O_2(\tilde{C})=Q$ and (i) of Hypothesis~\ref{MainHyp} holds.

\smallskip

Since $M$ is of characteristic $2$ and contains $S$, we have $C_S(R)\leq R$ and $R\unlhd S$ is centric in $\F$. In particular, $R\in \Delta$ and $N_\L(R)$ is a group. Observe that $\hat{M}$ is a subgroup of $N_\L(R)$ containing $S$ as a Sylow $2$-subgroup. In particular, $\hat{M}$ is a group of characteristic $2$ in which $R$ is normal. As $C_S(R)\leq R$, it follows thus from the Schur--Zassenhaus Theorem \cite[3.3.1]{Kurzweil/Stellmacher:2004} that $C_{\hat{M}}(R)=Z(R)\times O_{2^\prime}(C_{\hat{M}}(R))$. As $O_{2^\prime}(C_{\hat{M}}(R))\leq O_{2^\prime}(\hat{M})=1$, we can conclude that $C_{\hat{M}}(R)=Z(R)$. 

\smallskip

Note that $\Aut_{\hat{M}}(R)=\Aut_M(R)$ by definition of $\hat{M}$ and since $N_\L(R)$ realizes $N_\F(R)$. As $C_M(R)=Z(R)$, using assumption (I) in Theorem~\ref{mainMS}, we see that
\[\SL_3(2)\cong M/R\cong (M/Z(R))/(R/Z(R))\cong \Aut_M(R)/\Inn(R)\]
and $V$ is a natural $\SL_3(2)$-module for $\Aut_M(R)/\Inn(R)$. Since $C_{\hat{M}}(R)=Z(R)$, one ses similarly that $\hat{M}/R\cong \Aut_{\hat{M}}(R)/\Inn(R)=\Aut_M(R)/\Inn(R)\cong \SL_3(2)$ and $V$ is a natural $\SL_3(2)$-module for $\hat{M}/R$. In particular, $\hat{M}/R$ is simple and so $R=O_2(\hat{M})$.

\smallskip

By Lemma~\ref{L:pReduced}, $Y_M$ is the largest $2$-reduced $\Aut_M(R)$-submodule of $\Omega_1(Z(R))$, and $Y_{\hat{M}}$ is the largest $2$-reduced $\Aut_{\hat{M}}(R)$-submodule of $\Omega_1(Z(R))$. As $\Aut_M(R)=\Aut_{\hat{M}}(R)$, it follows that $Y_M=Y_{\hat{M}}$ and $V=[Y_M,M]=[Y_M,\Aut_M(R)]=[Y_M,\Aut_{\hat{M}}(R)]=[Y_{\hat{M}},\hat{M}]$. Recall also that assumption (II) in Theorem~\ref{mainMS} gives us $Y_M\not\leq Q$ and $V\leq Q$. Thus, (ii) and (iii) of Hypothesis~\ref{MainHyp} hold with $\hat{M}$ in place of $M$. 
\end{proof}

\bibliographystyle{amsplain-nomr}
\bibliography{rep-coh-local}

\providecommand{\bysame}{\leavevmode\hbox to3em{\hrulefill}\thinspace}
\providecommand{\MR}{\relax\ifhmode\unskip\space\fi MR }
\providecommand{\MRhref}[2]{%
  \href{http://www.ams.org/mathscinet-getitem?mr=#1}{#2}
}
\providecommand{\href}[2]{#2}
\begin{thebibliography}{10}

\bibitem{AschbacherChermak2010}
M.~Aschbacher and A.~Chermak, \emph{A group-theoretic approach to a family of
  2-local finite groups constructed by {L}evi and {O}liver}, Ann. of Math. (2)
  \textbf{171} (2010), no.~2, 881--978.

\bibitem{Aschbacher/Kessar/Oliver:2011}
M.~Aschbacher, R.~Kessar, and B.~Oliver, \emph{{Fusion systems in algebra and
  topology}}, London Math.\ Soc.\ Lecture Note Series, vol. 391, Cambridge
  University Press, 2011.

\bibitem{Aschbacher:2002}
Michael Aschbacher, \emph{Finite groups of ${G}_2(3)$-type}, Journal of Algebra
  \textbf{257} (2002), no.~2, 197--214.

\bibitem{Aschbacher:2011}
\bysame, \emph{{The generalized Fitting subsystem of a fusion system}},
  \textbf{209} (2011), no.~986, vi+110.

\bibitem{BMO}
C.~Broto, J.~M. M{\o}ller, and B.~Oliver, \emph{Equivalences between fusion
  systems of finite groups of {L}ie type}, J. Amer. Math. Soc. \textbf{25}
  (2012), no.~1, 1--20.

\bibitem{BMO2}
\bysame, \emph{{Automorphisms of fusion systems of finite simple groups of Lie
  type}}, Mem. Amer. Math. Soc. \textbf{262} (2019), no.~1267, 1--117.

\bibitem{BrotoMollerOliverRuiz2023}
Carles Broto, Jesper~M. M{\o}ller, Bob Oliver, and Albert Ruiz,
  \emph{Realizability and tameness of fusion systems}, Proc. Lond. Math. Soc.
  (3) \textbf{127} (2023), no.~6, 1816--1864.

\bibitem{Carter:1972a}
R.~W. Carter, \emph{{Simple groups of Lie type}}, Pure and Applied Mathematics,
  vol.~28, John Wiley \& Sons, London-New York-Sydney, 1972.

\bibitem{Chermak:2013}
A.~Chermak, \emph{Fusion systems and localities}, Acta Math. \textbf{211}
  (2013), no.~1, 47--139.

\bibitem{Chermak:2022}
\bysame, \emph{{Finite localities I}}, Forum of Mathematics, Sigma \textbf{10}
  (2022), e43.

\bibitem{Chermak/Henke}
A.~Chermak and E.~Henke, \emph{Fusion systems and localities -- a dictionary},
  Adv. Math. \textbf{410} (2022), Paper No. 108690, 92.

\bibitem{Clelland/Parker}
M.~Clelland and C.W. Parker, \emph{Two families of exotic fusion systems}, J.
  Algebra \textbf{323} (2010), 287--304.

\bibitem{Diaz/Glesser/Mazza/Park:2009}
A.~Diaz, A.~Glesser, N.~Mazza, and S.~Park, \emph{{Glauberman's and Thompson's
  theorems for fusion systems}}, Proc.\ Amer.\ Math.\ Soc. \textbf{137} (2009),
  no.~2, 495--503.

\bibitem{GlaubermanLynd2016}
George Glauberman and Justin Lynd, \emph{Control of fixed points and existence
  and uniqueness of centric linking systems}, Invent. Math. \textbf{206}
  (2016), no.~2, 441--484.

\bibitem{Henke:2010}
E.~Henke, \emph{{Recognizing $SL_2(q)$ in Fusion Systems}}, J. Group Theory
  \textbf{13} (2010), 679--702.

\bibitem{Henke:2013}
\bysame, \emph{Products in fusion systems}, J. Algebra \textbf{376} (2013),
  300--319.

\bibitem{Henke:2015}
\bysame, \emph{Subcentric linking systems}, Trans. Amer. Math. Soc.
  \textbf{371} (2019), no.~5, 3325--3373.

\bibitem{Henke:2021}
\bysame, \emph{Extensions of homomorphisms between localities}, Forum of
  Mathematics, Sigma \textbf{9} (2021), e63.

\bibitem{Henke:Regular}
\bysame, \emph{Commuting partial normal subgroups and regular localities}, Mem.
  Amer. Math. Soc. \textbf{311} (2025), no.~1575, v+112.

\bibitem{HenkeLynd2022}
E.~Henke and J.~Lynd, \emph{Fusion systems with {B}enson-{S}olomon components},
  Duke Math. J. \textbf{171} (2022), no.~3, 673--737.

\bibitem{HenkeLynd2026}
\bysame, \emph{Components and realizability of fusion systems},
  arXiv:2504.19826 [math.GR] (2025), 8 pp.

\bibitem{Kurzweil/Stellmacher:2004}
H.~Kurzweil and B.~Stellmacher, \emph{The theory of finite groups},
  Universitext, Springer-Verlag, New York, 2004.

\bibitem{LeviOliver2002}
R.~Levi and B.~Oliver, \emph{Construction of 2-local finite groups of a type
  studied by {S}olomon and {B}enson}, Geom. Topol. \textbf{6} (2002), 917--990.

\bibitem{Lynd2014}
Justin Lynd, \emph{The {T}hompson-{L}yons transfer lemma for fusion systems},
  Bull. Lond. Math. Soc. \textbf{46} (2014), no.~6, 1276--1282.

\bibitem{MS:2009}
U.~Meierfrankenfeld and B.~Stellmacher, \emph{{$F$}-stability in finite
  groups}, Trans. Amer. Math. Soc. \textbf{361} (2009), no.~5, 2509--2525.

\bibitem{MSS}
U.~Meierfrankenfeld, B.~Stellmacher, and G.~Stroth, \emph{The local structure
  theorem for finite groups with a large $p$-subgroup},  \textbf{242} (2016),
  no.~1147.

\bibitem{MStr:2008}
U.~Meierfrankenfeld and G.~Stroth, \emph{A characterization of
  {$\Aut(G_2(3))$}}, J. Group Theory \textbf{11} (2008), 479--494.

\bibitem{Oliver2004}
B.~Oliver, \emph{Equivalences of classifying spaces completed at odd primes},
  Math. Proc. Cambridge Philos. Soc. \textbf{137} (2004), no.~2, 321--347.

\bibitem{Oliver2006}
\bysame, \emph{Equivalences of classifying spaces completed at the prime two},
  Mem. Amer. Math. Soc. \textbf{180} (2006), no.~848, vi+102.

\bibitem{Oliver:2013}
\bysame, \emph{Existence and uniqueness of linking systems: {C}hermak's proof
  via obstruction theory}, Acta Math. \textbf{211} (2013), no.~1, 141--175.

\bibitem{Salati:PhD}
E.~Salati, \emph{{How local is the Local Structure Theorem for finite groups
  with a large $p$-subgroup?}}, {Ph.\ D.\ Dissertation}, TU Dresden, 2025.

\bibitem{Weibel:1994a}
C.~A. Weibel, \emph{{Introduction to homological algebra}}, Cambridge Studies
  in Advanced Mathematics, vol.~38, Cambridge University Press, 1994.

\bibitem{Weir:1955}
A.~J. Weir, \emph{Sylow p-subgroups of the general linear group over finite
  fields of characteristic p}, Proc.\ Amer.\ Math.\ Soc. \textbf{6} (1955),
  no.~3, 454--464.

\bibitem{Wilson:2009}
R.~A. Wilson, \emph{The finite simple groups}, Graduate Texts in Mathematics,
  Springer London, 2009.

\bibitem{Winter:1972}
D.~L. Winter, \emph{The automorphism group of an extraspecial p-group}, The
  Rocky Mountain Journal of Mathematics \textbf{2} (1972), no.~2, 159--168.

\end{thebibliography}

\end{document}